\documentclass[12pt, a4paper,leqno]{article}

\usepackage[english]{babel}
\usepackage[utf8]{inputenc}
\usepackage[T1]{fontenc}

\usepackage{graphicx}

\usepackage{paralist}

\usepackage{mathtools}
\usepackage{amssymb}
\usepackage{amsthm}
\usepackage{accents}
\usepackage{mathrsfs}
\usepackage{dsfont}

\newtheorem{thm}{Theorem}[section]

\newtheorem{lem}[thm]{Lemma}
\newtheorem{prop}[thm]{Proposition}

\newtheorem{rem}[thm]{Remark}
\newtheorem{cor}[thm]{Corollary}
\theoremstyle{definition}
\newtheorem{defi}[thm]{Definition}

\newcommand{\norm}[2]{||#1||_{#2}}
\newcommand{\skp}[2]{\langle #1 \rangle_{#2}}
\newcommand{\Fskp}[2]{\langle\langle #1 \rangle\rangle_{#2}}
\newcommand{\RR}{\mathbb{R}}
\newcommand{\NN}{\mathbb{N}}
\newcommand{\CC}{\mathbb{C}}
\newcommand{\CH}{\mathbb{H}}
\newcommand{\EH}{\mathcal{H}}
\newcommand{\BB}{\mathcal{B}}
\newcommand{\DS}{\mathscr{S}}

\newcommand{\EE}{\mathscr{E}}

\newcommand{\GG}{\mathscr{G}}
\newcommand{\OO}{\mathcal{O}}
\newcommand{\NH}{\mathcal{N}}
\newcommand{\CS}{\mathcal{C}}

\newcommand{\RE}{\textnormal{Re}\,}

\newcommand{\GL}{\textnormal{GL}}

\newcommand{\ID}{\mathds{1}}

\newcommand{\id}{\textnormal{id}}

\newcommand{\SP}{\textnormal{span}\,}
\newcommand{\diag}{\textnormal{diag}}

\newcommand{\tA}{\tilde{A}}

\newcommand{\fG}{\mathfrak{G}}
\newcommand{\fg}{\mathfrak{g}}

\usepackage{verbatim}

\title{Spectral stability of shock profiles for hyperbolically regularized systems of conservation laws}
\author{Johannes Bärlin}

\begin{document}
\maketitle
\begin{abstract}
We report a proof that under natural assumptions shock profiles viewed as heteroclinic travelling wave solutions to a hyperbolically regularized system of conservation laws are spectrally stable if the shock amplitude is sufficiently small. This means that an associated Evans function $\mathcal{E}:\Lambda\rightarrow\CC$ with $\Lambda\subset\CC$ an open superset of the closed right half plane $\CH^+\equiv\{\kappa\in\CC:\RE \kappa \geq 0\}$, has only one zero, namely a simple zero at $0$. The result is analogous to the one obtained in \cite{FS02} and \cite{PZ04} for parabolically regularized systems of conservation laws, and also distinctly extends findings on hyperbolic relaxation systems in \cite{MZ09,PZ04,U09}.
\end{abstract}
This preprint has not undergone peer review (when applicable) or any post-submission
improvements or corrections. The Version of Record of this article is published in \textit{Archive for Rational Mechanics and Analysis} as \cite{B24}, and
is available online at https://doi.org/10.1007/s00205-024-02066-9.
\section{Introduction}
Let $n \in \NN, \mathcal{V} \subset \RR^n$ open, $v^\ast \in \mathcal{V}$, and let $k \in \lbrace 1, \ldots, n \rbrace$. We consider a hyperbolically regularized conservation law of the form
\begin{align}
\label{HsCL}
g(v)_t + f(v)_x = B \square v := B(v_{xx} - v_{tt}) \;\; (v \in \RR^n)
\end{align}
for $(t,x) \in \RR^2$ with  smooth $f,g: \mathcal{V} \rightarrow \RR^n$ and $B \in \RR^{n\times n}$ satisfying
\begin{itemize}
\item[A1] $\forall v \in \mathcal{V}: B$ and the Jacobian matrices $Df(v)$ and $Dg(v)$ of $f$ and $g$ at $v$ are symmetric with  $B$ and $Dg(v^\ast)$ being positive definite;
\item[A2] $Df(v^\ast)$ has real and simple eigenvalues w.r.t.~$Dg(v^\ast)$;
\item[A3] The $k$-th eigenvalue of $Df$ w.r.t.~$Dg$ is genuinely nonlinear at $v^\ast$;
\item[A4] Sub-characteristicity: $Dg(v^\ast) \pm Df(v^\ast)$ is positive definite.
\end{itemize}
We think of system \eqref{HsCL} on the one hand as
\begin{itemize}
\item[(a)] a prototype for hyperbolic artificial viscosity,
\end{itemize}
i.e.~an alternative to the parabolic counterpart (c.f.~\cite{G59}), and on the other hand as
\begin{itemize}
\item[(b)] a simple model for the equations of relativistic fluid dynamics with hyperbolic dissipation as introduced in \cite{FT14}.
\end{itemize}
Let us briefly discuss assumptions A1-A4:

Assumption A1 is classical: Symmetric hyperbolic conservation laws are ubiquitous in mathematical physics \cite{G61,FL71}. In fact, choosing appropriate variables one may often assume that $f$ and $g$ possess potentials \cite{G61,B74,RS81}. Thus we think of $f$ and $g$ as gradient fields which justifies assumption A1.

Assumption A2 means that there are $n$ pairwise distinct real numbers $\mu_i \in \RR$ and $n$ vectors $r_i \in \RR^n \backslash \lbrace 0 \rbrace$ such that
\begin{align*}
Df(v^\ast) r_i = \mu_i Dg(v^\ast) r_i.
\end{align*}
A2 is equivalent to strict hyperbolicity of $Dg(v^\ast)^{-1/2} Df(v^\ast) Dg(v^\ast)^{-1/2}$ which is a common assumption on the flux of hyperbolic conservation laws.

Assumption A3 states that the $k$-th eigenvalue of $Dg^{-1/2} Df Dg^{-1/2}$ is genuinely nonlinear at $v^\ast$. This genuine non-linearity in the sense of Lax \cite{L57} is used, on the one hand, to show the existence of shock profiles \cite{G86,MP85} determining in the spirit of an entropy condition the direction in which the end-states of a shock can be smoothly connected. On the other hand it also enters in various contexts in the proofs on stability of said shock profiles enabling, vaguely speaking, control of contributions that emerge along the $k$-th eigenvector field (see e.g.~\cite{FS02,G86,PZ04,SX93}). Assumption A3 will play a similar role in the following analysis. (For results beyond genuine non-linearity the reader may consult, e.g., \cite{F00,FSW14,JGK93,L03}.)

By assumption A4 we find
\begin{align*}
\sigma (Dg(v^\ast)^{-1/2} Df(v^\ast) Dg(v^\ast)^{-1/2}) \subset (-1,1),
\end{align*}
i.e.~the characteristic speeds of the first-order operator on the left hand side in \eqref{HsCL} lie between those of the second-order operator on the right, the latter being (essentially) the d'Alembertian with speed of light equal to $1$ (see also \cite{W99}). The sub-characteristicity condition A4 (c.f.~\cite{S20}) is inspired by the notion of dissipativity in hyperbolic-parabolic systems introduced by Shizuta and Kawashima in \cite{SK85}, and is thereby closely related to considerations concerning hyperbolic relaxation systems (see e.g.~\cite{KY04,L87,Y04,Z99}). Hence one expects the dispersion curves of shock profiles of system \eqref{HsCL} to look similar to the ones in the case of hyperbolic relaxation systems \cite{MZ02} (see section \ref{Sec2}).

Finally, without loss of generality we may assume
\begin{align*}
f(v^\ast) = g(v^\ast) = 0 \;\textnormal{and} \; v^\ast = 0.
\end{align*}
Fix some notation by letting $M^t$ denote the transposed matrix for $M \in \CC^{d \times m}, d,m \in \NN$. The above discussion of assumptions A1-A4 readily leads to:
\begin{lem}
\label{assupmlem1}
If A1-A4 hold then there exist $\delta > 0$ and functions $\mu_i \in C^\infty(B_\delta(0);\RR), r_i \in C^\infty(B_\delta(0);\RR^n\backslash \lbrace 0 \rbrace)$ for $i=1, \ldots, n$ such that for all $v \in B_\delta(0) \subset \RR^n$ the matrices $Dg(v)$ and $Dg(v) \pm Df(v)$ are positive definite, and for all $v,w \in B_\delta(0)$ and $i,j$
\begin{align*}
Df(v) r_i(v) & = \mu_i(v) Dg(v) r_i(v),\\
r_i(v)^t Dg(v) r_j(v) & = \delta_{ij},
\end{align*}
where $\delta_{ij}$ denotes the Kronecker delta, and
\begin{align}
\label{GNL}
\skp{D\mu_k(0),r_k(0)}{} & < 0.
\end{align}
Furthermore, for all $v,w \in B_\delta(0)$ and $i,j$ we have
\begin{align}
\label{subchar}
|\mu_i(v)| < 1
\end{align}
and
\begin{align}
\label{eigorder}
\mu_i(v) < \mu_j(w) \;\; \textnormal{if} \;\; i < j.
\end{align}
\end{lem}
Multiplying \eqref{HsCL} from the left by $B^{-1/2}$ and changing variables via $\tilde{v} = B^{1/2} v$ we may assume with $\ID_n$ being the unit matrix in $\CC^{n}$ that
\begin{align*}
B = \ID_n
\end{align*}
in \eqref{HsCL} while keeping the assumptions A1-A4 (and the conclusions of Lemma \ref{assupmlem1}).
\begin{defi}
A triple $(v^-,v^+,s) \in \mathcal{V} \times \mathcal{V} \times \RR$ is called a shock of the conservation law $g(v)_t + f(v)_x = 0$ if the Rankine-Hugoniot condition
\begin{align*}
f(v^+)-f(v^-) = s(g(v^+)-g(v^-))
\end{align*}
is satisfied. It is called a (Lax-)$p$-shock for $p \in \lbrace 1, \ldots, d\rbrace$ if
\begin{align*}
\mu_p(v^+) < s < \mu_p(v^-) \;\; \textnormal{while} \;\; (\mu_j(v^-)-s)(\mu_j(v^+)-s) > 0 \;\; \textnormal{for} \;\; j \neq p.
\end{align*}
\end{defi}
We will move into the rest frame of a given $p$-shock $(v^-,v^+,s)$ with $s \in (-1,1)$ by the Lorentz-boost
\begin{align*}
\begin{pmatrix} t\\x \end{pmatrix} \mapsto \begin{pmatrix} \gamma(s)(t- s x)\\ \gamma(s)(x-st)\end{pmatrix}
\end{align*}
of speed $s$ where for $s \in (-1,1)$ one has
\begin{align*}
\gamma(s) := (1-s^2)^{-1/2}.
\end{align*}
 Under such a transformation the hyperbolically regularized conservation law \eqref{HsCL} keeps its form.
\begin{lem}
\label{lemLorentz}
Let $s \in (-1,1)$. Then a Lorentz-boost of speed $s$ transforms \eqref{HsCL} into
\begin{align*}
g_s(v)_t + f_s(v)_x = B\square v
\end{align*}
where
\begin{align*}
g_s(v) &:= \gamma(s)(g(v)-sf(v))\\
f_s(v) &:= \gamma(s)(f(v)-sg(v)).
\end{align*}
The assumptions A1-A4 hold with $f$ and $g$ replaced by $f_s$ and $g_s$. For $\delta > 0$ like in Lemma \ref{assupmlem1} the eigenvalues of $Df_s$ w.r.t.~$Dg_s$ on $B_\delta(0)$ are given by
\begin{align}
\label{eigenLorentz}
\mu_{j;s} = \frac{\mu_j-s}{1-\mu_js} \;\; (j \in \lbrace 1, \ldots, d \rbrace)
\end{align}
with eigenvectors $r_j$ and satisfying \eqref{GNL}, \eqref{subchar}, \eqref{eigorder}.
\end{lem}
\begin{proof}
The claimed transformation of equation \eqref{HsCL} with the given $f_s$ and $g_s$ is a direct calculation. Assumption A1 for $f_s$ and $g_s$ follows immediately, too. Using formula \eqref{eigenLorentz} and the fact that suitable eigenvectors are still given by $r_j$, yields A2 and A3 for $f_s$ and $g_s$ as well as \eqref{eigorder} and \eqref{GNL}. A4 for $f_s$ and $g_s$, implying \eqref{subchar}, is a direct calculation.
\end{proof}
Since $|\mu_k(0)| <1$ we can perform a Lorentz-boost of speed $s=\mu_k(0)$, and in view of Lemma \ref{lemLorentz} we assume from now on without loss of generality that
\begin{align*}
\mu_k(0) = 0
\end{align*}
by making the replacements $f \rightarrow f_{\mu_k(0)}$ and $g \rightarrow g_{\mu_k(0)}$. To simplify our set-up further we need a small lemma.
\begin{lem}
\label{splitlem1}
Let $F \in \RR^{n \times n}$ be symmetric. If there are $c>0$ and a splitting 
\begin{align*}
V^- \oplus \ker F \oplus V^+ = \RR^n
\end{align*}
such that for all $\eta^- \in V^-$ and $\eta^+ \in V^+$
\begin{align*}
\skp{\eta^-, F \eta^-}{} & \leq - c |\eta^-|^2,\\
\skp{\eta^+, F \eta^+}{} & \geq c |\eta^+|^2,
\end{align*}
then, counting with multiplicities, $F$ has exactly $\dim V^-$ negative and $\dim V^+$ positive eigenvalues.
\end{lem}
\begin{proof}
Let $W^-$ resp.~$W^+$ be the maximal $F$-invariant subspace of $\RR^n$ such that $F|_{W^-}$ resp.~$F|_{W^+}$ is negative resp.~positive definite. Obviously,
\begin{align*}
W^- \oplus \ker F \oplus W^+ = \RR^n.
\end{align*}
Assume $\dim W^- > \dim V^-$. Then $W^- \cap (\ker F \oplus V^+)$ is non-trivial, so we can pick $w = v^0 + v^+ \in W^- \cap (\ker F \oplus V^+) \backslash \lbrace 0 \rbrace$ with $v^0 \in \ker F$ and $v^+ \in V^+$. Since $w \not \in \ker F$ we find $v^+ \neq 0$ and
\begin{align*}
\skp{w, F w}{} = \skp{v^+,F v^+}{} \geq c |v^+|^2 > 0,
\end{align*}
a contradiction to the negative definiteness of $F|_{W^-}$. Hence, $\dim W^- \leq \dim V^-$. In the same way one shows $\dim W^+ \leq \dim V^+$. Since $\dim W^- + \dim W^+ = \dim V^- + \dim V^+$ we find $\dim W^- = \dim V^-$ and $\dim W^+ = \dim V^+$.
\end{proof}
\begin{lem}
\label{assumplem2}
Let $\delta, \mu_i$ and $r_i$ be like in Lemma \ref{assupmlem1}. Then there is an orthogonal change of variables leaving \eqref{HsCL} invariant such that
\begin{align*}
Df(0) = \diag(\lambda_1, \ldots, \lambda_d)
\end{align*}
with
\begin{align*}
\lambda_k = 0 \;\; \textnormal{and} \;\; \lambda_j < 0 < \lambda_{j^\prime} \; \; \textnormal{if} \; \; j < k < j^\prime,
\end{align*}
and
\begin{align*}
e_k = \frac{r_k(0)}{|r_k(0)|}
\end{align*}
where $e_k \in \CC^n$ is the unit vector along the $k$-th axis. With respect to this change of basis it holds
\begin{align}
\label{GNLcalc}
\skp{D\mu_k(0),r_k(0)}{} = |r_k(0)|^3 (\partial_k^2 f(0))_k.
\end{align}
\end{lem}
\begin{proof}
Set
\begin{align*}
V^- := \SP \lbrace r_j(0) \rbrace_{1\leq j<k} \;\; \textnormal{and} \; \; V^+ := \SP \lbrace r_j(0) \rbrace_{k<j\leq d}.
\end{align*}
The matrix $F:= Df(0)$ is symmetric and $\ker F = \SP \lbrace r_k(0) \rbrace$. Then $V^- \oplus \ker F \oplus V^+ = \RR^n$. On $Dg(0)^{1/2} V^-$ resp.~$Dg(0)^{1/2} V^+$ the matrix $Dg(0)^{-1/2} F Dg(0)^{-1/2}$ is negative resp.~positive definite. Because $Dg(0)$ is positive definite, too, this implies $F$ being negative resp.~positive definite on $V^-$ resp.~$V^+$. An application of Lemma \ref{splitlem1} shows that there is an orthonormal set of eigenvectors of $F$ with the $k$-th eigenvector equal to $r_k/|r_k|$ inducing a change of basis that diagonalizes $F$ in the claimed way.\\
To show \eqref{GNLcalc} use the above change of basis and apply the differential operator $\skp{r_k(0),D_v}{}=|r_k| \partial_k$ to the identity \cite{S94}
\begin{align*}
r_k(v)^t (Df(v) - \mu_k(v) Dg(v))r_k(v) = 0 \;\; \textnormal{for} \; \; v \in B_\delta(0).
\end{align*}
\end{proof}
From now on fix $\delta, \mu_i, r_i$ like in Lemma \ref{assupmlem1} and change basis like in Lemma \ref{assumplem2}.
For small enough $\varepsilon_0 \in(0,1]$ consider a smooth family of triples 
\begin{align*}
\lbrace(v^-_\varepsilon, v^+_\varepsilon, s(\varepsilon))\rbrace_{\varepsilon \in [0,\varepsilon_0]}
\end{align*}
such that:
\begin{itemize}
\item[(i)] For $\varepsilon > 0$ the triple $(v^-_\varepsilon, v^+_\varepsilon, s(\varepsilon))$ is a $k$-shock of $g(v)_t + f(v)_x = 0$;
\item[(ii)] $s(0) = 0$ and $s^\prime(0) = 0$;
\item[(iii)] $v^\pm_\varepsilon = \pm\varepsilon(e_k+\OO(\varepsilon))$;
\item[(iv)] $v^-_\varepsilon$ and $v^+_\varepsilon$ are the only solutions of
\begin{align*}
f(v) - s(\varepsilon) g(v) = f(v^-_\varepsilon) -s(\varepsilon) g(v^-_\varepsilon)
\end{align*}
for $v \in B_{3 \varepsilon}(0)$.
\end{itemize}
\begin{rem}
One can construct such a family of $k$-shocks by the usual (local) structure of the Rankine-Hugoniot locus
\begin{align*}
\lbrace v \in \mathcal{V}|\; \exists s \in \RR: f(v)-sg(v) = f(v^-_\varepsilon) -s g(v^-_\varepsilon) \rbrace
\end{align*}
at $v^-_\varepsilon$ (see \cite{L73}, and for a proof e.g.~\cite{E10}).
\end{rem}
For simplicity we only consider shocks with end-states symmetrically distributed over the $k$-th direction, i.e.
\begin{itemize}
\item[(iii)$^\prime$] $(v^\pm_\varepsilon)_k = \pm \varepsilon.$
\end{itemize}
Note $s(\varepsilon) \in (-1,1)$ since $(v^-_\varepsilon, v^+_\varepsilon, s(\varepsilon))$ is a $k$-shock.
\begin{defi}
We call $\phi \in C^\infty(\RR)$ a shock profile for a shock $(v^-,v^+,s) \in \mathcal{V} \times \mathcal{V} \in (-1,1)$ of $g(v)_t + f(v)_x = 0$ if $\phi$ is a solution of the boundary value problem
\begin{align}
\label{profileequation}
\phi^\prime = f_s(\phi) - f_s(v^-), \; \; \phi(\pm \infty) = v^\pm.
\end{align}
\end{defi}
For $\varepsilon \in [0,\varepsilon_0]$ set
\begin{align*}
c_\varepsilon := f_{s(\varepsilon)}(v^-_\varepsilon) \;\; \textnormal{and} \;\; \gamma_\varepsilon = \gamma(s(\varepsilon)).
\end{align*}
Viewing the construction of shock profiles by Majda and Pego \cite{MP85} from a geometric perspective one can obtain a smooth family of shock profiles $\lbrace \phi_\varepsilon \rbrace_\varepsilon$ for the family of shocks $\lbrace (v^-_\varepsilon,v^+_\varepsilon,s(\varepsilon) \rbrace_\varepsilon$ by a slow-fast decomposition like in \cite{FS02}. To make the slow-fast structure apparent set $\varepsilon u = \phi$ in \eqref{profileequation} for the shock $(v^-_\varepsilon, v^+_\varepsilon, s(\varepsilon))$. Then $u$ must solve
\begin{align}
\label{slowprofile1}
u^\prime = \tilde{f}_\varepsilon(u), \;\; u(\pm \varepsilon) = u^\pm_\varepsilon
\end{align}
where
\begin{align*}
\tilde{f}_\varepsilon(u) := \varepsilon^{-1} f_{s(\varepsilon)}(\varepsilon u) - \varepsilon^{-1} c_\varepsilon \;\; \textnormal{for} \;\; \varepsilon \in (0,\varepsilon_0], u \in \overline{B_3(0)}
\end{align*}
and
\begin{align*}
u_\varepsilon^\pm := \varepsilon^{-1} v^\pm_\varepsilon.
\end{align*}
A Taylor expansion shows
\begin{align*}
\varepsilon^{-1} f_{s(\varepsilon)}(\varepsilon u) = & \diag(\lambda_1, \ldots, \lambda_d) u + \frac{\varepsilon}{2}D^2f(0)\lbrack u,u \rbrack\\
& + \OO(\varepsilon^2)
\end{align*}
and
\begin{align*}
\varepsilon^{-1} c_\varepsilon = & \OO(\varepsilon) \;\; \textnormal{with} \;\; (\varepsilon^{-1} c_\varepsilon)_k = \varepsilon a + \OO(\varepsilon^2)
\end{align*}
for
\begin{align*}
a := \frac{1}{2}(\partial_k^2f(0))_k < 0.
\end{align*}
Hence, $\tilde{f}_\varepsilon(u)$ and $u_\varepsilon^\pm$ extend smoothly onto $\varepsilon = 0$ and \eqref{slowprofile1} is in Fenichel normal form \cite{F79}. The fast variables are $u_j$ for $j \neq k$, the slow variable is $u_k$. The set $\lbrace u_j =0 \rbrace_{j \neq k}$ is a normally hyperbolic critical manifold for \eqref{slowprofile1} and Fenichel theory shows (see also \cite{J95,S91}):
\begin{lem}
\label{slowlem}
For $\varepsilon_0$ small enough there exists a smooth function $u = u_\varepsilon(\tau)$ for $\varepsilon \in [0,\varepsilon_0], \tau \in J:=[-1,1],$ with $u_\varepsilon(\tau)_j = \OO(\varepsilon)$ if $j \neq k$ and $(u_\varepsilon(\tau))_k = \tau$ defining a slow manifold $u_\varepsilon(J)$ of \eqref{slowprofile1} with slow equations
\begin{align}
\label{slowprofileequ}
\tau^\prime = \varepsilon (1-\tau^2)h_\varepsilon(\tau)
\end{align}
for a smooth function $h=h_\varepsilon(\tau) = (- a + \OO(\varepsilon))$.
\end{lem}
We obtain a smooth family of shock profiles $\lbrace \phi_\varepsilon \rbrace_\varepsilon$ connecting $v^-_\varepsilon$ to $v^+_\varepsilon$ by letting $\phi_\varepsilon$ be the solution of
\begin{align*}
\phi^\prime = f_{s(\varepsilon)}(\phi) - c_\varepsilon, \; \; \phi(0) = \varepsilon u_\varepsilon(0).
\end{align*}
They correspond to solutions $\chi_\varepsilon$ of the slow equations
\begin{align}
\label{chi1}
\chi_\varepsilon^\prime = \varepsilon (1- \chi_\varepsilon^2)h_\varepsilon(\chi_\varepsilon), \;\; \chi_\varepsilon(0) = 0,
\end{align}
via
\begin{align*}
\phi_\varepsilon = \varepsilon u_\varepsilon(\chi_\varepsilon).
\end{align*}
The profiles $\phi_\varepsilon$ define travelling wave solutions $v_\varepsilon$ of \eqref{HsCL} with
\begin{align*}
v_\varepsilon(t,x) = \phi_\varepsilon(\gamma_\varepsilon(x-s(\varepsilon)t)).
\end{align*}
We apply a Lorentz-boost of speed $s(\varepsilon)$ to move into the rest frame of $v_\varepsilon$. Then \eqref{HsCL} reads
\begin{align}
\label{restframeHsCL}
g_{s(\varepsilon)}(v)_t + f_{s(\varepsilon)}(v)_x = \square v.
\end{align}
Linearizing \eqref{restframeHsCL} at $v_\varepsilon(t,x) = \phi_\varepsilon(x)$ yields
\begin{align}
\label{linearHsCL}
Dg_{s(\varepsilon)}(\phi_\varepsilon) v_t + (Df_{s(\varepsilon)}(\phi_\varepsilon) v)_x = \square v.
\end{align}
The Laplace modes $e^{\kappa t} p(x), \kappa \in \CC,$ of \eqref{linearHsCL} are defined by solutions $p$ of
\begin{align*}
(\kappa^2 \ID_d + \kappa Dg_{s(\varepsilon)}(\phi_\varepsilon))p + (Df_{s(\varepsilon)}(\phi_\varepsilon) p)_x = p_{xx}.
\end{align*}
Setting $q = p_x - Df_{s(\varepsilon)}(\phi_\varepsilon) p$ the above Laplace-mode differential equation has a first-order formulation
\begin{align}
\label{LapHsCL}
\xi^\prime = A_{s(\varepsilon),\kappa}(\phi_\varepsilon) \xi
\end{align}
for $\xi = (p,q) \in \CC^{2n}$ and
\begin{align*}
A_{s,\kappa}(v) = \begin{pmatrix}
Df_{s}(v) & \ID_n\\ \kappa^2\ID_n + \kappa Dg_{s}(v) & 0 \end{pmatrix} \in \CC^{2n \times 2n}.
\end{align*}
The study of the non-autonomous, parameter depending ordinary differential equation \eqref{LapHsCL} is the main objective of this paper.

We will follow the work of Freist{\"u}hler and Szmolyan on hyper\-bolic-para\-bolic conservation laws in \cite{FS02} and \cite{FS10}: Instead of treating the Laplace-mode equation \eqref{LapHsCL} as a non-autonomous system one couples \eqref{LapHsCL} to the slow profile equation \eqref{slowprofileequ} and obtains an autonomous system in $\tau$ and $\xi$ which depends on the parameters $\varepsilon$ and $\kappa$. Then invariant manifold theory is applied: Inspired by the construction of the shock  profiles $\phi_\varepsilon$ one uses normally hyperbolic theory \cite{F71,HPS77}, most times in the form of geo\-metric singular perturbation theory \cite{F79,J95,S91}. Making use of the linear structure of \eqref{LapHsCL} the evolution of vector spaces is tracked by the associated Grassmannian flows (c.f.~\cite{AGJ90}, see Section \ref{Sec3} for some basics concerning flows on Grassmann manifolds). Since the resulting state space is invariant and compact normally hyperbolic theory yields unique invariant manifolds perturbing smoothly from the initial manifolds. By considering three spectral regimes (see Section \ref{Sec4}) in combination with various rescalings a detailed construction of Evans bundles for the profiles $\phi_\varepsilon$ is achieved.

The use of spectral theory to answer questions on the stability of travelling waves has become popular in the last five decades starting with the fundamental work of Evans on nerve axon equations \cite{E74}. The \textit{Evans function method} was further developed by Jones \cite{J84} and Alexander et al.~\cite{AGJ90}, and since then Evans function theory has been applied successfully to reaction-diffusion equations, e.g.~ \cite{AGJ90,GJ91,GJ912,J84}, viscous conservations laws, e.g.~\cite{FS02,GZ98,JGK93,MZ04,PZ04,ZH98}, and hyperbolic relaxation systems, e.g.~\cite{MZ02,MZ05,PZ04}, yielding stability results in many different contexts. When studying the stability of travelling wave solutions associated to shocks of, e.g.,  viscous conservation laws one is met with the problem of the absence of a spectral gap in the sense that the essential spectrum touches the imaginary axis. At first glance this limits the scope of results of an ansatz by Evans function theory since an analytic extension of the Evans function into the essential spectrum is necessary \cite{GJ91,PW92}. But after the discovery that Evans bundles of shock profiles have suitable analytic extensions \cite{GZ98,KS98} it was Zumbrun and collaborators that showed the profound importance of spectral stability: Under natural assumptions spectral stability implies the non-linear stability of shock profiles for hyperbolic-parabolic conservation laws \cite{MZ04,ZH98} and hyperbolic relaxation systems \cite{MZ02,MZ05}. Whether such an implication also holds for hyperbolically regularized systems of conservation laws such as \eqref{HsCL} is an open question. However the reader may find a discussion of the associated resolvent kernel in section \ref{secresker}. We view our results as the first steps towards other stability results such as linearized and orbital stability.

Parallel to developments within the Evans function framework non-stan\-dard energy methods are used in the stability analysis of travelling waves. Along the lines of Sattinger \cite{S76}, Matsumura and Nishihara \cite{MN85} and Goodman \cite{G86} this ansatz was employed to great success, e.g.~\cite{F00,KMN86,L97,SX93,U09}, often solving directly the non-linear stability problem. The proofs are tailor-made for the specific problem at hand yielding strong results in the particular situation but are often difficult to generalize.

Despite the similarities to hyperbolic relaxation systems our result is not contained in the results on (spectral) stability of shock profiles for such systems \cite{H03,L03,MZ02,MZ05,MZ09,PZ04,U09}. This is due to the fact that in general $g$ need not be the identity, and hence it seems not possible to transform \eqref{HsCL} into a relaxation equation of Jin-Xin-type \cite{JX95} (c.f.~\cite{LMRS16,U09}). Furthermore, in the setting of special relativity as in \cite{FT14} we need to allow for $g \neq \id$ since by Lemma \ref{lemLorentz} the assumption $g = \id$ is in general not invariant under a Lorentz boost. Finally, we want to mention the work of Lattanzio et al.~\cite{LMRS16} where stability of travelling waves for the Allan-Cahn model with relaxation is discussed. The underlying model being a scalar semi-linear damped wave equation they also arrive at a quadratic non-linear eigenvalue problem. The dispersion curves of the latter equation lie entirely in the left half of the complex plane and do not touch the imaginary axis which is a major qualitative difference to the models considered in this paper.
\section{Domain of consistent splitting and main theorem}
\label{Sec2}
We begin this section with a lemma on dispersion curves for certain families of holomorphic, matrix-valued functions.
\begin{defi}
Let $\mathcal{A} = \lbrace A(\kappa) \rbrace_{\kappa \in \CC} \subset \CC^{2n \times 2n}$ be a smooth family of matrices and denote by $U(\kappa)$ resp.~$S(\kappa)$ the unstable resp.~stable subspace of $A(\kappa)$. Then a domain of constant splitting $\tilde{\Lambda}$ for $\mathcal{A}$ is an open subset of $\CC$ with
\begin{align*}
\dim U(\kappa) = n = \dim S(\kappa) \;\; \textnormal{for all} \;\; \kappa \in \tilde{\Lambda}.
\end{align*}
\end{defi}
\begin{lem}
\label{lemdisperion}
Let $F,G \in \RR^{n\times n}$ satisfy the following assumptions:
\begin{itemize}
\item[B1] $F,G$ are symmetric, and $G$ is positive definite;
\item[B2] $G \pm F$ is positive definite.
\end{itemize}
For $\kappa \in \CC$ define
\begin{align*}
A(\kappa) := A_{F,G}(\kappa) := \begin{pmatrix}
F & \ID_n\\ \kappa^2 \ID_n + \kappa G & 0
\end{pmatrix}
\end{align*}
and set $\mathcal{A} := \lbrace A(\kappa) \rbrace_{\kappa \in \CC}$. Then
\begin{align}
\label{domaindispersion}
\tilde{\Lambda} :=  M_\nu := \left\lbrace \kappa \in \CC\bigg\vert \; - \nu \frac{(\textnormal{Im}\, \kappa)^2}{1 + (\textnormal{Im} \kappa)^2} < \textnormal{Re}\, \kappa \right\rbrace
\end{align}
is a domain of constant splitting for $\mathcal{A}$. Furthermore, zero is not an element of any domain of constant splitting for $\mathcal{A}$.
\end{lem}
\begin{proof}
We invoke the results on dissipative symmetric hyperbolic-hyperbolic systems as found in \cite{FS21}. Let $\kappa \in \CC$. Then $A(\kappa)$ is not hyperbolic iff for some $\xi \in \RR$
\begin{align*}
i\xi \in \sigma(A(\kappa))
\end{align*}
which is equivalent to the dispersion relation
\begin{align}
\label{disprel2}
\det (\kappa G + i \xi F + \kappa^2 \ID_n + \xi^2 \ID_n) = 0.
\end{align}
But \eqref{disprel2} holds iff $e^{\kappa t + i \xi x}$ solves
\begin{align*}
G v_t + F v_x = \square v.
\end{align*}
By Theorem 3 and Proposition 1 in \cite{FS21} there exists a constant $\nu >0$ such that the last statement implies
\begin{align*}
\RE \kappa \leq - \nu \frac{\xi^2}{1+\xi^2}.
\end{align*}
As is easy to see, there exists a constant $c>0$ such that for all solutions $(\kappa,\xi) \in \CC \times \RR$ of the dispersion relation \eqref{disprel2} it holds
\begin{align*}
|\kappa| \leq c |\xi|,
\end{align*}
in particular
\begin{align*}
|\textnormal{Im}\, \kappa| \leq c |\xi|.
\end{align*}
By shrinking $\nu>0$ if necessary we find that $\kappa \in M_\nu$ implies
\begin{align*}
i\RR \cap \sigma(A(\kappa)) = \emptyset.
\end{align*}
Therefore, by continuity the real parts of the eigenvalues of $A(\kappa)$ may not change signs on $M_\nu$, and considering large $\kappa \in (0,\infty)$ one finds that the stable space $S(\kappa)$ and the unstable space $U(\kappa)$ of $A(\kappa)$ satisfy
\begin{align*}
\dim S(\kappa) = n = \dim U(\kappa)
\end{align*}
if $\kappa \in M_\nu$.
\end{proof}
\begin{defi}
(c.f.~\cite{AGJ90}) Let $\mathcal{A}= \lbrace (A^-(\kappa),A^+(\kappa)) \rbrace_{\kappa \in \CC} \subset (\CC^{2n \times 2n})^2$ be a smooth family of pairs of matrices, and set $\mathcal{A}^\pm = \lbrace A^\pm(\kappa) \rbrace_{\kappa \in \CC}$. Let $\tilde{\Lambda}^\pm$ be a domain of constant splitting for $\mathcal{A}^\pm$. Then a domain of consistent splitting $\Lambda \subset \CC$ for $\mathcal{A}$ is defined as the intersection
\begin{align*}
\Lambda = \tilde{\Lambda}^- \cap \tilde{\Lambda}^+.
\end{align*}
\end{defi}
Suppose we are given a shock $(v^-,v^+,s) \in \mathcal{V} \times \mathcal{V} \times (-1,1)$ of $g(v)_t + f(v)_x = 0$ with end-states $v^\pm$ connected by a smooth shock profile $\phi:\RR \rightarrow \mathcal{V}$. Consider the system
\begin{align}
\label{varequprof}
\xi^\prime = A_{s,\kappa}(\phi(x)) \xi \;\; \textnormal{on} \;\; \CC^{2n}.
\end{align}
Set at the end-states
\begin{align*}
A^\pm_{\kappa} := A_{s,\kappa}(v^\pm) = A_{Df_s^\pm, Dg_s^\pm}(\kappa)
\end{align*}
with
\begin{align*}
Df^\pm_s := Df_s(v^\pm) \;\; \textnormal{and} \;\; Dg_s^\pm := Dg_s(v^\pm),
\end{align*}
and let $\mathcal{A} := \lbrace (A^-_{\kappa}, A^+_{\kappa}) \rbrace_{\kappa \in \CC}$. Then $Df^\pm_s, Dg^\pm_s$ satisfy B1, B2 of Lemma \ref{lemdisperion}, and as a direct consequence we obtain:
\begin{cor}
\label{cordispersion}
There exists $\nu \in (0,\infty)$ such that $\Lambda:=M_\nu$ is a domain of consistent splitting for $\mathcal{A}$. In particular, setting $\CH^+:= \lbrace \kappa \in \CC|\; \RE \kappa \geq 0 \rbrace$ we have
\begin{align*}
\CH^+ \backslash \lbrace 0 \rbrace \subset \Lambda.
\end{align*}
Note $0 \not \in \Lambda$.
\end{cor}
Let $\Lambda$ be a domain of consistent splitting as found in Corollary \ref{cordispersion}. For $\kappa \in \Lambda$ there are unique $n$-dimensional subspaces \cite{C65}
\begin{align*}
\EH^\pm(\kappa) \subset \GG^{2n}_n(\CC)
\end{align*}
such that any solution $\xi$ of \eqref{varequprof} with $\xi(0) \in \EH^\pm(\kappa)$ decays to zero at $\pm \infty$. Here $\GG^{2n}_n(\CC)$ denotes the complex Grassmann manifold consisting of all $n$ -dimensional subspace of $\CC^{2n}$. On $\Lambda$ the subspaces $\EH^\pm(\kappa)$ depend analytically on $\kappa$ \cite{AGJ90,S02}:
\begin{defi}
The analytic bundles
\begin{align*}
\EH^\pm: \Lambda \rightarrow \GG^{2n}_n(\CC)
\end{align*}
are called Evans bundles for the shock profile $\phi$. An Evans function 
\begin{align*}
\mathcal{E}: \Lambda \rightarrow \CC
\end{align*}
is the Wronskian determinant of an analytic choice of bases for $\EH^-$ and $\EH^+$ on $\Lambda$.
\end{defi}
Later on we will restrict attention to the bundles $\EH^\pm$ on $\CH^+ \backslash \lbrace 0 \rbrace$.
\begin{defi}
\label{defspecstab}
The shock profile $\phi$ is called spectrally stable if the Evans bundles $\EH^\pm$ possess analytic extensions on $\CH^+$, and an Evans function $\mathcal{E}$ which is defined by taking the Wronskian determinant of an analytic choice of bases of the extensions of $\EH^-$ and $\EH^+$ on $\CH^+$ satisfies the Evans function condition (see \cite{MZ02}):
\begin{align*}
(\forall \kappa \in \CH^+: \; \mathcal{E}(\kappa) = 0 \; \Leftrightarrow \; \kappa = 0) \;\; \textnormal{and} \;\; (\mathcal{E}^\prime(0) \neq 0).
\end{align*}
\end{defi}
With the above definitions we can give a short statement of the central result of this paper:
\begin{thm}
\label{maintheorem}
For sufficiently small $\varepsilon >0$, the shock profile $\phi_\varepsilon$ is spectrally stable.
\end{thm}
Theorem \ref{maintheorem} is proved in section 4 by construction of holomorphic bundles $H^\pm_\varepsilon$ on $\CH^+$ which vary smoothly in $\varepsilon$ and are rescaled versions $H^\pm_\varepsilon(\zeta) = \EH^\pm_\varepsilon(\varepsilon^2 \zeta)$ of the Evans bundles $\EH^\pm_\varepsilon$ of $\phi_\varepsilon$.
\section{Resolvent kernel}
\label{secresker}
This section discusses the appearance of an Evans function in the construction of the resolvent kernel for the linearized system \eqref{linearHsCL} thereby giving a meaning to the Evans function condition in Definition \ref{defspecstab}. Following \cite{ZH98} this marks the starting point of an extended spectral theory for the linear operator associated with \eqref{linearHsCL} yielding an extension of its resolvent into the essential spectrum while characterizing the latter's regularity as well.
\subsection{Semi-group theory}
Suppose that we are given a tuple $(v^-,v^+,s) \in \mathcal{V} \times \mathcal{V} \times (-1,1)$ and $\phi \in C^\infty(\RR)$ such that
\begin{itemize}
\item[S1] $dg(v^\pm), dg(v^-) \pm df(v^-)$ and $dg(v^+) \pm df(v^+)$ are positive definite;
\item[S2] $df(v^\pm)$ is strictly hyperbolic w.r.t.~$dg(v^\pm)$;
\item[S3] $(v^-,v^+,s)$ is a $p$-shock for $g(v)_t +f(v)_x = 0$;
\item[S4] $\phi$ is a shock profile for the shock $(v^-,v^+,s)$.
\end{itemize}
A Lorentz boost of speed $s$ turns $\phi$ into a stationary solution of
\begin{align*}
g_s(v)_t + f_s(v)_x = \square v
\end{align*}
Linearizing at $\phi$ yields
\begin{align}
\label{linHsCL}
G v_t + (Fv)_x = \square v.
\end{align}
where we have set
\begin{align*}
G(x): = Dg_s(\phi(x)) \;\; \textnormal{and} \;\; F(x) := Df_s(\phi(x)).
\end{align*}
We will discuss the semi-group theory of \eqref{linHsCL} and its resolvent kernel. Associated to \eqref{linHsCL} we define the closed linear operator
\begin{align*}
D(L_x) & := H^2(\RR;\CC^n) \times H^1(\RR;\CC^n), \\
L_x V & := \begin{pmatrix} V_2 \\ \partial_x^2 V_1 - \partial_x (F V_1) - G V_2 \end{pmatrix}
\end{align*}
on the Hilbert space
\begin{align*}
H & = H^1(\RR;\CC^n) \times L^2(\RR;\CC^n),\\
\skp{V,W}{H} & := \skp{V_1,W_1}{L^2} + \skp{\partial_x V_1,\partial_x W_1}{L^2} + \skp{V_2,W_2}{L^2}
\end{align*}
$L_x$ is equal to $A + B$ where
\begin{align*}
D(A) := D(L_x), \; AV := \begin{pmatrix}
V_2 \\ \partial_x^2 V_1
\end{pmatrix}
\end{align*}
and
\begin{align*}
D(B) := H, \; BV := - \begin{pmatrix}
0 \\ \partial_x (F V_1) + G V_2.
\end{pmatrix}
\end{align*}
By Theorem 4.5 in Chapter 7 of Pazy's book \cite{P83} $A$ is the generator of a $C_0$-group $\lbrace e^{At} \rbrace_{t \in \RR}$ on $H$ satisfying
\begin{align*}
\norm{e^{At}}{} \leq e^{|t|}.
\end{align*}
The operator $B$ is a linear and bounded operator on $H$ and therefore $L_x$ is a perturbation of $A$ by the bounded linear operator $B$. By Theorem 1.1, Chapter 3 in \cite{P83} $L_x$ is the generator of a $C_0$-group $\lbrace e^{L_xt} \rbrace_{t \in \RR}$ of operators on $H$ satisfying
\begin{align*}
\norm{e^{L_x t}}{} \leq e^{(1+\norm{B}{})|t|}.
\end{align*}
Hence, we arrive at
\begin{lem}
(Corollary 7.5, Chapter 1, \cite{P83}) Let $e^{At}$ be the group of operators generated by $L_x$ in $H$. If $\gamma >0$ is large enough then for all $V \in D(L_x^2)$ it holds
\begin{align}
\label{resolventformula}
e^{At}V = \frac{1}{2\pi i} \int_{\gamma - i \infty}^{\gamma + i \infty} e^{\kappa t} (\kappa - L_x)^{-1} V d\kappa \in D(L_x)
\end{align}
and for every $0<\delta <1$, the integral converges in $H$ uniformly in $t$ for $|t| \in [\delta,1/\delta]$.
\end{lem}
\begin{rem}
\begin{itemize}
\item[(i)] Note $D(L_x^2) = \lbrace V \in H: \; V \in D(L_x), L_x V \in D(L_x) \rbrace$. Ellipticity of $\partial_x^2$ implies $D(L_x^2) = H^3(\RR;\CC^n) \times H^2(\RR;\CC^n)$.
\item[(ii)] The integral in \eqref{resolventformula} is a principal value integral in the sense that for all $U \in D(L_x^2)$ the line integral
\begin{align*}
\int_{\gamma -iK}^{\gamma + iK} e^{\kappa t} (\kappa - L_x)^{-1} V d\kappa
\end{align*}
converges for $K \rightarrow \infty$ in $H$ uniformly in $t$ for $|t| \in [\delta,1/\delta]$.
\end{itemize}
\end{rem}
The Green's function $\mathfrak{G}(t) = \fG(t,x,y)$ for $\partial_t - L_x$ is defined by its action on test functions $h,k \in C^\infty_0(\RR;\CC^n) \times C^\infty_0(\RR;\CC^n)$ via
\begin{align*}
\fG(t)(h,k) := (k,e^{L_x t} h)
\end{align*}
where for $h,k \in L^2(\RR;\CC^n) \times L^2(\RR;\CC^n) \cong L^2(\RR;\CC^{2n})$ we have used the pairing
\begin{align*}
(h,k) = \int_{\RR} h(x)^t k(x) dx \in \CC.
\end{align*}
On the resolvent set $\rho(L_x)$ of $L_x$ the resolvent kernel $\fG_\kappa = \fG_\kappa(x,y)$ is similarly defined by
\begin{align*}
\fG_{\kappa}(h,k) = (k,(L_x-\kappa)^{-1}h).
\end{align*}
Hence, for $\gamma > 0$ large enough one finds in the sense of distributions
\begin{align*}
\fG(t) = -\frac{1}{2\pi i} P.V. \int_{\gamma - i \infty}^{\gamma + i \infty} e^{\kappa t} \fG_\kappa d\kappa.
\end{align*}
For $\kappa \in \rho(L_x)$ the resolvent kernel is the uniquely defined distribution satisfying
\begin{align*}
(L_x-\kappa) \fG_\kappa = \delta_0(x-y) \;\; \textnormal{and}\;\; \fG_\kappa(h) \in H.
\end{align*}
Here $\delta_0$ is the Dirac distribution at $0$ and for $h,k \in C^\infty_0(\RR;\CC^{2n})$ the expression $\delta_0(x-y)$ denotes the distribution defined by inserting $h,k$ into the pairing $(\cdot,\cdot)$, i.e.
\begin{align*}
\delta_0(h,k) = (h,k).
\end{align*}
\begin{rem}
As customary we frequently switch between interpreting $\fG_\kappa$ as a distribution acting on test functions in $C^\infty_0(\RR^2;\CC^{2n})$ and as a map taking test functions $C^\infty_0(\RR;\CC^{2n})$ to distributions acting on $C^\infty_0(\RR;\CC^{2n})$. Since $\fG_\kappa$ is constructed as a measurable function there is no difficulty in changing between these view points.
\end{rem}
\subsection{Construction of the resolvent kernel}
For $\kappa \in \CC$ let
\begin{align*}
D(l_{x,\kappa}) := H^2(\RR;\CC^n) \subset L^2(\RR;\CC^n)
\end{align*}
and define the closed linear operator
\begin{align*}
l_{x,\kappa} v := \partial^2_x v - (Fv)_x - \kappa(\kappa + G) v, \;\; v \in D(l_{x,\kappa}),
\end{align*}
on $L^2(\RR;\CC^n)$. For $\kappa \in \CC$ with $ 0 \in \rho(l_{x,\kappa})$ let $\fg_\kappa$ be the uniquely defined distribution satisfying
\begin{align*}
l_{x,\kappa} \fg_\kappa = \delta_0(x-y) \;\; \textnormal{and} \;\; \fg_\kappa(h) \in L^2,
\end{align*}
where again for $h,k \in C^\infty_0(\RR;\CC^n)$
\begin{align*}
\delta_0(h,k) = \int_\RR h(x)^t k(x) dx .
\end{align*}
We find that indeed $\fg_\kappa$ is well-defined and for $h \in C^\infty_0(\RR;\CC^n)$ we have
\begin{align*}
\fg_\kappa(h) = l_{x,\kappa}^{-1} h.
\end{align*}
In particular, $\fg_\kappa(h) \in H^2$. We refer to $\fg_\kappa$ as the resolution kernel of $l_{x,\kappa}$.
\begin{lem}
Let $\kappa \in \CC$. Let $S_\kappa: H \rightarrow H$ denote the bounded multiplication operator
\begin{align*}
S_\kappa(V)(x) = \begin{pmatrix}
\ID_n & 0\\ \kappa + G(x) & \ID_n
\end{pmatrix} V(x),
\end{align*}
If $0 \in \rho(l_{x,\kappa})$ then the linear operator
\begin{align*}
\begin{pmatrix}
0 & l_{x,\kappa}^{-1}\\ \id_n & \kappa l_{x,\kappa}^{-1}
\end{pmatrix} \circ S_{\kappa}: H \rightarrow H
\end{align*}
is a bounded inverse of $L_x -\kappa$. If $\kappa \in \rho(L_x)$ then the linear operator 
\begin{align*}
(\id_n, 0) (L_x-\kappa)^{-1} \begin{pmatrix}
0\\ \id_n
\end{pmatrix}: L^2(\RR;\CC^n) \rightarrow L^2(\RR;\CC^n)
\end{align*}
is a bounded inverse of $l_{x,\kappa}$. In particular, it holds
\begin{align*}
0 \in \rho(l_{x,\kappa}) \; \Leftrightarrow \; \kappa \in \rho(L_x).
\end{align*}
\end{lem}
\begin{proof}
By elementary algebraic manipulations we have
\begin{align*}
\begin{pmatrix}
f_1\\f_2
\end{pmatrix} \in \ker (L_x-\kappa) \; \Leftrightarrow \; f_1 \in \ker l_{x,\kappa} \; \textnormal{and} \; \kappa f_1 = f_2.
\end{align*}
In particular,
\begin{align*}
 \ker (L_x -\kappa) = \lbrace 0 \rbrace \; \Leftrightarrow \; \ker l_{x,\kappa} = \lbrace 0 \rbrace.
\end{align*}
Since for $0 \in \rho(l_{x,\kappa})$ a bounded right inverse for $L_x - \kappa$ is given by
\begin{align*}
\begin{pmatrix}
0 & l_{x,\kappa}^{-1}\\ \id_n & \kappa l_{x,\kappa}^{-1}
\end{pmatrix} \circ S_{\kappa},
\end{align*}
and for $\kappa \in \rho(L_x)$ a bounded right inverse for $l_{x,\kappa}$ is given by
\begin{align*}
(\id_n, 0) (L_x-\kappa)^{-1} \begin{pmatrix}
0\\ \id_n
\end{pmatrix}, 
\end{align*}
the claim follows.
\end{proof}
By the previous lemma, or by direct calculation and use of definitions, we obtain for $\kappa \in \rho(L_x)$ the fundamental relation
\begin{align}
\label{Gcharg}
\fG_\kappa(x,y) = \begin{pmatrix}
(\kappa + G(y))\fg_\kappa(x,y) & \fg_\kappa(x,y)\\
\delta_0(x-y) + \kappa(\kappa + G(y)) \fg_\kappa(x,y) & \kappa \fg_\kappa(x,y)
\end{pmatrix}.
\end{align}
Hence a representation of $\fg_\kappa$ leads directly to a representation of $\fG_\kappa$. In the following we construct $\fg_\kappa$ in standard fashion from appropriately decaying solutions $(p,p^\prime)^t$ where $p$ solves $l_{x,\kappa} p = 0$: Consider the linear non-autonomous ordinary differential equation
\begin{align}
\label{ODEg}
\xi^\prime = \tA_\kappa(x) \xi \;\; (\xi = (p,p^\prime)^t \in \CC^{2n})
\end{align}
with coefficient matrix
\begin{align*}
\tA_{\kappa}(x) = \begin{pmatrix}
0 & \ID_n\\ \kappa(\kappa \ID_n + G(x)) + F^\prime(x) & F
\end{pmatrix}.
\end{align*}
The differential system \eqref{ODEg} is the standard first-order formulation of $l_{x,\kappa} p := 0$. Set $F^\pm := Df_s(v^\pm), G^\pm = Dg_s(v^\pm)$ and $\tA^\pm_\kappa := \tA_\kappa(\pm \infty)$, For any $m \in \NN_0$ there is a constant $\theta > 0$ such that
\begin{align*}
\lim\limits_{x \rightarrow \pm \infty} (|(\partial^m_x(F(x) - F^\pm)| + |\partial^m_x(G(x) - G^\pm)|)e^{\theta|x|} = 0,
\end{align*}
i.e.~the coefficient matrix $\tA_\kappa$ in \eqref{ODEg} is asymptotically constant with exponential rate. Because
\begin{align*}
\tA^\pm_\kappa \sim A_{F^\pm,G^\pm}(\kappa)= \begin{pmatrix}
F^\pm & \ID_n\\ \kappa(\kappa + G^\pm) & 0
\end{pmatrix}
\end{align*}
and $F^\pm, G^\pm$ satisfy the assumptions of Lemma \ref{lemdisperion} there exists $\nu >0$ such that $\Lambda := M_\nu$ is a domain of consistent splitting for $\lbrace (\tA^-_\kappa, \tA^+_\kappa)\rbrace_\kappa$. On $\Lambda$ the spectral projections $P^{\pm,S}_\kappa$ and $P^{\pm,U}_\kappa$ onto the $n$-dim.~stable subspace $S^\pm_\kappa$ and the $n$-dim.~unstable subspace $U^\pm_\kappa$ of $\tA^\pm_\kappa$ are well-defined and holomorphic in $\kappa$. The next lemma shows that these spectral projections have holomorphic extensions in a small neighbourhood of zero.
\begin{lem}
\label{lemprojectionextension}
Let $F, G$ satisfy the assumptions of Lemma \ref{lemdisperion}. Suppose in addition that $F$ is invertible, and let $\lambda_1, \ldots, \lambda_m$ be the eigenvalues of $F$ counted without their multiplicities. Denote by $\nu_1, \ldots, \nu_{\tilde{m}}$ the eigenvalues of $F^{-1}G (\sim G^{1/2} F^{-1} G^{1/2})$ with corresponding projections $p_1, \ldots, p_{\tilde{m}}$. Set $A(\kappa) := A_{F,G}(\kappa)$ and $\mathcal{A} := \lbrace A(\kappa) \rbrace_\kappa$, and let $\tilde{\Lambda}$ be a domain of constant splitting for $\mathcal{A}$ as found in Lemma \ref{lemdisperion}. On $\tilde{\Lambda}$ denote by $P^S=P^S(\kappa)$ resp.~$P^U=P^U(\kappa)$ the holomorphic spectral projection onto the stable resp.~unstable subspace of $A=A(\kappa)$. Then the slow eigenvalues of $A(\kappa)$ expand as
\begin{align*}
\nu_{jk}(\kappa) = -\nu_j \kappa + \beta_{jk} \kappa^2 + \mathcal{O}(\kappa^2).
\end{align*}
with total projection $P_j=P_j(\kappa)$ where $\beta_{jk}$ are the eigenvalues of
\begin{align*}
-(1 - (\nu_j)^2) p_j F^{-1}p_j|_{\mathcal{R}(p_j)} = \nu_j ((\nu_j)^2 -1) p_j G^{-1} p_j|_{\mathcal{R}(p_j)}
\end{align*}
with $\textnormal{sgn}\, \nu = \textnormal{sgn}\, \beta_{jk}$, and the fast eigenvalues of $A(\kappa)$ expand as
\begin{align*}
\lambda_j(\kappa) = \lambda_j + \OO(1).
\end{align*}
For $|\kappa| \ll 1$ denote by $P^<(\kappa)$ the total projection belonging to the $\lambda_j$-groups with $\lambda_j <0$ and by $P^>(\kappa)$ the total projection belonging to the $\lambda_j$-groups with $\lambda_j >0$. Then
\begin{align*}
\dim \left(P^>(\kappa) + \sum_{\nu_j < 0} P_j(\kappa)\right) = n = \dim \left(P^<(\kappa) + \sum_{\nu_j > 0} P_j(\kappa)\right),
\end{align*}
and if in addition $\kappa \in \Lambda$ then
\begin{align*}
P^>(\kappa) + \sum_{\nu_j < 0} P_j(\kappa) & = P^{U}(\kappa),\\
P^<(\kappa) + \sum_{\nu_j > 0} P_j(\kappa) & = P^{S}(\kappa),
\end{align*}
In particular $P^> + \sum_{\nu_j < 0} P_j$ is a holomorphic exentsion of $P^{U}$ into a small neighbourhood of $0$, and likewise for $P^< + \sum_{\nu_j > 0} P_j$ and $P^{S}$.
\end{lem}
\begin{proof}
Since $F$ is symmetric hyperbolic we may assume without loss of generality that $F = \diag(\lambda_1,\ldots, \lambda_m)$ where the eigenvalues $\lambda_j$ appear according to the ordering $\lambda_j < \lambda_{j^\prime}, j<j^\prime,$ counted with multiplicities.
A matrix consisting of eigenvectors of $A(0)$ is given by
\begin{align*}
\mathcal{W} := \begin{pmatrix}
\ID_n & \ID_n\\
0 & -F
\end{pmatrix}.
\end{align*}
$\mathcal{W} \in \RR^{2n \times 2n}$ is invertible with inverse
\begin{align*}
\mathcal{W}^{-1} = \begin{pmatrix}
\ID_n & F^{-1}\\
0 & -F^{-1}
\end{pmatrix}.
\end{align*}
Use $\mathcal{W}^{-1}$ to change basis and write $A$ as
\begin{align*}
T(\kappa) & := \mathcal{W}^{-1} A^\pm(\kappa) \mathcal{W}\\
& = \begin{pmatrix} F & 0\\ 0 & 0 \end{pmatrix} + \kappa \begin{pmatrix} C & C\\ -C & -C \end{pmatrix} + \kappa^2 \begin{pmatrix} F^{-1} & F^{-1}\\ -F^{-1} & -F^{-1} \end{pmatrix}\\
& := T^{(0)} + \kappa T^{(1)} + \kappa^2 T^{(2)}
\end{align*}
where $C := F^{-1} G$. We are interested in the $0$-group of $T$. Notice that the structure of $T$ is almost the same as in \cite{S20} and we refer to the results therein after making sure that the technique carries over (almost word by word). By hyperbolicity of $F$ we may separate the n-fold eigenvalue $0$ of $T(0)$ from the eigenvalues of $F$ and define for $|\kappa|$ small enough the total projection $P_0(\kappa)$ for the $0$-group of eigenvalues of $T(\kappa)$ by the usual Dunford formula. One has
\begin{align*}
P_0 := P_0(0) = \begin{pmatrix}
0 & 0 \\
0 & \ID_n
\end{pmatrix}.
\end{align*}
In order to apply the reduction process to expand $T(\kappa) P_0(\kappa)$ consider
\begin{align*}
T_0(\kappa) := \frac{1}{\kappa} T(\kappa) P_0(\kappa)
\end{align*}
which is holomorphic by semi-simplicity of the eigenvalue $0$. In particular it holds
\begin{align*}
T_0(0) = P_0 T^{(1)} P_0 = \begin{pmatrix} 0 & 0\\ 0 & -C \end{pmatrix}.
\end{align*}
Because $G$ is positive definite and
\begin{align*}
C \cong G^{1/2} F^{-1} G^{1/2}
\end{align*}
with the right hand side being symmetric and invertible we see that $C$ has only the pairwise distinct semi-simple eigenvalues $\lbrace \nu_1,\ldots,\nu_{\tilde{m}} \rbrace \subset \RR \backslash \lbrace 0 \rbrace$. In fact, the eigenvalues of $C$ are exactly the reciprocal eigenvalues of $F$ w.r.t.~$G$, and by $G\pm F> 0$ we find $|\nu_j|>1$. Again the total projections $P_j(\kappa)$ of the $-\nu_j$-group of $T_0$ are defined and holomorphic for $|\kappa|$ small enough and satisfy $P_0(\kappa) = \sum_{j} P_j(\kappa)$. If $p_j$ is the eigenprojection of the eigenvalue $\nu_j$ of $C$ one has
\begin{align*}
P_j := P_j(0) = \begin{pmatrix}0 & 0\\ 0 & p_j \end{pmatrix}.
\end{align*}
Carrying on as in the proof of Proposition 2.2.~and Corollary 2.3.~in \cite{S20} yields the claimed expansions of the eigenvalues of $A(\kappa)$ for small $\kappa$.\\
For $|\kappa| \ll 1$ and $\kappa \in \tilde{\Lambda} \cap \RR$ we deduce from the signs of the leading coefficients in the eigenvalue expansions that
\begin{align*}
P_>(\kappa) + \sum_{\nu_j < 0} P_j(\kappa) & = P^{U}(\kappa)\\
P_<(\kappa) + \sum_{\nu_j > 0} P_j(\kappa) & = P^{S}(\kappa).
\end{align*}
By analytic extension and the fact that the real parts of the eigenvalues of $A(\kappa)$ may not change sign in $\tilde{\Lambda}$ the above equalities of projections hold true for all small enough $\kappa \in \tilde{\Lambda}$. Since nearby projections have images of the same dimensions by Lemma 4.10.~\cite{K95} and $\dim P^U(\kappa) = n = \dim P^S(\kappa)$ we find for $|\kappa|$ small
\begin{align*}
\dim \left(P_>(\kappa) + \sum_{\nu_j < 0} P_j(\kappa)\right) = n = \dim \left(P_<(\kappa) + \sum_{\nu_j > 0} P_j(\kappa)\right).
\end{align*}
For fixed $j$ the coefficients $\beta_{jk}$ have as eigenvalues of
\begin{align*}
- (1 - \nu_j^2) p_jF^{-1}p_j & = - (1 - \nu_j^2) p_j CG^{-1} p_j\\
& = \nu_j (\nu_j^2-1) p_j G^{-1} p_j.
\end{align*}
restricted to the range $\mathcal{R}(p_j)$ of $p_j$ the same signs as the eigenvalue $\nu_j$ of $C=F^{-1}G$ because $|\nu_j| > 1$ and $G^{-1}$ is positive definite.
\end{proof}
\begin{rem}
In analogy to Remark 2.2.~in \cite{ZH98} we call 
\begin{align*}
B_{eff}^\pm := \sum_j p_j (G^\pm)^{-1} p_j
\end{align*}
an \grqq effective artificial viscosity matrix\grqq.
\end{rem}
Using Lemma \ref{lemprojectionextension} with $\tA_\kappa^\pm \sim A_{F^\pm,G^\pm}(\kappa)$ standard techniques in asymptotic ODE-theory (e.g.~\cite{GZ98,ZH98,MZ02}) imply:
\begin{lem}
Let $r>0$ so small that the conclusions of Lemma \ref{lemprojectionextension} hold. Let $\kappa_0 \in \Lambda \cup B_r(0)$. Then there exists locally in $\Lambda \cup B_r(0)$ analytic solutions $\phi^+_j(x;\kappa), \phi^-_{n+j}(x;\kappa), j = 1,\ldots,n,$ of \eqref{ODEg} such that
\begin{itemize}
\item[(i)] $\SP \lbrace \phi^-_{n+j}(x;\kappa) \rbrace_j \rightarrow U^-_\kappa$ for $x \rightarrow - \infty$, and $\SP \lbrace \phi^+_j(x;\kappa) \rbrace_j \rightarrow S^+_\kappa$ for $x \rightarrow + \infty$;
\item[(ii)] If $\kappa_0 \in \Lambda$ then there exists $\theta_0 >0$ such that for all $|\kappa - \kappa_0| \ll 1$ the solutions $\phi^+_j, \phi^-_{n+j}$ decay according to $|\phi^+_j(x;\kappa)| = \mathcal{O}(e^{-\theta_0 |x|}), x \geq 0,$ and $|\phi^-_{n+j}(x;\kappa)| = \mathcal{O}(e^{-\theta_0 |x|}), x \leq 0$.
\end{itemize}
\end{lem}
Following \cite{ZH98} $\fg_\kappa$ has to satisfy the jump condition
\begin{align*}
\left[ \begin{matrix} \fg_\kappa & \partial_x \fg_\kappa\\ \partial_y \fg_\kappa & \partial_x \partial_y \fg_\kappa \end{matrix} \right]_{x=y} = J(y):= \begin{pmatrix}
0 & -\ID_n\\ \ID_n & F(y)
\end{pmatrix}.
\end{align*}
Moreover, setting
\begin{align*}
\Phi^+ = (\phi^+_j)_{1\leq j \leq n} \;\; \textnormal{and} \;\; \Phi^- = (\phi^-_{n+j})_{1\leq j \leq n}
\end{align*}
one finds for $x>y$
\begin{align*}
\begin{pmatrix} \fg_\kappa & \partial_x \fg_\kappa\\ \partial_y \fg_\kappa & \partial_x \partial_y \fg_\kappa \end{pmatrix} = (\Phi^+(x;\kappa),0) (\Phi^+(y;\kappa),\Phi^-(y;\kappa))^{-1} J(y)
\end{align*}
and for $x<y$
\begin{align*}
\begin{pmatrix} \fg_\kappa & \partial_x \fg_\kappa\\ \partial_y \fg_\kappa & \partial_x \partial_y \fg_\kappa \end{pmatrix} = -(0,\Phi^-(x;\kappa)) (\Phi^+(y;\kappa),\Phi^-(y;\kappa))^{-1} J(y).
\end{align*}
Since the solution propagator of \eqref{ODEg} is holomorphic in $\kappa$ with values in $\textnormal{GL}_{2n}(\CC)$ the regularity of $\fg_\kappa$ in $\kappa$ and the position and order of its poles are determined by the local Evans function
\begin{align*}
\EE(\kappa) = \det (\Phi^+(0;\kappa),\Phi^-(0;\kappa))
\end{align*}
where the zeros of $\EE$ correspond to poles of $g_\kappa$. Summarizing we obtain the analogue of Proposition 4.7.~\cite{ZH98}:
\begin{prop}
For $\theta > 0$ small enough the distribution $\fg_\kappa(x,y)$ is meromorphic on $\CH^+-\theta$ with isolated poles of finite multiplicity corresponding in position and multiplicity to the zeros of a local Evans function of \eqref{ODEg}. On $\Lambda \backslash \sigma(L_x)$ $\fg_\kappa$ agrees with the resolution kernel of $l_{x,\kappa}$.
\end{prop}
Using the relation \eqref{Gcharg} we obtain as a corollary of the previous proposition:
\begin{cor}
For $\theta > 0$ small enough the distribution $\fG_\kappa(x,y)$ defined through $\fg_\kappa$ by \eqref{Gcharg} is meromorphic on $\CH^+-\theta$ with isolated poles of finite multiplicity corresponding in position and multiplicity to the zeros of a local Evans function of \eqref{ODEg}. On $\Lambda \backslash \sigma(L_x)$ $\fG_\kappa$ agrees with the resolvent kernel of $L_x$.
\end{cor}
\section{Grassmannian flows}
\label{Sec3}

An important tool in the following investigations are flows on Grassmann manif\-olds that are induced on the latter by the flow of a linear ordinary differential equation (see \cite{AGJ90}). In this section the definitions and results needed in the later analysis are presented including a derivation of the well-known algebraic Riccati equation that governs the evolution of vector spaces under a linear differential equation. More background material may be found in many geometry textbooks, such as \cite{GH78} and \cite{L13}. Furthermore, flows on Grassmannian manifolds are a popular topic in control theory, see e.g.~\cite{DM90}, where the algebraic Riccati equation is treated, too. In \cite{FGP94} the reader may find a nice overview of the topic and some characterizations of smoothness of functions taking values in Grassmannian manifolds.

Let $d,m \in \NN$ with $m \leq d$. Set
\begin{align*}
M_{d,m} := \lbrace B \in \CC^{d \times m}: \; \textnormal{rank}\, B = m \rbrace
\end{align*}
and for $B \in M_{d,m}$ denote by $\SP B$ the $m$-dimensional subspace of $\CC^d$ spanned by the columns of $B$. For any matrix $B \in \CC^{d \times m}$ we will write
\begin{align*}
B = \begin{pmatrix}
B_1 \\ B_2
\end{pmatrix}
\end{align*} 
for appropriate matrices $B_1 \in \CC^{m \times m}$ and $B_2 \in \CC^{(d-m) \times m}$, and $B=(b_j)_j$ will mean that $b_j \in \CC^d$ is the $j$-th column of $B$. $M_{d,m}$ is an open subset of $\CC^{d \times m}$ and is endowed with the standard analytic structure as an open submanifold of $\CC^{d \times m}$. The set
\begin{align*}
\GG^d_m(\CC) := \lbrace X| \; X \;\textnormal{is an} \; m \; \textnormal{dim. subspace of} \; \CC^d \rbrace.
\end{align*}
can be given an analytic structure yielding a compact complex analytic manifold of dimension $(d-m)m$ which is called a Grassmann or Grassmannian manifold. For an ordered basis $\BB=(b_1,\ldots,b_d)$ of $\CC^d$ canonical local coordinates of $\GG^d_m(\CC)$ centered at $X=\SP \lbrace b_1,\ldots,b_m \rbrace$ are defined by
\begin{align*}
\phi_\BB: \CC^{(d-m)\times m} & \rightarrow \GG^d_m(\CC)\\
T & \mapsto \SP M_\BB \begin{pmatrix} \ID_m\\T \end{pmatrix}
\end{align*}
where
\begin{align*}
M_\BB = (b_1 \ldots b_d) \in \textnormal{GL}_d(\CC).
\end{align*}
Note $\phi_\BB(0) = X$. The analytic manifolds $M_{d,m}$ and $\GG^d_m(\CC)$ are connected by the fact that
\begin{align*}
\pi_{d,m}: M_{d,m} \rightarrow \GG^d_m(\CC), \; B \mapsto \SP B
\end{align*}
is an analytic quotient map.\\
Let $A \in \CC^{d \times d}$ and consider the linear ordinary differential equation
\begin{align}
\label{ODE}
\xi^\prime = A\xi \;\;(\xi \in \CC^d)
\end{align}
with flow $((t,\xi) \mapsto e^{At} \xi)$. For $X \in \GG^d_m(\CC)$ with basis $\BB_X =(b_1,\ldots,b_m)$ the Grassmannian flow on $\GG^d_m(\CC)$ associated to \eqref{ODE} starting in $X$ is defined by
\begin{align*}
X(t) = \pi_{d,m} ((e^{At} b_j)_{j\leq m}) = \SP (e^{At} b_j)_{j\leq m}.
\end{align*}
The definition of $(t \mapsto X(t))$ is independent of the choice of basis of $X$ and gives rise to a flow $((t,X) \mapsto \Phi(t,X))$ on $\GG^d_m(\CC)$ by letting $(t \mapsto \Phi(t,X))$ be the Grassmannian flow on $\GG^d_m(\CC)$ associated to \eqref{ODE} starting in $X$. Augment $\BB_X$ to a basis $\BB = (b_1,\ldots,b_d)$ of $\CC^d$. Then the map
\begin{align*}
\tilde{\Phi}_\BB: \RR \times \CC^{(d-m) \times m} &\rightarrow M_{d,m}\\
(t,T) & \mapsto e^{At} M_\BB \begin{pmatrix} \ID_m\\T \end{pmatrix}
\end{align*}
is smooth in $t$ and analytic in $T$. For all $(t,T) \in \RR \times \CC^{(d-m) \times m}$ we have by definitions
\begin{align}
\label{repGrassflow1}
\Phi(t,\phi_\BB(T)) = \pi_{d,m} \tilde{\Phi}_\BB(t,T)
\end{align}
which shows that $\Phi$ is smooth in $t$ and analytic in $X$ w.r.t.~to the analytic structure of $\GG^d_m(\CC)$. We give a local representation of the vector field induced by $\Phi$ on $\GG^d_m(\CC)$ by calculating the derivative
\begin{align*}
\Gamma^m(A) X := \left.\frac{\partial}{\partial t}\right\vert_{t=0} \Phi(t,X) \in T\GG^d_m(\CC).
\end{align*}
in $\phi_\BB$ coordinates. First we are interested in the differential of $\psi:= \phi^{-1} \circ \pi|_U$ where $\phi = \phi_\BB, \pi= \pi_{d,m}$ and $U:= \pi^{-1}(\phi(\CC^{(d-m) \times m}))$ is open. For $B \in U$ there is a $T \in C^{d-m \times m}$ with $\SP B = \phi(T)$. Hence, we may find $E \in \CC^{m \times m}$ such that
\begin{align*}
M_\BB \begin{pmatrix} \ID_m \\ T \end{pmatrix} = B E \;\; \Leftrightarrow \;\; \begin{pmatrix} \ID_m \\ T \end{pmatrix} = \tilde{B} E
\end{align*}
with $\tilde{B} = M_\BB^{-1}B$. We deduce that $\tilde{B}_1$ is invertible since $\tilde{B}_1 E = \ID_m$, and that $T = \tilde{B}_2 \tilde{B}_1^{-1}$ holds. Vice versa, if for $B \in M_{d,m}$ the matrix $\tilde{B} = M_\BB^{-1} B$ has invertible submatrix $\tilde{B}_1$ then 
\begin{align*}
\SP \tilde{B} = \SP \begin{pmatrix} \ID_m \\ \tilde{B}_2 \tilde{B}_1^{-1}\end{pmatrix}
\end{align*}
because multiplying $\tilde{B}$ by $\tilde{B}_1^{-1}$  from the right results only in a change of basis of $\SP B$. This implies $\SP B = \phi(\tilde{B}_2 \tilde{B}_1^{-1})$ and we have found
\begin{align*}
U = \lbrace B \in M_{d,m}: \; \tilde{B}_1 \in \textnormal{GL}_m(\CC) \; \textnormal{where} \; \tilde{B} = M_\BB^{-1} B \rbrace.
\end{align*}
From these considerations we find that the map $\psi$ is given by $\psi= \tilde{\psi} \circ M_\BB^{-1}$ where
\begin{align*}
\tilde{\psi}: \tilde{U}:=M_\BB^{-1} U & \rightarrow \CC^{(d-m) \times m},\\
\tilde{B} & \mapsto \tilde{B}_2 \tilde{B}_1^{-1}.
\end{align*}
The differential of $\tilde{\psi}$ at $\tilde{B} \in \tilde{U}$ is calculated to be
\begin{align}
\label{differentialPhi}
D\tilde{\psi}(\tilde{B}): \CC^{d \times m} \rightarrow \CC^{(d-m) \times m}, E \mapsto E_2 \tilde{B}_1^{-1} - \tilde{B}_2 \tilde{B}_1^{-1} E_1 \tilde{B}_1^{-1}.
\end{align}
If $\tilde{B}_1 = \ID_m$ then
\begin{align*}
D\tilde{\psi}(\tilde{B})E = E_2 - \tilde{B_2} E_1 = (-\tilde{B}_2,\ID_{n-m}) E.
\end{align*}
We find
\begin{align*}
\left.\frac{\partial}{\partial t}\right\vert_{t=0} (\phi^{-1} \circ\Phi)(t,\phi(T)) & = \left.\frac{\partial}{\partial t}\right\vert_{t=0} (\psi \circ \tilde{\Phi}_\BB)(t,T)\\
& = D\psi(\tilde{\Phi}_\BB(0,T))\left.\frac{\partial}{\partial t}\right\vert_{t=0} \tilde{\Phi}_\BB(t,T)\\
& = D\tilde{\psi}\left(M_\BB^{-1} M_\BB \begin{pmatrix} \ID_m\\T\end{pmatrix} \right) M_\BB^{-1} A M_\BB \begin{pmatrix} \ID_m\\T\end{pmatrix}\\
& = (-T, \ID_{d-m}) A_\BB \begin{pmatrix} \ID_m\\T\end{pmatrix}
\end{align*}
where $A_\BB \in \CC^{d \times d}$ is the representative of $A$ with respect to $\BB$, i.e.~$A_\BB = M_\BB^{-1} A M_\BB$. Hence, in canonical local coordinates $\phi_\BB$ the Grassmannian flow consists of solutions to the Riccati-type equation
\begin{align}
\label{repGrassflow2}
T^\prime = (-T, \ID_{d-m}) A_\BB \begin{pmatrix} \ID_m\\T\end{pmatrix} \;\; \textnormal{on} \;\; \CC^{(d-m)\times m}.
\end{align}
The global expression for the vector field associated to $\Phi$ reads
\begin{align}
\label{Grassvecfield}
X^\prime = \Gamma^m(A) X \;\; \textnormal{on} \;\; \GG^d_m(\CC).
\end{align}
The fact that the Grassmannian flow induced by a linear ordinary differential equations is locally of Riccati-type is of course well-known. This is summarized in the following lemma which is also a more general form of Lemma 3 (i) in \cite{FS02}.
\begin{lem}
\label{lemlocalGrass}
Let $\BB$ be an ordered basis of $\CC^d$. Then in canonical local coordinates $\phi_\BB$ the Grassmannian flow on $\GG^d_m(\CC)$ associated to \eqref{ODE} consists of solutions to the Riccati-type equation
\begin{align}
\label{RicattiGrass}
T^\prime = (-T, \ID_{d-m}) A_\BB \begin{pmatrix} \ID_m\\T\end{pmatrix} \;\; \textnormal{on} \;\; \CC^{(d-m)\times m}
\end{align}
where $A_\BB \in \CC^{d \times d}$ is the representative of $A$ w.r.t.~the basis $\BB$. If $A_\BB$ is diagonal with
\begin{align*}
A_\BB = \diag(\mu_1,\ldots,\mu_d)
\end{align*}
then \eqref{RicattiGrass} reads for $1\leq j \leq d-m, 1 \leq k \leq m$
\begin{align*}
(T)_{jk}^\prime = (\mu_{m+j}-\mu_k) (T)_{jk}.
\end{align*}
\end{lem}
On $\CC^{(d-m) \times m}$ let $\Fskp{\cdot,\cdot}{}$ denote the Frobenius inner product given by
\begin{align*}
\Fskp{E,F}{} = \sum_{j=1}^{d-m} \sum_{k=1}^m (E)_{jk} \overline{(F)}_{jk} \; \; (E,F \in \CC^{(d-m) \times m})
\end{align*}
and let $\norm{\cdot}{}$ be the norm on $\CC^{(d-m) \times m}$ induced by $\Fskp{\cdot,\cdot}{}$. We give some useful consequences of Lemma \ref{lemlocalGrass}.
\begin{cor}
\label{corhyperattractor}
Let $\BB = (b_1,\ldots,b_d)$ be an ordered basis of $\CC^d$ such that $X:= \SP \lbrace b_1,\ldots, b_m \rbrace$ and $Y:= \SP \lbrace b_{m+1},\ldots,b_d \rbrace$ are complementary invariant subspaces of $A$. Then
\begin{align*}
A_\BB = \begin{pmatrix}
A^X & 0\\0 & A^Y
\end{pmatrix}
\end{align*}
for appropriate $A^X \in \CC^{m \times m}, A^Y \in \CC^{(d-m) \times (d-m)}$ and the local representation of \eqref{Grassvecfield} at $X$ in $\phi_\BB$-coordinates is given by the linear ordinary differential equation
\begin{align}
\label{repGrassflowlin}
T^\prime = -TA^X + A^Y T.
\end{align}
If there exist $c^X, c^Y \in \RR$ with $c^X+c^Y > 0$ and for all $\eta^X \in \CC^m, \eta^Y \in \CC^{d-m}$
\begin{align}
\label{realpartsGrass}
\left\lbrace\begin{aligned}
\RE \skp{\eta^X, A^X\eta^X}{} & \geq c^X |\eta^X|^2\\
\RE \skp{\eta^Y, A^Y \eta^Y}{} & \leq -c^Y |\eta^Y|^2
\end{aligned} \right.
\end{align}
then $X$ is a hyperbolic attractor for the Grassmannian flow $\Phi$. For any $r>0$ the set $\phi_\BB(B_r(0))$ is a positively invariant neighbourhood of $X$ in $\GG^d_m(\CC)$.
\end{cor}
\begin{proof}
The assertion on $A_\BB$ and and the local form of $\Gamma^m$ follow immediately by linear algebra and \eqref{repGrassflow2}. Denote by
\begin{align*}
F(T) := -TA^X + A^YT
\end{align*}
the local representation of the vector field $\Gamma^m$ as given in \eqref{repGrassflowlin}. $X$ is a fixed point of $\Phi$ since $\phi_\BB(0) = X$ and $F(0) = 0$. We show that $DF(0)$ is negative definite and that for any $r>0$ the function
\begin{align*}
L: B_r(0) & \rightarrow \RR\\
T & \mapsto \norm{T}{}^2
\end{align*}
is a strict Lyapunov function for \eqref{repGrassflow2}. The differential of $F$ at $0$ is
\begin{align*}
DF(0)E = -E A^X + A^YE \;\; (E \in \CC^{(d-m)\times m}).
\end{align*}
For $E \in \CC^{(d-m)\times m}$ denote by $\tilde{E}_i$ the $i$-th row of $E$ and by $E_j$ the $j$-th column of $E$. Then
\begin{align*}
2 \RE \Fskp{E,DF(0)E}{} = & -\Fskp{E,EA^X}{} + \Fskp{E,A^YE}{}\\
& - \Fskp{EA^X,E}{} + \Fskp{A^YE,E}{}\\
= & - \sum_{i=1}^{d-m} \skp{\tilde{E}_i,(A^X)^t \tilde{E}_i}{} + \sum_{j=1}^m \skp{E_j, A^Y E_j}{}\\
& - \sum_{i=1}^{d-m} \skp{(A^X)^t \tilde{E}_i,\tilde{E}_i}{} + \sum_{j=1}^m \skp{A^Y E_j,E_j}{}\\
=& 2 \left( - \sum_{i=1}^{d-m} \RE \skp{\tilde{E}_i,(A^X)^t \tilde{E}_i}{} + \sum_{j=1}^m \RE \skp{E_j, A^Y E_j}{}\right)\\
\leq & 2 \left(- c^X \sum_{i=1}^{d-m} |\tilde{E}_i|^2 - c^Y \sum_{j=1}^m |E_j|^2 \right)\\
=& -2 (c^X +c^Y) \norm{E}{}^2,
\end{align*}
i.e.~$DF(0)$ is negative definite and $X$ a hyperbolic fixed point of $\Phi$ since $c^X+c^Y > 0$. Note that for the last inequality we have used
\begin{align*}
\skp{\eta, A^t \eta}{} = \skp{\overline{\eta}, A \overline{\eta}}{}.
\end{align*}
for $\eta \in \RR^m$. By exactly the same calculation we find
\begin{align*}
L(T)^\prime & = 2 \RE \Fskp{T,T^\prime}{} = 2 \RE \Fskp{T,F(T)}{} \leq - 2(c^X+c^Y) \norm{T}{}^2
\end{align*}
showing that $L$ is a strict Lyapunov function and $B_r(0)$ is positively invariant for \eqref{repGrassflowlin}.
\end{proof}
\begin{cor}
\label{corGrasssaddle}
Let $m_1, m_2 \in \NN$ such that $m+m_1+m_2 = d$. Let $\BB = (b_1,\ldots,b_d)$ be an ordered basis of $\CC^d$ such that $X:= \SP \lbrace b_1,\ldots, b_m \rbrace, Y:= \SP \lbrace b_{m+1},\ldots,b_{m+m_1} \rbrace$ and $Z:= \SP \lbrace b_{m+m_1+1},\ldots,b_{d} \rbrace$ are complementary invariant subspaces of $A$. Then
\begin{align*}
A_\BB = \begin{pmatrix}
A^X & 0 & 0\\0 & A^Y & 0\\ 0 & 0 & A^Z
\end{pmatrix}
\end{align*}
for appropriate $A^X \in \CC^{m \times m}, A^Y \in \CC^{m_1 \times m_1}, A^Z \in \CC^{m_2 \times m_2}$ and the local representation of $\Gamma^m(A)$ at $X$ in $\phi_\BB$-coordinates is given by the linear ordinary differential equation
\begin{align}
\label{repGrassflowlin2}
T^\prime = -T A^X + \begin{pmatrix}A^Y & 0 \\ 0 & A^Z \end{pmatrix} T.
\end{align}
If $\ker A = X$ and for some $c > 0$ and all $\eta^Y \in \CC^{m_1}, \eta^Z \in \CC^{m_2}$
\begin{align}
\label{realpartsGrass2}
\left\lbrace\begin{aligned}
\RE \skp{\eta^Y, A^Y \eta^Y}{} & \leq -c |\eta^Y|^2\\
\RE \skp{\eta^Z, A^Z \eta^Z}{} & \geq c |\eta^Y|^2
\end{aligned} \right.
\end{align}
then $A^X = 0$ and $X$ is a hyperbolic saddle of $\Phi$.
\end{cor}
\begin{proof}
The first part follows from the same arguments as used in the beginning of the proof of Lemma \ref{corhyperattractor}. Suppose $\ker A = X$ and \eqref{realpartsGrass2} hold. Set
\begin{align*}
V^Y := \left\lbrace E \in \CC^{(d-m) \times m}\Big\vert\; \exists E_1 \in \CC^{m_1 \times m}: E = \begin{pmatrix} E_1 \\0 \end{pmatrix}\right\rbrace,\\
V^Z := \left\lbrace E \in \CC^{(d-m) \times m}\Big\vert\; \exists E_2 \in \CC^{m_2 \times m}: E = \begin{pmatrix} 0 \\ E_2 \end{pmatrix}\right\rbrace.
\end{align*}
Let $F(T) = \begin{pmatrix}A^Y & 0 \\ 0 & A^Z \end{pmatrix} T$. Then
\begin{align*}
DF(0)E = \begin{pmatrix}A^Y & 0 \\ 0 & A^Z \end{pmatrix}E
\end{align*}
and $V^Y$ resp.~$V^Z$ are complementary invariant subspace of $DF(0)$ of dimension $m_1 m$ resp.~$m_2 m$. With a similar calculation as in the proof of Lemma \ref{corhyperattractor} one finds
\begin{align*}
\RE \skp{E,DF(0)E}{} & \leq -c \norm{E}{}^2 \;\; \textnormal{for all}\;\; E \in V^Y,\\
\RE \skp{E,DF(0)E}{} & \geq c \norm{E}{}^2 \;\; \textnormal{for all} \;\; E \in V^Z.
\end{align*}
Hence, $X$ is a hyperbolic saddle for $\Phi$ with an $m_1m$-dimensional stable manifold and an $m_2 m$-dimensional unstable manifold for $\Phi$ at $X$.
\end{proof}
This section is closed by a variant of Lemma 6 in \cite{FS02} replacing diagonalization by a suitable block-diagonalization which will be important for the treatment of \eqref{LapHsCL} in the outer and outmost spectral regime (see section \ref{Sec4}).
\begin{lem}
\label{lemma6a}
For every $d \in \NN$ there exists a constant $c = c(d)$ such that the following holds: Let $c^>,c^<: J \rightarrow \RR$ be functions with $c^>(\tau) + c^<(\tau) > 0, \tau \in J,$ and let $A,R: J \rightarrow \GL_{2d}(\CC)$ be smooth functions such that
\begin{align}
\label{blockdiag}
R(\tau)^{-1} A(\tau) R(\tau) = \begin{pmatrix}
A^>(\tau) & 0\\ 0 & A^<(\tau)
\end{pmatrix} \in \CC^{2d} \; (\tau \in J)
\end{align}
with $A^\gtrless(\tau) \in \CC^{d \times d}$ and
\begin{align}
\label{definite}
\left\lbrace\begin{aligned}
\textnormal{Re}\, \skp{\eta, A^>(\tau) \eta}{} & \geq c^>(\tau) |\eta|^2,\\
\textnormal{Re}\, \skp{\eta, A^<(\tau) \eta}{} & \leq -c^<(\tau) |\eta|^2,
\end{aligned} \right.
\end{align}
for all $\eta \in \CC^d$ and $\tau \in J$. With $\chi: \RR \rightarrow (-1,1)$ being the solution of
\begin{align}
\label{chiequatio}
\chi^\prime = (1-\chi^2) h(\chi), \; \chi(0) = 0,
\end{align}
for some smooth function $h:J \rightarrow (0,\infty)$ consider the equation
\begin{align}
\label{Grassmann}
X^\prime = \Gamma^d A(\chi) (X) \; \textnormal{on} \; \GG^{2d}_d(\CC)
\end{align}
which is the associated Grassmannian version for $d$-dimensional vector spaces of the non-autonomous linear system
\begin{align}
\label{linearsystem}
\xi^\prime = A(\chi) \xi
\end{align}
for $\xi \in \CC^{2d}$. Let $U(\tau),S(\tau)$ denote the invariant subspace of $A(\tau), \tau \in J,$ spanned by the first $d$ resp.~last $d$ columns of $R(\tau)$ and define $X^{\pm}: \RR \rightarrow \GG^{2d}_d(\CC)$ as the two solutions of \eqref{Grassmann} with
\begin{align}
\label{stablemfdstart}
X^-(-\infty) = U(-1), \; X^+(+\infty) = S(+1).
\end{align}
If
\begin{align}
\left\vert R(0)^{-1} \frac{d}{d\tau} R(\tau) \right\vert \leq c \; (\tau \in J)
\end{align}
then
\begin{align}
X^-(+\infty) = U(+1), \; X^+(-\infty) = S(-1).
\end{align}
\end{lem}
\begin{rem}
If in the above lemma the constants $c^>,c^<$ are larger than zero then $U(\tau)$ resp.~$S(\tau)$ is the unstable resp.~stable subspace of $A(\tau)$.
\end{rem}
\begin{proof}
We treat $X^-$. Fix a $\tau \in J$. Let $\phi_\tau: \CC^{d \times d} \rightarrow \mathcal{S}_\tau$ be the canonical coordinates on a maximal portion $\mathcal{S}_\tau \subset \GG^{2d}_d$ adapted to the columns of $R(\tau)$ as an ordered basis of $\CC^{2d}$ and centered at $\SP (R(\tau)_j)_{1 \leq j \leq d}$. Then equation \eqref{Grassmann} reads in local coordinates
\begin{align}
\label{Tequation1}
T^\prime = (-T, \ID_n) \tilde{A}(\tau) \begin{pmatrix}
\ID_n \\ T
\end{pmatrix} =: F_\tau(T)
\end{align}
where $T \in \CC^{d \times d}$ and $\tilde{A}(\tau)$ is the representative of $A(\tau)$ with respect to the basis given by the columns of $R(\tau)$, i.e.
\begin{align*}
\tilde{A}(\tau) = R(\tau)^{-1} A(\tau) R(\tau) = \begin{pmatrix}
A^>(\tau) & 0\\ 0 & A^<(\tau)
\end{pmatrix}.
\end{align*}
In particular,
\begin{align*}
F_\tau(T) = - T A^>(\tau) + A^<(\tau) T.
\end{align*}
The point $0$ corresponds in $\varphi_\tau$-coordinates to the subspace $U(\tau)$. Using Corollary \ref{corhyperattractor} we find that the differential $D_T F_\tau(T)$ of $F_\tau(T)$ with respect to $T$ at $T = 0$ is negative definite. In particular, if $\tau = -1$ then the coupled system \eqref{chiequatio}-\eqref{Grassmann} has exactly one repelling direction at $(-1,U(-1))$ coming from the $\chi$-equation.  To see this note that $\lbrace \tau = -1 \rbrace$ is an invariant set for the system \eqref{chiequatio}-\eqref{Grassmann}, and in local $\id_1 \times \phi_{-1}$-coordinates the vector field for the coupled $(\tau,X)$-equations has its differential given by
\begin{align*}
\begin{pmatrix}
2h(-1) & 0\\
\ast & D_T F_{-1}(0)
\end{pmatrix}
\end{align*}
at $(-1,0)$. Using the stable manifold theorem we find a unique orbit giving rise to the solution $X^-$ of \eqref{Grassmann} satisfying \eqref{stablemfdstart}.\\
Furthermore, Corollary \ref{corhyperattractor} tells us that for any $\sigma> 0$ the ball $B_\sigma(0)$ of radius $\sigma$ centered at $0$ is a positively invariant set for the local representation
\begin{align*}
T^\prime = (-T, \ID_n) \tilde{A}(\tau) \begin{pmatrix}
\ID_n \\ T
\end{pmatrix}
\end{align*}
of the Grassmannian flow induced by $A(\tau)$ on $\GG^{2d}_d(\CC)$ in $\phi_\tau$-coordinates. Therefore the construction of an appropriate Lyapunov function for the coupled $(\tau,X)$-system carries over from the proof of Lemma 6 in \cite{FS02}, and we find
\begin{align*}
X^-(+\infty) = U(+1).
\end{align*}
\end{proof}
\section{Spectral stability: Proof of main theorem}
\label{Sec4}
Let us turn to the proof of our main result Theorem \ref{maintheorem}. Set
\begin{align*}
A_{\varepsilon,\kappa}(\tau) := A_{s(\varepsilon),\kappa}(\varepsilon u_\varepsilon(\tau)).
\end{align*}
Instead of the non-autonomous ordinary differential equation \eqref{LapHsCL} it will be convenient to consider the autonomous system
\begin{align}
\label{LapHsCLauto}
\left\lbrace \begin{aligned}
\tau^\prime & = \varepsilon (1-\tau^2)h_\varepsilon(\tau)\\
\xi^\prime & = A_{\varepsilon,\kappa}(\tau) \xi
\end{aligned} \right.
\end{align}
on $J \times \CC^{2n}$ with the obvious correspondence between solutions of \eqref{LapHsCLauto} and \eqref{LapHsCL}. Making use of the linear structure of the $\xi$-equations in \eqref{LapHsCLauto} Grassmannian versions
\begin{align}
X^\prime = \Gamma^m A_{\varepsilon,\kappa}(\tau)(X) \;\; \textnormal{on} \;\; \GG^{2n}_m(\CC)
\end{align}
of the $\xi$-equations in \eqref{LapHsCLauto}, still coupled to the dominant $\tau$-equation, are investigated for various $m \in \lbrace 1, \ldots, 2n \rbrace$. Here the Grassmannian flow is again defined by taking the span of the $\xi$-part of appropriately many solutions to the system \eqref{LapHsCLauto} with linearly independent $\xi$-initial values. With this point of view the Evans bundles $\EH_\varepsilon^\pm(\kappa)$ are constructed in a way that, besides analyticity in $\kappa$, implies the smooth dependence on $\varepsilon$, too. In fact different rescalings of $\kappa$ and $\varepsilon$ will lead to certain rescaled bundles revealing different properties of the Evans bundles in different spectral regimes (c.f.~\cite{FS02,FS10}):\\
\textbf{Inner regime:} For all $r_0 > 0$ and small $\varepsilon_0=\varepsilon_0(r_0)>0$ consider all $\varepsilon \in (0,\varepsilon_0]$ and $\kappa \in \CH^+$ with
\begin{align*}
0 \leq |\kappa| \leq r_0 \varepsilon^2;
\end{align*}
\\
\textbf{Outer regime:} For some $r_0,r_1 > 0$ and $\varepsilon_0>0$ consider all $\varepsilon \in (0,\varepsilon_0]$ and $\kappa \in \CH^+$ with
\begin{align*}
r_0 \varepsilon^2 \leq |\kappa| \leq r_1;
\end{align*}
\textbf{Outmost regime:} For all $r_1 > 0$ and small $\varepsilon_0=\varepsilon_0(r_1)>0$ consider all $\varepsilon \in (0,\varepsilon_0]$ and $\kappa \in \CH^+$ with
\begin{align*}
r_1 \leq |\kappa|.
\end{align*}
The different regimes are treated in the following subsections. Before we enter into the considerations regarding the spectral regimes let us give a useful block-diagonalization of
\begin{align*}
(\varepsilon, u) \mapsto Df_{s(\varepsilon)}(\varepsilon u)
\end{align*}
where $u \in B_r(0)$ for some $r>0$ and $\varepsilon \leq \varepsilon_0= \varepsilon_0(r)$.
\begin{lem}
\label{lemdiagf}
Let $r>0$. Then there exists $\varepsilon_0 = \varepsilon_0(r)$ and smooth functions
\begin{align*}
W: [0,\varepsilon_0] \times B_r(0) & \rightarrow \textnormal{SO}_n(\RR),\\
M^<: [0,\varepsilon_0] \times B_r(0) & \rightarrow \RR^{(k-1) \times (k-1)},\\
M^0: [0,\varepsilon_0] \times B_r(0) & \rightarrow \RR,\\
M^>: [0,\varepsilon_0] \times B_r(0) & \rightarrow \RR^{(n-k) \times (n-k)},
\end{align*}
such that for all $u \in B_r(0)$
\begin{align*}
W(0,u) & = \ID_n\\
M^<(0,u) & =  \diag(\lambda_1,\ldots,\lambda_{k-1})\\
M^0(0,u) & = 0\\
\partial_\varepsilon M^0(0,u) &= e_k^t \partial_{v_k}D_v f(0) u\\
M^>(0,u) & = \diag(\lambda_{k+1},\ldots,\lambda_n),
\end{align*}
and all $(\varepsilon,u) = [0,\varepsilon_0] \times B_r(0)$
\begin{align*}
W(\varepsilon,u)^T Df_{s(\varepsilon)}(\varepsilon u) W(\varepsilon,u) = \begin{pmatrix}
M^<(\varepsilon,u) & 0 & 0\\
0 & M^0(\varepsilon,u) & 0\\
0 & 0 & M^>(\varepsilon,u)
\end{pmatrix}.
\end{align*}
\end{lem}
\begin{proof}
Let $r>0$ and consider all $\varepsilon \geq 0$ such that $B_{\varepsilon r}(0) \subset \mathcal{V}$. Define the smooth function $F(\varepsilon,u) = Df_{s(\varepsilon)}(\varepsilon u)$. Then $F(0,u) = \diag(\lambda_1,\ldots,\lambda_n)$ for all $u$ with
\begin{align*}
\lambda_j < 0 = \lambda_k = 0 < \lambda_{j^\prime}\;\; \textnormal{for} \;\; j<k<j^\prime.
\end{align*}
Hence for $\varepsilon$ small enough there are total spectral projections
\begin{align*}
P^<(\varepsilon,u) &  = \frac{1}{2\pi i} \int_{\Gamma^<} (z-F(\varepsilon,u))^{-1} dz,\\
P^0(\varepsilon,u) &  = \frac{1}{2\pi i} \int_{\Gamma^0} (z-F(\varepsilon,u))^{-1} dz,\\
P^>(\varepsilon,u) &  = \frac{1}{2\pi i} \int_{\Gamma^>} (z-F(\varepsilon,u))^{-1} dz,
\end{align*}
onto the algebraic eigenspaces of the groups of eigenvalues of $F(\varepsilon,u)$ with real parts smaller, close-to and larger than zero. $\Gamma^\gtrless, \Gamma^0$ are suitably chosen closed curves enclosing the respective group of eigenvalues. The above projections are smooth in $(\varepsilon,u)$, form a disjoint partition of unity in the sense that
\begin{align*}
P^< + P^0 + P^> &= \ID_n,\\
P^<P^0 = P^0 P^< & = 0,\\
P^>P^0 = P^0 P^> & = 0,\\
P^<P^> = P^> P^< & = 0,
\end{align*}
and for all $u$ it holds
\begin{align*}
P^<(0,u) & = \Pi^< :=\begin{pmatrix} \ID_{k-1} & 0 \\ 0 & 0 \end{pmatrix} \in \CC^{n \times n},\\
P^0(0,u) & = E_{kk} \in \CC^{n \times n},\\
P^>(0,u) & = \Pi^> := \begin{pmatrix} 0 & 0 \\ 0 & \ID_{n-k} \end{pmatrix} \in \CC^{n \times n}
\end{align*}
where for $j =1,\ldots,n$ we have set
\begin{align*}
E_{jj} := e_j e_j^t.
\end{align*}
Because $F(\varepsilon,u)$ is real symmetric we find
\begin{align*}
P^{\gtrless}, P^0 \in \textnormal{Sym}_n(\RR).
\end{align*}
Following chapter 2.4.~in \cite{K95} on the construction of transformation functions adapted to several (disjoint) projections set
\begin{align*}
Q = (\partial_\varepsilon P^<)P^< + (\partial_\varepsilon P^0)P^0 + (\partial_\varepsilon P^>)P^>.
\end{align*}
Then $Q$ is again smooth in $(\varepsilon,u)$ and real skew-symmetric. For $u \in B_R(0)$ let $W=W(\varepsilon,u)$ solve
\begin{align}
\label{transequaouter}
\partial_\varepsilon W(\varepsilon,u) = Q(\varepsilon,u) W(\varepsilon,u), \; W(0,u) = \ID_n.
\end{align}
As a solution of a linear ordinary differential equation depending smoothly on parameters $W$ is smooth in $(\varepsilon,u)$ and $\det W > 0$. Since $Q$ is real skew-symmetric we find
\begin{align*}
W^{-1} = W^t,
\end{align*}
i.e.~$W \in \textnormal{SO}_n(\RR)$. Furthermore, for all $(\varepsilon,u)$ the relations
\begin{align*}
W(\varepsilon,u) P^{\gtrless}(0,u) W(\varepsilon,u)^t & = P^{\gtrless}(\varepsilon,u)\\
W(\varepsilon,u) P^0(0,u) W(\varepsilon,u)^t & = P^0(\varepsilon,u)
\end{align*}
hold. We find
\begin{align*}
M(\varepsilon,u) := & W(\varepsilon,u)^t F(\varepsilon,u) W(\varepsilon,u)\\
= & \Pi^< M(\varepsilon,u) \Pi^< + E_{kk} M(\varepsilon,u) E_{kk} + \Pi^> M(\varepsilon,u) \Pi^> \\
= & \begin{pmatrix}
M^<(\varepsilon,u) & 0 & 0\\
0 & M^0(\varepsilon,u) & 0\\
0 & 0 & M^>(\varepsilon,u)
\end{pmatrix}.
\end{align*}
The functions $M^{\gtrless},M^0$ are smooth in $(\varepsilon,u)$ with
\begin{align*}
M^<(0,u) = \diag(\lambda_1,\ldots,\lambda_{k-1}), M^0(0,u) = 0, M^>(0,u) = \diag(\lambda_{k+1},\ldots,\lambda_n)
\end{align*}
for all $u$. A direct calculation shows
\begin{align*}
\partial_\varepsilon M^0(0,u) = (\partial_\varepsilon F(0,u))_{kk} = e_k^t \partial_{v_k}D_v f(0) u.
\end{align*}
\end{proof}
Set
\begin{align*}
F_{\varepsilon}(\tau) := Df_{s(\varepsilon)}(\varepsilon u_\varepsilon(\tau)), G_{\varepsilon}(\tau) := Dg_{s(\varepsilon)}(\varepsilon u_\varepsilon(\tau)).
\end{align*}
and $F_0 := F_{0}(0), G_0 := G_{0}(0)$, and for appropriate $\nu_\varepsilon > 0$ let
\begin{align*}
\Lambda_\varepsilon := M_{\nu_\varepsilon}
\end{align*}
be a domain of consistent splitting for $\lbrace (A_{\varepsilon,\kappa}(-1),A_{\varepsilon,\kappa}(+1)) \rbrace_\kappa$ as guaranteed to exist by Corollary \ref{cordispersion}.
\subsection{Inner regime}
Let $r_0 > 0$ and set
\begin{align*}
D_{r_0} := \lbrace \kappa \in \CH^+: \; |\kappa| \leq r_0 \rbrace.
\end{align*}
The scalings
\begin{align*}
\varepsilon^2 \zeta = \kappa \;\; \textnormal{and} \;\; \varepsilon\tilde{q} = q
\end{align*}
yield
\begin{align*}
A^I_{\varepsilon,\zeta}(\tau) = \begin{pmatrix}
F_\varepsilon(\tau) & \varepsilon \ID_d\\ \varepsilon^3 \zeta^2 \ID_d + \varepsilon \zeta G_\varepsilon(\tau) & 0
\end{pmatrix}
\end{align*}
with
\begin{align*}
A^I_{\varepsilon,\zeta}(\tau) \sim A_{s(\varepsilon),\varepsilon^2 \zeta}(\tau) \;\; \textnormal{if} \;\; \varepsilon >0.
\end{align*}
Hence, for $\varepsilon > 0$ we can consider 
\begin{align}
\label{LapHsCLautoin}
\left\lbrace \begin{aligned}
\tau^\prime & = \varepsilon (1-\tau^2)h_\varepsilon(\tau)\\
\xi^\prime & = A^I_{\varepsilon,\zeta}(\tau) \xi
\end{aligned} \right.
\end{align}
on $J \times \CC^{2n}$ instead of \eqref{LapHsCLauto}. First we are interested in the splitting of eigenvalues and eigenspaces of $A^I_{\varepsilon,\zeta}(\tau)$ at the end points $\tau = \pm 1$ for $\zeta \in D_{r_0}$ and small $\varepsilon \geq 0$. Fix some notation by setting
\begin{align*}
F^{\pm}_{\varepsilon} & := F_\varepsilon(\pm 1),\\
G^{\pm}_{\varepsilon} &:= G_\varepsilon(\pm 1),\\
\underline{G}_0 & := (G_0)_{kk}^{-1}G_0,\\
A^{I}_0 & := A^I_{0,0}(0),\\
A^{I,\pm}_{\varepsilon,\zeta} & := A^I_{\varepsilon,\zeta}(\pm 1).
\end{align*}
\begin{lem}
\label{caplemma}
\begin{itemize}
\item[(i)] For any $r_0 > 0$ there exists $\varepsilon_0 = \varepsilon_0(r_0)$ such that for $\varepsilon \in [0,\varepsilon_0]$ there are holomorphic mappings
\begin{align*}
\begin{aligned}
S^{\pm,f}_\varepsilon: D_{r_0} & \rightarrow \GG^{2n}_{k-1}(\CC), && S^{\pm,sf}_\varepsilon: D_{r_0} & \rightarrow \GG^{2n}_{1}(\CC), && S^{\pm,ss}_\varepsilon: D_{r_0} & \rightarrow \GG^{2n}_{n-k}(\CC),\\
U^{\pm,f}_\varepsilon: D_{r_0} & \rightarrow \GG^{2n}_{n-k}(\CC), && U^{\pm,sf}_\varepsilon: D_{r_0} & \rightarrow \GG^{2n}_{1}(\CC), && U^{\pm,ss}_\varepsilon: D_{r_0} & \rightarrow \GG^{2n}_{k-1}(\CC),
\end{aligned}
\end{align*}
and
\begin{align*}
N^\pm_\varepsilon: D_{r_0} \rightarrow \GG^{2n}_{n+1}(\CC)
\end{align*}
depending smoothly on $\varepsilon$ such that for all $\zeta \in D_{r_0}$ their images define a splitting of $\CC^{2n}$ consisting of invariant subspaces of $A^{I,\pm}_{\varepsilon,\zeta}$:\\
For all $\varepsilon \in [0,\varepsilon_0]$ and $\zeta \in D_{r_0}$ it holds
\begin{align*}
\CC^{2n} = S^{\pm,f}_{\varepsilon,\zeta} \oplus N^\pm_{\varepsilon,\zeta} \oplus U^{\pm,f}_{\varepsilon,\zeta}
\end{align*}
and \begin{align*}
N^\pm_{\varepsilon,\zeta} = S^{\pm,sf}_{\varepsilon,\zeta} \oplus S^{\pm,ss}_{\varepsilon,\zeta} \oplus U^{\pm,sf}_{\varepsilon,\zeta}
\oplus U^{\pm,ss}_{\varepsilon,\zeta}.
\end{align*}
Furthermore, for all $\zeta \in D_{r_0}$
\begin{align*}
S^{\pm,f}_{0,\zeta} & = S_0^f := \CC^{k-1} \times \lbrace 0 \rbrace^{n-k+1} \times \lbrace 0 \rbrace^n\\
 U^{\pm,f}_{0,\zeta} & = U_0^f := \lbrace 0 \rbrace^{k} \times  \CC^{n-k} \times \lbrace 0 \rbrace^{n},
\end{align*}
and
\begin{align*}
N^\pm_{0,\zeta} = N_0 := \lbrace 0 \rbrace^{k-1} \times \CC \times \lbrace 0 \rbrace^{n-k} \times \CC^n.
\end{align*}
The super-slow invariant subspaces are constant on $\lbrace \varepsilon = 0 \rbrace$, i.e.
\begin{align*}
S^{\pm,ss}_{0,\zeta} & = S_0^{ss} := S^{+,ss}_{0,0}\\
U^{\pm,ss}_{0,\zeta} & = U_0^{ss} := U^{+,ss}_{0,0}
\end{align*}
\item[(ii)] The eigenvalues
\begin{align*}
\lbrace \mu^{\pm,S,f}_{j;\varepsilon,\zeta}\rbrace_{j<k}, \lbrace\mu^{\pm,U,f}_{j;\varepsilon,\zeta}\rbrace_{j>k}, \mu^{\pm,S,sf}_{\varepsilon,\zeta}, \mu^{\pm,U,sf}_{\varepsilon,\zeta}, \lbrace \mu^{\pm,S,ss}_{j;\varepsilon,\zeta}\rbrace_{j>k}, \lbrace\mu^{\pm,U,ss}_{j;\varepsilon,\zeta}\rbrace_{j<k}
\end{align*}
of $A^{I,\pm}_{\varepsilon,\zeta}$ expand as
\begin{align*}
\left\lbrace\begin{aligned}
&\lambda_j + \OO(\varepsilon) \; & \textnormal{for} \; j < k\\
&\lambda_j + \OO(\varepsilon) \; & \textnormal{for} \; j > k\\
&\varepsilon(\tilde{\mu}^{S,sf}_\zeta + \OO(\varepsilon))&\\
&- \varepsilon \zeta(G_0)_{kk}(\tilde{\mu}^{S,sf-1}_\zeta+ \OO(\varepsilon))&\\
&\varepsilon \zeta (G_0)_{kk}(\tilde{\mu}^{S,sf-1}_\zeta + \OO(\varepsilon))&\\
&- \varepsilon(\tilde{\mu}^{S,sf}_\zeta + \OO(\varepsilon))&\\
&-\varepsilon^2 \zeta(\mu_j^{-1}/2 +\OO(\varepsilon)) \; & \textnormal{for} \; j > k\\
& -\varepsilon^2 \zeta(\mu_j^{-1}/2 + \OO(\varepsilon)) \; & \textnormal{for} \; j < k
\end{aligned}\right\rbrace \textnormal{on} \left\lbrace\begin{aligned}
S^{\pm,f}_{\varepsilon,\zeta}\\
U^{\pm,f}_{\varepsilon,\zeta}\\
S^{+,sf}_{\varepsilon,\zeta}\\
U^{+,sf}_{\varepsilon,\zeta}\\
S^{-,sf}_{\varepsilon,\zeta}\\
U^{-,sf}_{\varepsilon,\zeta}\\
S^{\pm,ss}_{\varepsilon,\zeta}\\
U^{\pm,ss}_{\varepsilon,\zeta}
\end{aligned}\right\rbrace
\end{align*}
where
\begin{align*}
\tilde{\mu}^{U/S,sf}_\zeta :=  a \pm \sqrt{a^2 + \zeta(G_0)_{kk}}
\end{align*}
satisfy
\begin{align*}
- \zeta (G_0)_{kk} = \tilde{\mu}^{S,sf}_\zeta \tilde{\mu}^{U,sf}_\zeta, \textnormal{Re}\,\tilde{\mu}^{U,sf}_\zeta > 0 \;\; \textnormal{and} \;\; \textnormal{Re}\,\tilde{\mu}^{S,sf}_\zeta \leq a <0 \;\;\textnormal{for} \;\; \zeta \in \CH^+ \backslash \lbrace 0 \rbrace.
\end{align*}
\item[(iii)] There is a smooth choice of bases
\begin{align*}
\lbrace R^{\pm,S,f}_{j;\varepsilon,\zeta} \rbrace_{j< k}, \lbrace R^{\pm,U,f}_{j;\varepsilon,\zeta} \rbrace_{j< k}, R^{\pm,S,sf}_{\varepsilon,\zeta}, R^{\pm,U,sf}_{\varepsilon,\zeta}, \lbrace R^{\pm,S,ss}_{j;\varepsilon,\zeta} \rbrace_{j> k}, \lbrace R^{\pm,U,ss}_{j;\varepsilon,\zeta} \rbrace_{j< k}
\end{align*}
for the above invariant subspaces of $A^{I,\pm}_{\varepsilon,\zeta}$ such that for $\varepsilon = 0$ the fast subspaces are spanned by
\begin{align*}
R^{\pm,S,f}_{j;0,\zeta} & = \tilde{e}_j = \begin{pmatrix} e_j\\ 0 \end{pmatrix} \;\; \textnormal{for} \;\; j< k,\\
R^{\pm,U,f}_{j;0,\zeta} & = \tilde{e}_j = \begin{pmatrix} e_j\\ 0 \end{pmatrix} \;\; \textnormal{for} \;\; j> k,
\end{align*}
the slow-fast subspaces by
\begin{align*}
R^{+,S,sf}_{0,\zeta} & = \begin{pmatrix}e_k \\ -\tilde{\mu}^{U,sf}_\zeta \underline{G}_0 e_k \end{pmatrix},\\
R^{+,U,sf}_{0,\zeta} & = \begin{pmatrix}e_k \\ -\tilde{\mu}^{S,sf}_\zeta \underline{G}_0 e_k \end{pmatrix},\\
R^{-,S,sf}_{0,\zeta} & = \begin{pmatrix}e_k \\ \tilde{\mu}^{S,sf} \underline{G}_0 e_k \end{pmatrix}, \\
R^{-,U,sf}_{0,\zeta} & = \begin{pmatrix}e_k \\ \tilde{\mu}^{U,sf} \underline{G}_0 e_k \end{pmatrix},
\end{align*}
and the super-slow subspaces by
\begin{align*}
R^{\pm,S,ss}_{j;0,\zeta} & = \begin{pmatrix} 0 \\ \underline{G}_0r_j\end{pmatrix} \;\;\textnormal{for} \;\; j > k,\\
R^{\pm,U,ss}_{j;0,\zeta} & = \begin{pmatrix} 0 \\ \underline{G}_0r_j\end{pmatrix} \;\; \textnormal{for} \;\; j < k.
\end{align*}
\end{itemize}
\end{lem}
\begin{proof}
Let $r_0 > 0$. We only consider the right end $\tau = +1$, i.e.~only the $+$-superscript, since the arguments in the case $\tau = -1$ are the same modulo obvious modifications. Let $\varepsilon_0>0$ be so small that Lemma \ref{lemdiagf} holds, and in its notation we may assume
\begin{align*}
F_{\varepsilon} = \begin{pmatrix}
M^<_\varepsilon(+1) & 0 & 0\\ 0 & M^0_\varepsilon(+1) & 0 \\ 0 & 0 & M^>_\varepsilon(+1)
\end{pmatrix}
\end{align*}
for
\begin{align*}
M^<_\varepsilon(\tau) & = M^<(\varepsilon,\varepsilon u_\varepsilon(\tau)) \\
M^0_\varepsilon(\tau) & = M^0(\varepsilon,\varepsilon u_\varepsilon(\tau))\\
M^>_\varepsilon(\tau) & =M^>(\varepsilon,\varepsilon u_\varepsilon(\tau)).
\end{align*}
For $\varepsilon = 0$ we find that
\begin{align*}
A^I_{0,\zeta} = A^I_0 = \begin{pmatrix}
\diag(\lambda_1,\ldots,\lambda_n) & 0\\ 0 & 0
\end{pmatrix}
\end{align*}
is independent of $\zeta$ and of diagonal form. Hence for small $\varepsilon$ the total projections
\begin{align*}
P^{S,f}_{\varepsilon,\zeta}, P^s_{\varepsilon,\zeta}\; \textnormal{and}\; P^{U,f}_{\varepsilon,\zeta}
\end{align*}
onto the eigenvalues with real part (strongly) smaller, close-to and (strongly) larger than zero are well-defined and smooth in their parameters. For $\varepsilon = 0$ they have constant values
\begin{align}
\label{PSf0}
P^{S,f}_0:= P^{S,f}_{0,\zeta} & = \begin{pmatrix}
\Pi^< & 0\\0 & 0
\end{pmatrix}\\
\nonumber P_0^s:=P^s_{0,\zeta} & = \begin{pmatrix}
E_{kk} & 0\\0 & \ID_n
\end{pmatrix}\\
\label{PUf0}
P^{U,f}_0:=P^{U,f}_{0,\zeta}& = \begin{pmatrix}
\Pi^>& 0\\0 & 0
\end{pmatrix}.
\end{align}
We apply Kato's reduction process \cite{K95} to the $0$-group. A Taylor expansion shows that
\begin{align*}
A^{I,s}_{\varepsilon,\zeta} := \frac{1}{\varepsilon} A^I_{\varepsilon,\zeta} P^s_{\varepsilon,\zeta}
\end{align*}
defines a smooth function with
\begin{align*}
A^{I,s}_{0,\zeta} = P^s_0 \partial_\varepsilon A^I_{0,\zeta} P^s_0.
\end{align*}
From Lemma \ref{lemdiagf} and $u_\varepsilon(\tau) = \varepsilon \tau e_k + \mathcal{O}(\varepsilon^2)$ one finds $\partial_\varepsilon M^0_{0} = 2 a $, and hence
\begin{align*}
A^{I,s}_{0,\zeta} = \begin{pmatrix}
2  a E_{kk} & E_{kk} \\ \zeta G_0 E_{kk} & 0
\end{pmatrix}.
\end{align*}
For $\zeta \in D_{r_0}$ the matrix $A^{I,s}_{0,\zeta}$ has two slow-fast eigenvalues given by
\begin{align}
\label{slowfastEV}
\tilde{\mu}^{U/S,sf}_\zeta =  a \pm \sqrt{a^2 + \zeta(G_0)_{kk}}
\end{align}
and an $2(n-1)$-fold zero eigenvalue. It holds for all $\zeta \in D_{r_0}$
\begin{align*}
\RE \tilde{\mu}^{S,sf}_\zeta < a < 0 \leq \RE \tilde{\mu}^{S,sf}_\zeta.
\end{align*}
Hence for $\varepsilon$ small and all $\zeta \in D_{r_0}$ there is a spectral gap of size $a/2$ separating a stable simple eigenvalue $\tilde{\mu}^{S,sf}_{\varepsilon,\zeta}$ from all other eigenvalues of $A^{I,s}_{\varepsilon,\zeta}$, and we find a stable slow-fast eigenvalue $\mu_{\varepsilon,\zeta}^{S,sf}$ of $A^I_{\varepsilon,\zeta}$ with 
\begin{align*}
\mu_{\varepsilon,\zeta}^{S,sf} = \varepsilon (\tilde{\mu}_{\zeta}^{S,sf} + \mathcal{O}(\varepsilon))
\end{align*}
and one dimensional eigenprojection $P^{S,sf}_{\varepsilon,\zeta}$ which is smooth in $\varepsilon$ and holomorphic in $\zeta$.\\
To get a hold of the slow-fast eigenvalue associated to $\tilde{\mu}^{U,sf}_{\zeta}$ let us consider
\begin{align*}
A^{I,U,sf}_{\varepsilon,\zeta} = A^{I,s}_{\varepsilon,\zeta}(\ID_{2n}-P^{S,sf}_{\varepsilon,\zeta}).
\end{align*}
Using the Dunford formula we find at $\zeta = 0$
\begin{align*}
P^s_{\varepsilon,0} = \begin{pmatrix}
E_{kk} & -\varepsilon \tilde{F}^{-1}_{\varepsilon}\\
0 & \ID_n
\end{pmatrix}
\end{align*}
where
\begin{align*}
\tilde{F}^{-1}_{\varepsilon} = \begin{pmatrix}
M^<_\varepsilon(+1)^{-1} & 0 & 0 \\ 0 & 0 & 0\\ 0 & 0 & M^>_\varepsilon(+1)^{-1}
\end{pmatrix},
\end{align*}
which leads to
\begin{align*}
A^{I,s}_{\varepsilon,0} = \varepsilon^{-1} A^I_{\varepsilon,0} P^s_{\varepsilon,0} = \begin{pmatrix}
\frac{M_\varepsilon^0(+1)}{\varepsilon} E_{kk} & E_{kk}\\
0 & 0
\end{pmatrix}
\end{align*}
and
\begin{align*}
P^{S,sf}_{\varepsilon,0} = \begin{pmatrix}
E_{kk} & \frac{\varepsilon}{M_\varepsilon^0(+1)} E_{kk} \\ 0 & 0
\end{pmatrix}.
\end{align*}
We obtain for all small $\varepsilon$
\begin{align}
\label{slowfastunstable1}
A^{I,U,sf}_{\varepsilon,0} = A^{I,s}_{\varepsilon,0}(\ID_{2n}-P^{S,sf}_{\varepsilon,0}) = 0.
\end{align}
Set $\Delta^{sf}_\zeta = \tilde{\mu}^{S,sf}_\zeta - \tilde{\mu}^{U,sf}_\zeta$ and
\begin{align*}
\tilde{R} := \begin{pmatrix}
(r_j)_{j<k},0,(r_j)_{j>k}
\end{pmatrix} \in \RR^{n \times n}
\end{align*}
with $r_j$ from Lemma \ref{assupmlem1}. By the just quoted lemma and Lemma \ref{assumplem2} the columns of
\begin{align*}
\tilde{R} + (G_0)_{kk}^{-1}E_{kk}
\end{align*}
are eigenvectors of $F_0$ w.r.t.~$G_0$, and
\begin{align*}
(\tilde{R}^t+E_{kk})G_0(\tilde{R} + (G_0)_{kk}^{-1}E_{kk}) = \ID_n.
\end{align*}
Diagonalize $A^{I,s}_{0,\zeta}$ via
\begin{align}
\label{diagR+}
R_\zeta & := \begin{pmatrix}
\ID_n & E_{kk}\\
- \frac{\tilde{\mu}^{S,sf}_\zeta}{(G_0)_{kk}} G_0 E_{kk} & G_0\left(\tilde{R} -\frac{\tilde{\mu}^{U,sf}_\zeta}{(G_0)_{kk}} E_{kk}\right)
\end{pmatrix},\\
\nonumber R_\zeta^{-1} & := \begin{pmatrix}
\ID_n - \frac{\tilde{\mu}^{S,sf}_\zeta}{\Delta^{sf}_\zeta} E_{kk} & - \frac{1}{\Delta^{sf}_\zeta} E_{kk}\\ \frac{\tilde{\mu}^{S,sf}_\zeta}{\Delta^{sf}_\zeta} E_{kk} & \tilde{R}^t + \frac{1}{\Delta^{sf}_\zeta} E_{kk}
\end{pmatrix},
\end{align}
to find
\begin{align*}
R_\zeta^{-1}A^{I,s}_{0,\zeta}R_\zeta = \begin{pmatrix}
\tilde{\mu}^{U,sf}_\zeta E_{kk} & 0 \\ 0 & \tilde{\mu}^{S,sf}_\zeta E_{kk}
\end{pmatrix}.
\end{align*}
Hence at $\varepsilon = 0$ the slow-fast stable projections is given by
\begin{align}
\label{PSsf}
P^{S,sf}_{0,\zeta} = R_{\zeta} \begin{pmatrix}
0 & 0\\0 & E_{kk}
\end{pmatrix}R_{\zeta}^{-1}.
\end{align}
A calculation shows
\begin{align}
\label{slowfastunstable2}
A^{I,U,sf}_{0,\zeta} = \tilde{\mu}^{U,sf}_\zeta P^{U,sf}_{\zeta} = - \zeta \frac{(G_0)_{kk}}{\tilde{\mu}^{S,sf}_\zeta} P^{U,sf}_{\zeta}
\end{align}
where
\begin{align}
\label{PUsf}
P^{U,sf}_{\zeta} = R_{\zeta} \begin{pmatrix}
E_{kk} & 0\\0 & 0
\end{pmatrix}R_{\zeta}^{-1}.
\end{align}
By combining \eqref{slowfastunstable1} and \eqref{slowfastunstable2} there exists a $\varepsilon$-smooth and $\zeta$-holomorphic function $H = H(\varepsilon,\zeta)$ such that for $\varepsilon$ small and $\zeta \in D_{r_0}$
\begin{align*}
A^{I,U,sf}_{\varepsilon,\zeta} = \tilde{\mu}^{U,sf}_\zeta P^{U,sf}_{\zeta} + \varepsilon \zeta H(\varepsilon,\zeta).
\end{align*}
We find an unstable slow-fast eigenvalue $\mu^{U,sf}_{\varepsilon,\zeta}$ of $A^I_{\varepsilon,\zeta}$ with $\varepsilon$-smooth and $\zeta$-holomorphic one dimensional eigenprojection $P^{U,sf}_{\varepsilon,\zeta}$ and
\begin{align*}
\mu^{U,sf}_{\varepsilon,\zeta} & = - \varepsilon \zeta\left(\frac{(G_0)_{kk}}{\tilde{\mu}^{S,sf}_\zeta} + \mathcal{O}(\varepsilon)\right),\\
P^{U,sf}_{0,\zeta} & = P^{U,sf}_{\zeta}.
\end{align*}
For the last reduction step consider
\begin{align*}
A^{I,ss}_{\varepsilon,\zeta} := \frac{1}{\varepsilon \zeta}A^{I,U,sf}_{\varepsilon,\zeta}(\ID_{2n} - P^{U,sf}_{\varepsilon,\zeta}) = \frac{1}{\varepsilon \zeta}A^{I,s}_{\varepsilon,\zeta}(\ID_{2n} - P^{sf}_{\varepsilon,\zeta})
\end{align*}
where $P^{sf}= P^{S,sf} + P^{U,sf}$. Setting
\begin{align*}
\tilde{F}_0^{-1} := \begin{pmatrix}
(M^<_0)^{-1} & 0 & 0\\ 0 & 0 & 0 \\ 0 & 0 & (M^>_0)^{-1}
\end{pmatrix}
\end{align*}
and noting
\begin{align*}
\partial_\varepsilon A^I_{0,\zeta} & = \begin{pmatrix}
\partial_\varepsilon F_0(1) & \ID_n\\ \zeta G_0 & 0
\end{pmatrix}\\
\partial_\varepsilon^2 A^I_{0,\zeta} & = \begin{pmatrix}
\partial_\varepsilon^2 F_0(1) & 0\\ \zeta \partial_\varepsilon G_0(1) & 0
\end{pmatrix}\\
\partial_\varepsilon P_{0;0,\zeta} & = \begin{pmatrix}
0 & - \tilde{F}_0^{-1}\\  - \zeta G_0 \tilde{F}^{-1}_0 & 0
\end{pmatrix}\\
P^{sf}_{0,\zeta} & = \begin{pmatrix}
E_{kk} & 0 \\ 0 & \underline{G}_0 E_{kk}
\end{pmatrix}
\end{align*}
one finds with $p^{ss} := \ID_n - \underline{G}_0 E_{kk}$
\begin{align*}
A^{I,ss}_{0,\zeta} = \frac{1}{2}\begin{pmatrix}
0 & 0\\0 & - p^{ss} G_0 \tilde{F}^{-1}_0 p^{ss}
\end{pmatrix}.
\end{align*}
Because
\begin{align*}
(\tilde{R}^t+E_{kk}) p^{ss} G_0 \tilde{F}^{-1}_0 p^{ss} G_0(\tilde{R} + &(G_0)_{kk}^{-1} E_{kk})\\
& = \diag(\mu_1^{-1},\ldots,\mu_{k-1}^{-1},0,\mu_{k+1}^{-1},\ldots,\mu_n^{-1})
\end{align*}
we find one dimensional eigenprojections $P^{S,ss}_{j;\varepsilon,\zeta},j>k,$ and $P^{U,ss}_{j;\varepsilon,\zeta},j<k,$ associated to $n-k$ stable eigenvalues $\mu^{S,ss}_{j;\varepsilon,\zeta},j>k,$ and $k-1$ unstable eigenvalues $\mu^{U,ss}_{j;\varepsilon,\zeta},j<k,$ of $A^I_{\varepsilon,\zeta}$ satisfying
\begin{align}
\nonumber
\mu^{S/U,ss}_{j;\varepsilon,\zeta} & = - \varepsilon^2 \zeta\left( \frac{\mu^{-1}_j}{2} + \mathcal{O}(\varepsilon) \right)\\
\label{PSUss}
P^{S/U,ss}_{j;0,\zeta} & = \begin{pmatrix}
0 & 0\\0 & G_0 \tilde{R} E_{jj} \tilde{R}^t
\end{pmatrix}.
\end{align}
For $\varepsilon$ small enough and all $\zeta \in D_{r_0}$ the vectors
\begin{align*}
R^{S,f}_{j;\varepsilon,\zeta} & = P^{S,f}_{\varepsilon,\zeta}\begin{pmatrix} e_j\\ 0 \end{pmatrix} & \textnormal{for} \;\; j< k,\\
R^{U,f}_{j;\varepsilon,\zeta} & = P^{U,f}_{\varepsilon,\zeta}\begin{pmatrix} e_j\\ 0 \end{pmatrix} & \textnormal{for} \;\; j> k,\\
R^{S,sf}_{\varepsilon,\zeta} & = P^{S,sf}_{\varepsilon,\zeta}\begin{pmatrix}e_k \\ -\tilde{\mu}^{U,sf}_{\zeta} \underline{G}_0 e_k \end{pmatrix}, &\\
R^{U,sf}_{\varepsilon,\zeta} & = P^{U,sf}_{\varepsilon,\zeta}\begin{pmatrix}e_k \\ -\tilde{\mu}^{S,sf}_{\zeta} \underline{G}_0 e_k \end{pmatrix},& \\
R^{S,ss}_{j;\varepsilon,\zeta} & = P^{S,ss}_{j;\varepsilon,\zeta}\begin{pmatrix} 0 \\ \underline{G}_0r_j\end{pmatrix}& \textnormal{for} \;\; j > k,\\
R^{U,ss}_{j;\varepsilon,\zeta} & = P^{U,ss}_{j;\varepsilon,\zeta}\begin{pmatrix} 0 \\ \underline{G}_0r_j\end{pmatrix}& \textnormal{for} \;\; j < k.
\end{align*}
define $\varepsilon$-smooth and $\zeta$-holomorphic choices of bases for the invariant subspaces of $\tilde{A}_{\varepsilon,\zeta}$ which are given by the images of the respective (total) projections. Using the explicit expressions \eqref{PSf0}, \eqref{PUf0}, \eqref{PSsf}, \eqref{PUsf}, \eqref{PSUss} at $\varepsilon = 0$ finishes the proof of the lemma.
\end{proof}
\begin{rem}
\label{remsplittinginner}
By Corollary \ref{cordispersion} we have
\begin{align*}
\CH^+\backslash \lbrace 0 \rbrace \subset \Lambda_\varepsilon.
\end{align*}
Hence, in the situation of Lemma \ref{caplemma} the real parts of $A^{I,\pm}_{\varepsilon,\zeta}$ may not change signs as $(\varepsilon,\zeta)$ is varied in $(0,\varepsilon_0] \times (D_{r_0} \backslash \lbrace 0 \rbrace)$, and the eigenvalues with superscript $S$ resp.~$U$ indeed have real parts smaller resp.~greater than $0$ if $\varepsilon,\zeta \neq 0$.
\end{rem}
With the help of Lemma \ref{caplemma} and singular perturbation theory we give a detailed construction of the (rescaled) Evans bundles for $\phi_\varepsilon$ following \cite{FS02}. Define the one dimensional submanifolds
\begin{align*}
\CS^{S,f}_0 & = J \times \lbrace S^f_0 \rbrace \subset J \times \GG^{2n}_{k-1}(\CC)\\
\CS^{U,f}_0 & = J \times \lbrace U^f_0 \rbrace \subset J \times \GG^{2n}_{n-k}(\CC)\\
\NH_0 & = J \times \lbrace N_0 \rbrace \subset J \times \GG^{2n}_{n+1}(\CC).
\end{align*}
The statement for the fast part reads:
\begin{lem}
\label{leminnerfast}
For all $r_0 > 0$ there exists $\varepsilon_0 = \varepsilon_0(r_0)$ such that for all $\varepsilon \in [0,\varepsilon_0]$ there are holomorphic mappings
\begin{align*}
H^{+,f}_\varepsilon: D_{r_0} & \rightarrow \GG^{2n}_{k-1}(\CC)\\
H^{-,f}_\varepsilon: D_{r_0} & \rightarrow \GG^{2n}_{n-k}(\CC)
\end{align*}
depending smoothly on $\varepsilon$ and the solutions $X^+$ and $X^-$ of
\begin{align*}
\begin{aligned}
{X^+}^\prime & = \Gamma^{k-1} A^I_{\varepsilon,\zeta}(\chi_\varepsilon)(X^+), && X^+(0) = H^{+,f}_\varepsilon(\zeta), \;\; \textnormal{on} \;\; \GG^{2n}_{k-1}(\CC),\\
{X^-}^\prime & = \Gamma^{n-k} A^I_{\varepsilon,\zeta}(\chi_\varepsilon)(X^-), && X^-(0) = H^{-,f}_\varepsilon(\zeta),  \;\; \textnormal{on} \;\; \GG^{2n}_{n-k}(\CC),
\end{aligned}
\end{align*}
converge on the left and right ends:
\begin{align*}
X^+(\pm \infty) & = S^{\pm,f}_{\varepsilon,\zeta}\\
X^-(\pm \infty) & = U^{\pm,f}_{\varepsilon,\zeta}.
\end{align*}
\end{lem}
\begin{proof}
Let $r_0 > 0$ and choose $\varepsilon_0$ so small such that the conclusions of Lemma \ref{caplemma} hold. Since
\begin{align*}
A^I_{0,\zeta}(\tau) = A^I_0 := \begin{pmatrix}
\diag(\lambda_1,\ldots,\lambda_n) & 0 \\ 0 & 0
\end{pmatrix}
\end{align*}
we find by Corollary \ref{corhyperattractor} that $\CS^{S,f}_0$ is a repelling normally hyperbolic critical manifold for the system
\begin{align}
\label{Grassfaststable}
\left\lbrace \begin{aligned}
\tau^\prime & = \varepsilon (1-\tau^2)h_\varepsilon(\tau)\\
X^\prime & = \Gamma^{k-1}(A^I_{\varepsilon,\zeta}(\tau))X
\end{aligned} \right. \;\; \textnormal{on} \;\; J \times \GG^{2n}_{k-1}(\CC).
\end{align}
By Fenichel theory \cite{F79} and compactness of $J \times \GG^{2n}_{k-1}(\CC)$ it perturbs smoothly to a unique repelling slow manifold $\CS^{S,f}_{\varepsilon,\zeta}$. Because $\CS^{S,f}_0$ is the graph of
\begin{align*}
C^{S,f}_0(\tau) := S^f_0
\end{align*}
we deduce from the presentation of geometric singular perturbation theory in \cite{J95} that for $\varepsilon$ small enough there exists an $\varepsilon$-smooth and $\zeta$-holomorphic function
\begin{align*}
C^{S,f}_{\varepsilon,\zeta}: J \rightarrow \GG^{2n}_{k-1}(\CC)
\end{align*}
such that $C^{S,f}_{0,\zeta}=C^{S,f}_0$ and $\CS^{S,f}_{\varepsilon,\zeta}$ is a graph of $C^{S,f}_{\varepsilon,\zeta}$ for all $\zeta \in D_{r_0}$. Since $(\pm 1,S^{\pm,f}_{\varepsilon,\zeta})$ are fixed points of \eqref{Grassfaststable} and $S^{\pm,f}_{\varepsilon,\zeta}$ is $\varepsilon$-close to $S^f_0$ we find 
\begin{align*}
(\pm 1,S^{\pm,f}_{\varepsilon,\zeta}) \in \CS^{S,f}_{\varepsilon,\zeta}\;\; \textnormal{and}\;\; C^{S,f}_{\varepsilon,\zeta}(\pm 1) = S^{\pm,f}_{\varepsilon,\zeta}.
\end{align*}
In a similar fashion one finds a unique attracting slow manifold $\CS^{U,f}_{\varepsilon,\zeta}$ of
\begin{align}
\label{Grassfastunstable}
\left\lbrace \begin{aligned}
\tau^\prime & = \varepsilon (1-\tau^2)h_\varepsilon(\tau)\\
X^\prime & = \Gamma^{n-k}(A^I_{\varepsilon,\zeta}(\tau))X
\end{aligned} \right. \;\; \textnormal{on} \;\; J \times \GG^{2n}_{n-k}(\CC)
\end{align}
which perturbs from $\CS^{U,f}_0$ and is given as the graph of a smooth function $C^{U,f}_{\varepsilon,\zeta}$ with $C^{U,f}_{0,\zeta}(\tau) = U^f_0$ and
\begin{align*}
C^{U,f}_{\varepsilon,\zeta}(\pm 1) = U^{\pm,f}_{\varepsilon,\zeta}.
\end{align*}
With an appropriately small choice of $\varepsilon_0>0$ define for all $(\varepsilon,\zeta) \in [0,\varepsilon_0] \times D_{r_0}$ the fast bundles by setting
\begin{align*}
H^{-,f}_{\varepsilon}(\zeta) := C^{U,f}_{\varepsilon,\zeta}(0) \;\; \textnormal{and} \;\; H^{+,f}_{\varepsilon}(\zeta) := C^{S,f}_{\varepsilon,\zeta}(0).
\end{align*}
From their definitions the bundles $H^{-,f}_\varepsilon$ and $H^{+,f}_\varepsilon$ inherit the desired smoothness properties form $C^{U,f}_{\varepsilon,\cdot}$ and $C^{S,f}_{\varepsilon,\cdot}$. Because $H^{-,f}_{\varepsilon}(\zeta)$ resp.~$H^{+,f}_{\varepsilon}(\zeta)$ are $\varepsilon$-close to $U^f_0$ resp.~$S^f_0$ transversality of the fast bundles follows from transversality of $S^f_0$ and $U^f_0$.
\end{proof}
Let $r_0>0$ and choose $\varepsilon_0>0$ so small that the conclusions of Lemma \ref{caplemma} hold. We turn to the slow part of the Evans bundles. By Corollary \ref{corGrasssaddle} $N_0$ is a hyperbolic saddle for the constant coefficient equation
\begin{align*}
X^\prime = \Gamma^{n+1}(A^I_0)X \;\; \textnormal{on} \;\; \GG^{2n}_{n+1}(\CC).
\end{align*}
Hence, the manifold
\begin{align*}
\NH_0 = J \times \lbrace N_0 \rbrace \subset J \times \GG^{2n}_{n+1}(\CC)
\end{align*}
is a normally hyperbolic critical manifold of the system
\begin{align}
\label{Grassslow}
\left\lbrace \begin{aligned}
\tau^\prime & = \varepsilon (1-\tau^2)h_\varepsilon(\tau)\\
X^\prime & = \Gamma^{n+1}(A^I_{\varepsilon,\zeta}(\tau))X
\end{aligned} \right. \;\; \textnormal{on} \;\; J \times \GG^{2n}_{n+1}(\CC).
\end{align}
Again by Fenichel theory and compactness of $J \times \GG^{2n}_{n+1}(\CC)$ the critical manifold $\NH_0$ perturbs to a unique slow manifold $\NH_{\varepsilon,\zeta}$ depending smoothly on $\varepsilon$ and holomorphically on $\zeta$. It is the graph of a smooth function
\begin{align*}
N_{\varepsilon,\zeta}: J \rightarrow \GG^{2n}_{n+1}(\CC).
\end{align*}
Set $\hat{p}:=(p_1,\ldots,p_{k-1},p_{k+1},\ldots,p_n)^t\in \CC^{n-1}$ for $p \in \CC^n$. By choosing an appropriate canonical chart for $\GG^{2n}_{n+1}(\CC)$ centered at $N_0$ there is a function
\begin{align*}
\tilde{N}_{\varepsilon,\zeta}: J \rightarrow \CC^{(n-1) \times (n+1)}
\end{align*}
depending smoothly on $\varepsilon$ and holomorphically on $\zeta$ such that
\begin{align*}
N_{\varepsilon,\zeta}(\tau) = \left\lbrace \begin{pmatrix} p\\ q \end{pmatrix} \in \CC^{2n}\left\vert\; \hat{p}\right. = \tilde{N}_{\varepsilon,\zeta}(\tau) \begin{pmatrix}
p_k\\q
\end{pmatrix}\right\rbrace
\end{align*}
and the slow equations associated to \eqref{Grassslow} on $\NH_{\varepsilon,\zeta}$ read in linear $(p_k,q)$-coordinates
\begin{align*}
\dot{\tau} & = (1-\tau^2) h_\varepsilon(\tau)\\
\dot{p}_k & = 2 a \tau p_k + q_k + \OO(\varepsilon) \begin{pmatrix}
p_k\\q
\end{pmatrix}\\
\dot{q}_j & = \zeta (G_0)_{jk}p_k + \OO(\varepsilon) \begin{pmatrix}
p_k\\q
\end{pmatrix}
\end{align*}
with $p_j,j \neq k,$ being the fast variables and $p_k, q_1,\ldots,q_n$ being the slow variables.
By the eigenvalue expansions in Lemma \ref{caplemma} we obtain the analogue of Lemma 4 (ii) and Corollary 3 in \cite{FS02}:
\begin{lem}
\label{leminnnerslow}
For all $r_0 > 0$ there exists $\varepsilon_0 = \varepsilon_0(r_0)$ such that the following two statements hold:\\
(i) There are unique holomorphic bundles
\begin{align*}
H^{+,s}_\varepsilon: D_{r_0} & \rightarrow \GG^{2n}_{n-k+1}(\CC)\\
H^{-,s}_\varepsilon: D_{r_0} & \rightarrow \GG^{2n}_{k}(\CC)
\end{align*}
depending smoothly on $\varepsilon$ such that for all $\varepsilon > 0$ and $\zeta \in D_{r_0}$ the solutions $Y^+$ and $Y^-$ of
\begin{align*}
\begin{aligned}
{Y^+}^\prime & = \Gamma^{n-k+1}(A^I_{\varepsilon,\zeta}(\chi_\varepsilon))(Y^+), && Y^+(0) = H^{+,s}_\varepsilon(\zeta), \;\; \textnormal{on} \;\; \GG^{2n}_{n-k+1}(\CC),\\
{Y^-}^\prime & = \Gamma^{k}(A^I_{\varepsilon,\zeta}(\chi_\varepsilon))(Y^-), && Y^-(0) = H^{-,s}_\varepsilon(\zeta),  \;\; \textnormal{on} \;\; \GG^{2n}_{k}(\CC),
\end{aligned}
\end{align*}
converge on the right resp.~left end:
\begin{align*}
Y^+(+ \infty) & = S^{+,s}_{\varepsilon,\zeta}\\
Y^-(- \infty) & = U^{-,s}_{\varepsilon,\zeta}.
\end{align*}
(ii) Restrict \eqref{LapHsCLautoin} on the slow manifold $\mathcal{N}_{\varepsilon,\zeta}$, i.e.~
\begin{align*}
\dot{\tau}= (1-\tau^2) h_{\varepsilon}(\tau), \;\; \dot{\xi} = \varepsilon^{-1} \tilde{A}_{\varepsilon,\zeta}(\tau) \xi
\end{align*}
on $\mathcal{N}_{\varepsilon,\zeta}$, and consider the $1$-dimensional Grassmannian version of the restricted flow:
\begin{align*}
\dot{Z} = \varepsilon^{-1} \Gamma^1 A^I_{\varepsilon,\zeta}(\chi_\varepsilon) (Z), \;\; Z(0) \in \lbrace V \in \GG^{2n}_1(\CC): \;\; V \subset N_{\varepsilon,\zeta}(0) \rbrace.
\end{align*}
Then there are unique holomorphic bundles
\begin{align*}
H^{\pm,sf}_{\varepsilon}: D_{r_0} \rightarrow \GG^{2n}_1(\CC)
\end{align*}
depending smoothly on $\varepsilon \in [0,\varepsilon_0]$ such that the solutions $Z^{\pm}$ of
\begin{align*}
\dot{Z}^{\pm} = \varepsilon^{-1} \Gamma^1 A^I_{\varepsilon,\zeta}(\chi_\varepsilon) (Z^\pm), \;\; Z^\pm(0) = H^{\pm,sf}_{\varepsilon,\zeta}
\end{align*}
satisfy
\begin{align*}
Z^+(+\infty) = S^{+,sf}_{\varepsilon,\zeta} \;\; \textnormal{and} \;\; Z^-(-\infty) = U^{-,sf}_{\varepsilon,\zeta}.
\end{align*}
Furthermore, at $\varepsilon = 0$ we have
\begin{align*}
H^{+,s}_0(\zeta) = H^{+,sf}_0(\zeta) \oplus S^{ss} \;\; \textnormal{and} \;\; H^{-,s}_0(\zeta) = H^{-,sf}_0(\zeta) \oplus U^{ss}.
\end{align*}
\end{lem}
\begin{proof}
Let $r_0>0$, and choose $\varepsilon_0>0$ so small such that the conclusions of Lemma \ref{caplemma} hold and the slow manifold $\NH_{\varepsilon,\zeta}$ exists for all $(\varepsilon,\zeta) \in [0,\varepsilon_0] \times D_{r_0}$ depending smoothly on its parameters. We shall only treat the right end $\tau = +1$, the statement on the bundles with asymptotics at $\tau = -1$ is proved in a similar fashion with obvious modifications.\\
It is convenient to change basis such that $N^+_{\varepsilon,\zeta}$ is constant and spanned by the vectors
\begin{align*}
\begin{pmatrix} e_k \\0 \end{pmatrix}, \begin{pmatrix} 0 \\ e_1 \end{pmatrix}, \ldots, \begin{pmatrix} 0 \\ e_n \end{pmatrix},
\end{align*}
and such that the (linear) slow $(p_k,q)$-flow is given by a diagonal matrix at $\tau = +1$. To this end let $\DS_{\varepsilon,\zeta}^+$ be the $\varepsilon$-smooth and $\zeta$-holomorphic solution of
\begin{align*}
\partial_\varepsilon \DS_{\varepsilon,\zeta}^+ = Q^+_{\varepsilon,\zeta} \DS_{\varepsilon,\zeta}^+, \;\; \DS_{0,\zeta}^+ = \ID_{2n}
\end{align*}
where
\begin{align*}
Q^+_{\varepsilon,\zeta} := & (\partial_{\varepsilon} P^{+,S,f}_{\varepsilon,\zeta}) P^{+,S,f}_{\varepsilon,\zeta} +(\partial_{\varepsilon} P^{+,U,f}_{\varepsilon,\zeta}) P^{+,U,f}_{\varepsilon,\zeta}\\
& +(\partial_{\varepsilon} P^{+,S,sf}_{\varepsilon,\zeta}) P^{+,S,sf}_{\varepsilon,\zeta} + (\partial_{\varepsilon} P^{+,U,sf}_{\varepsilon,\zeta}) P^{+,U,sf}_{\varepsilon,\zeta}\\
& + \sum_{j>k}(\partial_{\varepsilon} P^{+,S,ss}_{j;\varepsilon,\zeta}) P^{+,S,ss}_{\varepsilon,\zeta} + \sum_{j<k}(\partial_{\varepsilon} P^{+,U,ss}_{j;\varepsilon,\zeta}) P^{+,U,ss}_{\varepsilon,\zeta}
\end{align*}
where the above projections have been constructed in the proof of Lemma \ref{caplemma}. Then setting
\begin{align*}
\mathcal{W}^+_{\varepsilon,\zeta} = \DS^+_{\varepsilon,\zeta} R^+_\zeta
\end{align*}
with $R^+_\zeta = R_\zeta$ from \eqref{diagR+} the change of basis
\begin{align*}
\mathcal{W}^+_{\varepsilon,\zeta} \xi^+ = \xi
\end{align*}
gives the desired equality
\begin{align*}
N^+_{\varepsilon,\zeta} = \SP \left\lbrace \begin{pmatrix} e_k \\0 \end{pmatrix}, \begin{pmatrix} 0 \\ e_1 \end{pmatrix}, \ldots, \begin{pmatrix} 0 \\ e_n \end{pmatrix} \right\rbrace
\end{align*}
for all $\varepsilon$ and $\zeta$. In fact we have fixed the space
\begin{align*}
S^{+,sf}_{\varepsilon,\zeta} & = \lbrace 0 \rbrace^{k-1} \times \lbrace 0 \rbrace \times \lbrace 0 \rbrace^{n-k} \times \lbrace 0 \rbrace^{k-1} \times \CC \times \lbrace 0 \rbrace^{n-k},\\
U^{+,sf}_{\varepsilon,\zeta} & = \lbrace 0 \rbrace^{k-1} \times \CC \times \lbrace 0 \rbrace^{n-k} \times  \lbrace 0 \rbrace^{k-1} \times \lbrace 0 \rbrace \times \lbrace 0 \rbrace^{n-k},\\
U^{+,ss}_{\varepsilon,\zeta} & = \lbrace 0 \rbrace^{k-1} \times \lbrace 0 \rbrace \times \lbrace 0 \rbrace^{n-k} \times \CC^{k-1} \times \lbrace 0 \rbrace \times \lbrace 0 \rbrace^{n-k},\\
S^{+,ss}_{\varepsilon,\zeta} & = \lbrace 0 \rbrace^{k-1} \times \lbrace 0 \rbrace \times \lbrace 0 \rbrace^{n-k} \times \lbrace 0 \rbrace^{k-1} \times \lbrace 0 \rbrace \times \CC^{n-k},
\end{align*}
at constant values, too. The system \eqref{LapHsCLautoin} reads
\begin{align}
\label{LapHsCLautoinslow}
\left\lbrace \begin{aligned}
\tau^\prime & = \varepsilon (1-\tau^2)h_\varepsilon(\tau)\\
\xi^{+\prime} & = A^{I,+}_{\varepsilon,\zeta}(\tau) \xi^+
\end{aligned} \right.
\end{align}
with
\begin{align*}
A^{I,+}_{\varepsilon,\zeta}(\tau) = (\mathcal{W}^+_{\varepsilon,\zeta})^{-1} A^I_{\varepsilon,\zeta}(\tau) \mathcal{W}^+_{\varepsilon,\zeta}.
\end{align*}
We still have at $\varepsilon = 0$
\begin{align*}
A^{I,+}_{0,\zeta}(\tau) = A^I_0.
\end{align*}
Hence, with $\xi^+ = (p^+,q^+)$ the fast variables are again $p_j^+, j \neq k,$ and the slow variables are $p_k^+, q_1^+, \ldots, q_n^+$. The slow equations for the system \eqref{LapHsCLautoinslow} on $\NH_{\varepsilon,\zeta}$ in $(p_k^+,q^+)$-coordinates read
\begin{align}
\label{slowinner+}
\left\lbrace \begin{aligned}
\dot{\tau} & = \varepsilon (1-\tau^2)h_\varepsilon(\tau)\\
\dot{\begin{pmatrix}
p^+_k \\ q^+
\end{pmatrix}} & = \tilde{A}^+_{\varepsilon,\zeta}(\tau) \begin{pmatrix}
p^+_k \\ q^+
\end{pmatrix}
\end{aligned} \right.
\end{align}
with a matrix-valued function $\tilde{A}^+_{\varepsilon,\zeta}(\tau) \in \CC^{(n+1)\times(n+1)}$ depending smoothly on $\varepsilon$ and $\tau$, and holomorphically on $\zeta$. At $\tau = +1$ we have
\begin{align*}
N_{\varepsilon,\zeta}(+1) = \left\lbrace (p,q) \in \CC^{2n}\left\vert\; \hat{p} \right. = 0\right\rbrace
\end{align*}
and
\begin{align*}
A^{I,+}_{\varepsilon,\zeta}(+1) = \diag(A^{I,+,S}_{\varepsilon,\zeta}, \mu^{+,U,sf}_{\varepsilon,\zeta}, A^{I,+,U}_{\varepsilon,\zeta}, (\mu^{+,U,ss}_{j;\varepsilon,\zeta})_{j<k},\mu^{+,S,sf}_{\varepsilon,\zeta},(\mu^{+,S,ss}_{j;\varepsilon,\zeta})_{j>k})
\end{align*}
for appropriate matrices
\begin{align*}
A^{I,+,S}_{\varepsilon,\zeta} \in \CC^{(k-1) \times (k-1)} \;\; \textnormal{and} \;\; A^{I,+,U}_{\varepsilon,\zeta} \in \CC^{(n-k) \times (n-k)}.
\end{align*}
Hence, we find for $\varepsilon$ and $\zeta$ that at the right end the matrix describing the linear slow $(p^+_k,q^+)$-flow in \eqref{slowinner+} is given by the diagonal matrix
\begin{align}
\label{diagslow+}
\tilde{A}^+_{\varepsilon,\zeta}(+1) = \varepsilon^{-1} \diag(\mu^{+,U,sf}_{\varepsilon,\zeta}, (\mu^{+,U,ss}_{j;\varepsilon,\zeta})_{j<k},\mu^{+,S,sf}_{\varepsilon,\zeta},(\mu^{+,S,ss}_{j;\varepsilon,\zeta})_{j>k}).
\end{align}
In $(p_k^+,q^+)$-coordinates the (constant) spaces $S^{+,sf}_{\varepsilon,\zeta}, U^{+,sf}_{\varepsilon,\zeta}, S^{+,ss}_{\varepsilon,\zeta}$ and $U^{+,ss}_{\varepsilon,\zeta}$ have unique constant counter parts give by
\begin{align*}
\tilde{S}^{+,sf}_{\varepsilon,\zeta} & = \tilde{S}^{+,sf}_0 := \lbrace 0 \rbrace \times \lbrace 0 \rbrace^{k-1} \times \CC \times \lbrace 0 \rbrace^{n-k},\\
\tilde{U}^{+,sf}_{\varepsilon,\zeta} & = \tilde{U}^{+,sf}_0 := \CC \times \lbrace 0 \rbrace^{k-1} \times \lbrace 0 \rbrace \times \lbrace 0 \rbrace^{n-k},\\
\tilde{S}^{+,ss}_{\varepsilon,\zeta} & = \tilde{S}^{+,sf}_0 := \lbrace 0 \rbrace \times \lbrace 0 \rbrace^{k-1} \times \lbrace 0 \rbrace \times \CC^{n-k},\\
\tilde{U}^{+,ss}_{\varepsilon,\zeta} & = \tilde{U}^{+,ss}_0 := \lbrace 0 \rbrace \times \CC^{k-1} \times \lbrace 0 \rbrace \times \lbrace 0 \rbrace^{n-k}.
\end{align*}
Consider the Grassmannian flow
\begin{align}
\label{slowflow+}
\left\lbrace \begin{aligned}
\dot{\tau} & = (1-\tau^2)h_\varepsilon(\tau)\\
\tilde{Y}^\prime & = \Gamma^{n-k+1} (\tilde{A}^{I,+}_{\varepsilon,\zeta}(\tau))\tilde{Y}
\end{aligned} \right. \;\; \textnormal{on} \;\; J \times \GG^{n+1}_{n-k+1}(\CC).
\end{align}
associated to \eqref{slowinner+}, and define
\begin{align*}
\mathcal{C}^{+,s} = \lbrace +1 \rbrace \times \lbrace \tilde{Y} \in \GG^{n+1}_{n-k+1}(\CC)| \; \tilde{S}^{+,sf}_0 \subset \tilde{Y} \rbrace.
\end{align*}
which is an invariant manifold of \eqref{slowflow+} for all $\varepsilon$ and $\zeta$ by \eqref{diagslow+}. For $\varepsilon = 0$ one deduces from Lemma \ref{lemlocalGrass},
\begin{align}
\label{diagslow+eps0}
\tilde{A}^+_{0,\zeta}(+1) = \diag(\tilde{\mu}^{+,U,sf}_{\zeta},0, \ldots,0,\tilde{\mu}^{+,S,sf}_{\zeta},0,\ldots,0),
\end{align}
with $\tilde{\mu}^{+,S/U,sf}_{\zeta} = \tilde{\mu}^{S/U,sf}_{\zeta}$ from \eqref{slowfastEV}, and
\begin{align*}
\RE (\tilde{\mu}^{+,U,sf}_{\zeta} - \tilde{\mu}^{+,S,sf}_{\zeta}) = 2 \RE \sqrt{a^2 + (G_0)_{kk}\zeta} > 0
\end{align*}
that $C^{+,s}$ is a normally hyperbolic manifold for \eqref{slowflow+} with the $\tau$-direction being attracting and $C^{+,s}$ being repelling in $\lbrace \tau = 1 \rbrace$. Applying the theory on normally hyperbolic manifolds \cite{F71,F79,HPS77,S91} the stable manifold $W^S_{\varepsilon,\zeta}(C^{+,s})$ possesses an invariant fibration with one dimensional leaves $W^S_{\varepsilon,\zeta}((1,\tilde{Y}))$ each based at some point $(1,\tilde{Y}) \in C^{+,s}$. If the base point of the leaf is a fixed point of \eqref{slowflow+} then the leaf is itself invariant under the flow of \eqref{slowflow+}. The intersection of $\lbrace \tau = 0 \rbrace$ with the leaf $W^S_{\varepsilon,\zeta}((1,\tilde{S}^{+,s}))$ based at the fixed point
\begin{align*}
\tilde{S}^{+,s}_0 := \tilde{S}^{+,sf}_0 \oplus \tilde{S}^{+,ss}_0
\end{align*}
varies smoothly in $\varepsilon$ and holomorphically in $\zeta$. Hence, by setting
\begin{align*}
\tilde{H}^{+,s}_{\varepsilon}(\zeta) := W^S_{\varepsilon,\zeta}((1,\tilde{S}^{+,s}_0)) \cap \lbrace \tau = 0 \rbrace
\end{align*}
we obtain a holomorphic bundle depending smoothly on $\varepsilon$, and the solution $(\tau, \tilde{Y}^{+,s})$ of \eqref{slowflow+} starting at $(0,\tilde{H}^{+,s}_{\varepsilon}(\zeta))$ satisfies
\begin{align*}
(\tau,\tilde{Y}^{+,s})(+\infty) = (1,\tilde{S}^{+,s}_0).
\end{align*}
Noting that with our choice of basis we have
\begin{align*}
S^{+,s}_{\varepsilon,\zeta} = \lbrace 0 \rbrace^{n-1} \times \tilde{S}^{+,s}_0
\end{align*}
and returning to the full system \eqref{LapHsCLautoinslow} we find the desired slow Evans bundle
\begin{align*}
H^{+,s}_\varepsilon(\zeta) := \left\lbrace \begin{pmatrix}
p\\ q
\end{pmatrix} \in \CC^{2n} \left\vert \; \begin{pmatrix}
p_k\\ q
\end{pmatrix} \in \tilde{H}^{+,s}_{\varepsilon}(\zeta), \hat{p} = N_{\varepsilon,\zeta}(0) \begin{pmatrix}
p_k\\ q
\end{pmatrix} \right.\right\rbrace.
\end{align*}
This proves part (i) of the lemma. We continue with the resolution of the slow bundle $H^{+,s}_\varepsilon(\zeta)$ at $\varepsilon =0$. Using the definition of the reduced vector field associated to \eqref{LapHsCLautoinslow} on the slow manifold $\NH_{\varepsilon,\zeta}$ one finds after a short calculation that at $\varepsilon=0$ the slow flow is given by 
\begin{align}
\label{slowflow+eps0}
\tilde{A}^+_{0,\zeta}(\tau) = \begin{pmatrix}
\ast & \ast e_k^t\\
\ast e_k & \ast E_{kk}
\end{pmatrix}
\end{align}
where each $\ast$ stands for some (different) holomorphic scalar-valued function. From \eqref{slowflow+eps0} we see that the spaces
\begin{align*}
V & := \CC \times \lbrace 0 \rbrace^{k-1} \times \CC \times \lbrace 0 \rbrace^{n-k}\\
\tilde{U}^{ss}_0 & := \lbrace 0 \rbrace \times \CC^{k-1} \times \lbrace 0 \rbrace \times \lbrace 0 \rbrace^{n-k}\\
\tilde{S}^{ss}_0 & := \lbrace 0 \rbrace \times \lbrace 0 \rbrace \times \lbrace 0 \rbrace \times \CC^{n-k}
\end{align*}
are a splitting for $\CC^{n+1}$ and define disjoint invariant sets
\begin{align*}
J \times V, J \times \tilde{U}^{ss}_0, J \times \tilde{S}^{ss}_0
\end{align*}
of the slow equations \eqref{slowinner+}. By the asymptotics
\begin{align*}
\tilde{Y}^{+,s}(+ \infty) = \tilde{S}^{+,s}_0 = \tilde{S}^{+,sf}_0 \times \tilde{S}^{+,ss}_0
\end{align*}
the holomorphic bundle $\tilde{H}^{+,s}_0(\zeta) \cap V$ must be one dimensional, and
\begin{align*}
\tilde{S}^{ss}_0 \subset \tilde{H}^{+,s}_0(\zeta).
\end{align*}
Finally, let us consider the Grassmannian flow
\begin{align}
\label{slowfastflow+}
\left\lbrace \begin{aligned}
\dot{\tau} & = (1-\tau^2)h_\varepsilon(\tau)\\
\tilde{Y}^\prime & = \Gamma^{n-k+1} (\tilde{A}^{I,+}_{\varepsilon,\zeta}(\tau))\tilde{Y}
\end{aligned} \right. \;\; \textnormal{on} \;\; J \times \GG^{n+1}_{1}(\CC).
\end{align}
associated to \eqref{slowinner+}. Since the matrix $\tilde{A}^{I,+}_{0,\zeta}(\tau)$ at $\tau = 1$ is of diagonal form given in \eqref{diagslow+eps0} and there is a positive spectral gap between the eigenvalue $\tilde{\mu}^{+,S,sf}_\zeta$ and the other $n$ eigenvalues of $\tilde{A}^{I,+}_{0,\zeta}(1)$ we find by Lemma \ref{lemlocalGrass} that $\lbrace 1 \rbrace \times \tilde{S}^{+,sf}_0$ is a hyperbolic fixed point of \eqref{slowfastflow+} for all $\zeta$. At $\lbrace 1 \rbrace \times \tilde{S}^{+,sf}_0$ the $\tau$-direction is attracting and in $\lbrace 1 \rbrace \times \GG^{n+1}_1(\CC)$ the point $\lbrace 1 \rbrace \times \tilde{S}^{+,sf}_0$ is repelling. Applying the stable manifold theorem yields a unique holomorphic bundle
\begin{align*}
\tilde{H}^{+,sf}_\varepsilon: D_{r_0} \rightarrow \GG^{n+1}_1(\CC)
\end{align*}
depending smoothly on $\varepsilon$ such that the solution $(\tau,\tilde{Z}^{+,sf})$ of \eqref{slowfastflow+} starting at $(0,\tilde{H}^{+,sf}_\varepsilon(\zeta))$ satisfies
\begin{align*}
\tilde{Z}^{+,sf}(+\infty) = \tilde{S}^{+,sf}_0.
\end{align*}
In particular, by sharp asymptotics we find
\begin{align*}
\tilde{H}^{+,sf}_0(\zeta) = \tilde{H}^{+,s}_0(\zeta) \cap V.
\end{align*}
Embedding $\tilde{H}^{+,sf}_\varepsilon$ into the full phase space like $\tilde{H}^{+,s}_\varepsilon$ before yields the slow-fast bundle $H^{+,sf}_\varepsilon$ and finishes the proof of the lemma.
\end{proof}
From the slow equations and the decay properties in Lemma \ref{leminnnerslow} we can resolve the structure of the bundles $H^{\pm,s}_\varepsilon(\zeta)$ at $\varepsilon = 0$ by determining the slow-fast bundle $H^{\pm,sf}_0(\zeta)$:
\begin{lem}
\label{lemslowEvans}
Consider the Burgers' equation
\begin{align}
\label{Burgersequ}
v_t - (v^2)_x = v_{xx}\;\; (v \in \RR),
\end{align}
and let $\phi_0:\RR \rightarrow \RR$ be a standing shock profile of \eqref{Burgersequ} with
\begin{align*}
\phi_0^\prime = 1-\phi_0^2, \;\; \phi_0(\pm \infty) = \pm 1, \;\; \phi_0(0) = 0.
\end{align*}
Linearize \eqref{Burgersequ} at $\phi_0$ and choose
\begin{align*}
\tilde{\xi}^\prime = \begin{pmatrix}
-2 \phi_0 & 1 \\
\tilde{\zeta} & 0
\end{pmatrix} \tilde{\xi} \;\; (\tilde{\xi} \in \CC^2, \tilde{\zeta} \in \CH^+)
\end{align*}
as a first-order formulation for the equation of Laplace-modes of the resulting linearized equation. Let 
\begin{align*}
\tilde{h}^{\pm}: \CH^+\rightarrow \CC^2, \tilde{h}^\pm(\tilde{\zeta}) = (\tilde{h}^\pm_p(\tilde{\zeta}),\tilde{h}_q^\pm(\tilde{\zeta}))^t,
\end{align*}
be a holomorphic choice of basis for the Evans bundles $\tilde{\EH}^\pm(\tilde{\zeta}) = \SP \lbrace \tilde{h}^\pm(\tilde{\zeta}) \rbrace$ of $\phi_0$, extended onto $\lbrace 0 \rbrace$. Define an Evans function for $\phi_0$ by
\begin{align*}
\tilde{E}_0(\tilde{\zeta}) := \det (\tilde{h}^-(\tilde{\zeta}),\tilde{h}^+(\tilde{\zeta})).
\end{align*}
Then $\tilde{E}_0(\tilde{\zeta}) = 0$ iff $\tilde{\zeta} = 0$ and $\tilde{E}_0(0)^\prime \neq 0$. Setting $\tilde{\zeta}(\zeta) = a^{-2} (G_0)_{kk}\zeta$ it holds
\begin{align*}
H^{+,sf}_0(\zeta) & = \SP \left\lbrace \begin{pmatrix}
\tilde{h}^+_p(\tilde{\zeta}(\zeta))e_k\\
|a|\tilde{h}_q^+(\tilde{\zeta}(\zeta))\underline{G}_0e_k
\end{pmatrix} \right\rbrace\\
H^{-,sf}_0(\zeta) & = \SP \left\lbrace \begin{pmatrix}
\tilde{h}^-_p(\tilde{\zeta}(\zeta))e_k\\
|a| \tilde{h}_q^-(\tilde{\zeta}(\zeta))\underline{G}_0e_k
\end{pmatrix} \right\rbrace.
\end{align*}
\end{lem}
\begin{proof}
The statement on the spectral stability of the standing shock profile $\phi_0$ for Burgers' equation can be directly verified by using the explicit expressions for $\tilde{h}^\pm$ and $\tilde{E}_0$ in B.~Barker's master thesis \cite{BB09}.

For $\varepsilon = 0$ the slow equations are
\begin{align*}
\dot{\tau} & = (1-\tau^2) |a|\\
\dot{p}_k & = 2 a \tau p_k + q_k\\
\dot{q}_j & = \zeta (G_0)_{jk} p_k.
\end{align*}
Rescaling the independent variable with $|a|$, setting $\tilde{\zeta} = a^{-2} (G_0)_{kk}\zeta$ and changing basis via
\begin{align*}
\begin{pmatrix} p_k\\q \end{pmatrix} \mapsto \begin{pmatrix} 1 & 0\\0 & |a| \underline{G}_0 \end{pmatrix}^{-1}\begin{pmatrix} p_k\\q \end{pmatrix}
\end{align*}
yields the system
\begin{align*}
\dot{\tau} & = (1-\tau^2)\\
\dot{p}_k & = -2\tau p_k + e_k^T \underline{G}_0 q\\
\dot{q}_j & = \delta_{jk} \tilde{\zeta} p_k.
\end{align*}
The subsystem for $(\tau, p_k, q_k)$-variables is exactly the system for the standing Burgers' shock profile $\phi_0$ given in the statement of the lemma. With
\begin{align*}
\begin{pmatrix}
\tilde{h}^-_p(\tilde{\zeta})\\
\tilde{h}_q^-(\tilde{\zeta})\end{pmatrix} \in \CC^2\;\;\textnormal{and} \;\; \begin{pmatrix}
\tilde{h}^+_p(\tilde{\zeta})\\
\tilde{h}_q^+(\tilde{\zeta})
\end{pmatrix} \in \CC^2
\end{align*}
spanning the Evans bundles for $\phi_0$ we find by changing back to original variables that the dynamical parts of the slow bundles are
\begin{align*}
H^{-,sf}_0(\zeta) & = \SP\left\lbrace\begin{pmatrix}
\tilde{h}^-_p(\tilde{\zeta}(\zeta))e_k\\
|a|\tilde{h}_q^-(\tilde{\zeta}(\zeta))\underline{G}_0e_k
\end{pmatrix} \right\rbrace,\\
H^{+,sf}_0(\zeta) & =\SP\left\lbrace\begin{pmatrix}
\tilde{h}^+_p(\tilde{\zeta}(\zeta))e_k\\
|a|\tilde{h}_q^+(\tilde{\zeta}(\zeta))\underline{G}_0e_k
\end{pmatrix}\right\rbrace
\end{align*}
with $\tilde{\zeta}(\zeta) = a^{-2} (G_0)_{kk}\zeta$.
\end{proof}
The previous considerations yield:
\begin{cor}
\label{corregime1}
Let $r_0> 0$ and choose $\varepsilon_0 = \varepsilon_0(r_0)$ such that the statements of Lemma \ref{leminnerfast} and Lemma \ref{leminnnerslow} hold. For $\varepsilon \in [0,\varepsilon_0]$ define holomorphic bundles
\begin{align*}
H^\pm_\varepsilon := H^{\pm,f}_\varepsilon \oplus H^{\pm,s}_\varepsilon \;\; \textnormal{on} \;\; D_{r_0}.
\end{align*}
Then if $\varepsilon >0$ the bundles $H^\pm_\varepsilon$ are rescaled versions of the Evans bundles $\EH^\pm_\varepsilon$ of $\phi_\varepsilon$ via
\begin{align*}
\EH^{\pm}_\varepsilon(\varepsilon^2 \zeta) = H^\pm_\varepsilon(\zeta)
\end{align*}
with corresponding Evans functions
\begin{align*}
\det (\EH^-_\varepsilon(\kappa), \EH^+_\varepsilon(\kappa)) =: \mathcal{E}_\varepsilon(\kappa) = E_\varepsilon(\zeta) := \det (H^-_\varepsilon(\zeta), H^+_\varepsilon(\zeta))
\end{align*}
for $\kappa = \varepsilon^2 \zeta$. At $\varepsilon = 0$ there is $c \in \RR \backslash \lbrace 0 \rbrace$ such that
\begin{align*}
E_0(\zeta) = c \tilde{E}_0(\tilde{\zeta}(\zeta)).
\end{align*}
If $\varepsilon \rightarrow 0$ then
\begin{align*}
H^\pm_\varepsilon \rightarrow H^\pm_0 \;\; \textnormal{and} \;\; E_\varepsilon \rightarrow E_0
\end{align*}
as holomorphic functions on $D_{r_0}$. For $\varepsilon$ small enough the Evans function condition on $D_{r_0}$
\begin{align*}
E_\varepsilon(\zeta) = 0 \;\; \textnormal{for} \;\; \zeta \in D_{r_0} \;\; \textnormal{iff} \;\; \zeta = 0, \;\; \textnormal{and} \;\; E_\varepsilon^\prime(0) \neq 0
\end{align*}
holds.
\end{cor}
\begin{proof}
Using
\begin{align*}
H^{-,f}_0(\zeta) & = \lbrace 0 \rbrace^k \times \CC^{n-k} \times \lbrace 0 \rbrace^n,\\
H^{+,f}_0(\zeta) &= \CC^{k-1} \times \lbrace 0 \rbrace^{n-k+1} \times \lbrace 0 \rbrace^{n},
\end{align*}
the form of $H^{\pm,s}_0(\zeta)$ in Lemma \ref{lemslowEvans} and the basis of $U^{ss}_0$ and $S^{ss}_0$ in Lemma \ref{caplemma} one finds
\begin{align*}
E_0(\zeta) & = \tilde{c}|a| \det \begin{pmatrix} \tilde{h}^-_p(\tilde{\zeta}(\zeta)) & \tilde{h}^+_p(\tilde{\zeta}(\zeta)) & 0\\
\tilde{h}^-_q(\tilde{\zeta}(\zeta)) & \tilde{h}^+_q(\tilde{\zeta}(\zeta)) & 0\\
\ast & \ast & C^{ss}\end{pmatrix}\\
& = \tilde{c}|a| \det \begin{pmatrix} \tilde{h}^-_p(\tilde{\zeta}(\zeta)) & \tilde{h}^+_p(\tilde{\zeta}(\zeta))\\
\tilde{h}^-_q(\tilde{\zeta}(\zeta)) & \tilde{h}^+_q(\tilde{\zeta}(\zeta))\end{pmatrix} \det C^{ss}\\
& = c \tilde{E}_0(\tilde{\zeta}(\zeta))
\end{align*}
where $\tilde{c} \in \lbrace -1,1 \rbrace$ and $C^{ss} \in \textnormal{GL}_{n-1}(\CC)$ consists of the $n-1$ non-zero rows of $(\underline{G}_0 r_j)_{j\neq k} \in \CC^{n \times (n-1)}$, namely all but the $k$-th row.\\
Let $\varepsilon>0$. It remains to show that $E_\varepsilon(0) = 0$. If $\zeta = 0$ then system \eqref{LapHsCLautoin} is given by
\begin{align}
\label{LapHsCLautoineps0}
\left\lbrace \begin{aligned}
\tau^\prime & = \varepsilon (1-\tau^2)h_\varepsilon(\tau)\\
\begin{pmatrix}
p\\ q
\end{pmatrix}^\prime & = \begin{pmatrix}
F_\varepsilon(\tau) & \varepsilon \ID_n \\ 0 & 0
\end{pmatrix}\begin{pmatrix}
p\\ q
\end{pmatrix}.
\end{aligned} \right.
\end{align}
The set $J \times \lbrace q=0 \rbrace$ is invariant for \eqref{LapHsCLautoineps0}. Furthermore, since $A^I_{\varepsilon,0}(\tau) \CC^{2n} \subset \lbrace q = 0 \rbrace$ we find that $U^{-,sf}_{\varepsilon,0}$ resp.~$S^{+,sf}_{\varepsilon,0}$ is a subset of $\lbrace q = 0 \rbrace$ with an unique counter part $\tilde{U}^{-,sf}_{\varepsilon,0}$ resp.~$\tilde{S}^{+,sf}_{\varepsilon,0}$ in $\CC^n$. The counter part is the invariant subspace of $F_\varepsilon(-1)$ resp.~$F_\varepsilon(+1)$ associated to the slow-fast eigenvalue $\mu^{-,U,sf}_{\varepsilon,0}$ resp.~$\mu^{+,S,sf}_{\varepsilon,0}$. On $J \times \lbrace q=0 \rbrace$ \eqref{LapHsCLautoineps0} reduces to
\begin{align}
\label{LapHsCLautoinzeta0}
\left\lbrace \begin{aligned}
\tau^\prime & = \varepsilon (1-\tau^2)h_\varepsilon(\tau)\\
p^\prime & = F_\varepsilon(\tau) p.
\end{aligned} \right.
\end{align}
Consider the Grassmannian flow
\begin{align}
\label{slowfastflowzeta0}
\left\lbrace \begin{aligned}
\tau^\prime & = \varepsilon(1-\tau^2)h_\varepsilon(\tau)\\
X^\prime & = \Gamma^{1} (F_\varepsilon(\tau))X
\end{aligned} \right. \;\; \textnormal{on} \;\; J \times \GG^{n}_{1}(\CC),
\end{align}
and note that by $F_0(\tau) = \diag(\lambda_1,\ldots,\lambda_n)$ and Lemma \ref{lemlocalGrass} the manifold
\begin{align*}
\mathcal{M}_0 := J \times \lbrace \SP \lbrace e_k \rbrace \rbrace
\end{align*}
is a one dimensional normally hyperbolic critical manifold for \eqref{slowfastflowzeta0}. It gives rise to a slow manifold $\mathcal{M}_\varepsilon$. Because
\begin{align*}
\tilde{U}^{-,sf}_{0,0} = \SP \lbrace e_k \rbrace = \tilde{S}^{+,sf}_{0,0}
\end{align*}
the fixed points $(-1,\tilde{U}^{-,sf}_{\varepsilon,0})$ and $(+1,\tilde{S}^{+,sf}_{\varepsilon,0})$ of \eqref{slowfastflowzeta0} are elements of $\mathcal{M}_\varepsilon$. For $\varepsilon > 0$ there is a point $(0,\tilde{H}^{sf}_\varepsilon) \in \mathcal{M}_\varepsilon$ such that the solution $(\tau,X)$ of \eqref{slowfastflowzeta0} starting in $(0,\tilde{H}^{sf}_\varepsilon)$ satisfies
\begin{align*}
X(-\infty) = \tilde{U}^{-,sf}_{\varepsilon,0} \;\; \textnormal{and} \;\; X(+\infty) = \tilde{S}^{+,sf}_{\varepsilon,0}.
\end{align*}
Returning to full phase space we find by Lemma \ref{leminnnerslow} (ii) and sharp asymptotics that
\begin{align*}
H^{-,sf}_{\varepsilon}(0) = \tilde{H}^{sf}_\varepsilon \times \lbrace 0 \rbrace^n = H^{+,sf}_{\varepsilon}(0),
\end{align*}
i.e.~ $H^-_\varepsilon(0) \cap H^+_\varepsilon(0) \neq \lbrace 0 \rbrace$ and
\begin{align*}
E_\varepsilon(0) = 0.
\end{align*}
\end{proof}
\subsection{Outer regime}
Set for $\kappa \in \CC$ and $\varepsilon \in [0,\varepsilon_0]$ (c.f.~\cite{FS10})
\begin{align*}
\beta^2 = |\kappa|, \varepsilon = \alpha \beta \;\; \textnormal{and} \;\; \tilde{\kappa} = \frac{\kappa}{|\kappa|} \;\; \textnormal{if} \;\; \kappa \neq 0.
\end{align*}
Then finding appropriate $r_0,r_1> 0$ and $\varepsilon_0 > 0$ for the outer regime $r_0 \varepsilon^2 \leq |\kappa|\leq r_1$ translates to finding $r_1, r_0^{-1}$ small enough:
\begin{align*}
\beta & \leq r_1^{1/2}\\
\alpha & \leq r_0^{-1/2}.
\end{align*}
We consider
\begin{align*}
A^{II}_{\alpha,\beta,\tilde{\kappa}}(\tau) := \begin{pmatrix}
F_{\alpha \beta}(\tau) & \beta \ID_n\\
\beta^3 \tilde{\kappa}^2 \ID_n + \beta \tilde{\kappa} G_{\alpha \beta}(\tau) & 0
\end{pmatrix}
\end{align*}
for $\alpha, \beta > 0$ and
\begin{align*}
\tilde{\kappa} \in \partial D_1^+ := \lbrace \kappa \in \CH^+: \; |\kappa| = 1 \rbrace.
\end{align*}
If $\beta > 0$ then
\begin{align*}
A^{II}_{\alpha,\beta,\tilde{\kappa}}(\tau) \sim A_{\alpha \beta, \beta^2 \tilde{\kappa}}(\tau).
\end{align*}
The task is to find $\alpha_0, \beta_0>0$ small enough such that for all
\begin{align*}
(\alpha,\beta,\tilde{\kappa}) \in [0,\alpha_0] \times (0,\beta_0] \times \partial D_1^+
\end{align*}
the non-autonomous system
\begin{align}
\label{outernonauto}
\xi^\prime = A^{II}_{\alpha,\beta,\tilde{\kappa}}(\chi_{\alpha \beta}) \xi \;\; (\xi \in \CC^{2n})
\end{align}
has Evans bundles with trivial intersection. First we make an observation on the spectrum of $A^{II}_{\alpha,\beta,\tilde{\kappa}}(\tau)$.
\begin{lem}
\label{lemspecouter}
There are $\alpha_0, \beta_0>0$ such that for all
\begin{align*}
(\alpha,\beta,\tilde{\kappa},\tau) \in [0,\alpha_0] \times (0,\beta_0] \times \partial D_1^+ \times J
\end{align*}
the stable subspace $S_{\alpha,\beta,\tilde{\kappa}}(\tau)$ and the unstable subspace $U_{\alpha,\beta,\tilde{\kappa}}(\tau)$ of $A^{II}_{\alpha,\beta,\tilde{\kappa}}(\tau)$ satisfy
\begin{align*}
\dim S_{\alpha,\beta,\tilde{\kappa}}(\tau) = n = \dim U_{\alpha,\beta,\tilde{\kappa}}(\tau).
\end{align*}
In particular, $i \RR \cap \sigma(A^{II}_{\alpha,\beta,\tilde{\kappa}}(\tau)) = \emptyset$.
\end{lem}
\begin{proof}
Let $(\alpha,\beta,\tilde{\kappa},\tau) \in [0,\alpha_0] \times (0,\beta_0] \times \partial D_1^+ \times J$ with $\alpha_0 \beta_0(=\varepsilon_0)$ small enough. Because $F_{\alpha \beta}(\tau),G_{\alpha \beta}(\tau)$ satisfy the assumptions of Lemma \ref{lemdisperion} there exists a $\nu > 0$ such that $\tilde{\Lambda} := M_\nu$ is a domain of constant splitting for $\lbrace A_{F_{\alpha \beta}(\tau), G_{\alpha \beta}(\tau)}(\kappa) \rbrace_\kappa$ with
\begin{align}
\label{tausplitting}
\CH^+\backslash \lbrace 0 \rbrace \subset \tilde{\Lambda}.
\end{align}
Because $\beta >0$ we have $A^{II}_{\alpha,\beta,\tilde{\kappa}}(\tau) \sim A_{\alpha \beta,\beta^2 \tilde{\kappa}}(\tau)= A_{F_{\alpha \beta}(\tau), G_{\alpha \beta}(\tau)}(\beta^2\tilde{\kappa})$ and $\beta \tilde{\kappa} \in \CH^+\backslash \lbrace 0 \rbrace$, thus the claim follows from \eqref{tausplitting}.
\end{proof}
With the above lemma in mind we can hope to apply Lemma \ref{lemma6a}. To this end we need slowly varying block-diagonalizers for $A^{II}_{\alpha,\beta,\tilde{\kappa}}(\tau)$ splitting the latter into a contracting and an expanding part of dimension $n$.
\begin{lem}
\label{lemdiagouter}
There exist $\alpha_0,\beta_0 >0$ and smooth functions
\begin{align*}
R: [0,\alpha_0] \times [0,\beta_0] \times \partial D_1^+ \times J & \rightarrow \textnormal{GL}_{2n}(\CC),\\
A^>: [0,\alpha_0] \times [0,\beta_0] \times \partial D_1^+ \times J & \rightarrow \CC^n\\
A^<: [0,\alpha_0] \times [0,\beta_0] \times \partial D_1^+ \times J & \rightarrow \CC^n,
\end{align*}
such that for all $(\alpha,\beta,\tilde{\kappa},\tau) \in [0,\alpha_0] \times [0,\beta_0] \times \partial D_1^+ \times J$ it holds
\begin{align*}
R^{-1}_{\alpha,\beta,\tilde{\kappa}}(\tau)A^{II}_{\alpha,\beta,\tilde{\kappa}}(\tau)R_{\alpha,\beta,\tilde{\kappa}}(\tau) & = \begin{pmatrix}
A^>_{\alpha,\beta,\tilde{\kappa}}(\tau) & 0\\ 0 & A^<_{\alpha,\beta,\tilde{\kappa}}(\tau)
\end{pmatrix},
\end{align*}
and if $\beta > 0$ then
\begin{align*}
\textnormal{Re}\, A^<_{\alpha,\beta,\tilde{\kappa}}(\tau) < 0 < \textnormal{Re}\, A^>_{\alpha,\beta,\tilde{\kappa}}(\tau).
\end{align*}
Furthermore, there is a constant such that
\begin{align*}
|R_{\alpha,\beta,\tilde{\kappa}}(0)^{-1}| & \leq c\\
|\partial_\tau R_{\alpha,\beta,\tilde{\kappa}}(\tau)| & \leq c(\alpha + \beta).
\end{align*}
\end{lem}
\begin{proof}
We proceed in a similar fashion as in the proof of Lemma \ref{caplemma}. Let $r>0$ such that $u_\varepsilon(\tau) \in B_r(0)$ for all $\varepsilon \in [0,\varepsilon_0]$, and, if necessary shrink $\varepsilon_0$ such that the statement of Lemma \ref{lemdiagf} holds. Then we may assume that with the notation of the latter lemma
\begin{align*}
F_{\alpha\beta}(\tau) = \begin{pmatrix}
M^<_{\alpha \beta}(\tau) & 0 & 0\\ 0 & M^0_{\alpha \beta}(\tau) & 0 \\ 0 & 0 & M^>_{\alpha \beta}(\tau)
\end{pmatrix}
\end{align*}
for
\begin{align*}
M^<_{\alpha \beta}(\tau) & = M^<(\alpha \beta,\alpha \beta u_{\alpha \beta}(\tau)) \\
M^0_{\alpha \beta}(\tau) & = M^0(\alpha \beta,\alpha \beta u_{\alpha \beta}(\tau))\\
M^>_{\alpha \beta}(\tau) & =M^>(\alpha \beta,\alpha \beta u_{\alpha \beta}(\tau)).
\end{align*}
For $\beta = 0$ we find that
\begin{align*}
A^{II}_{\alpha,0,\tilde{\kappa}}(\tau) = \begin{pmatrix}
\diag(\lambda_1,\ldots,\lambda_n) & 0\\ 0 & 0
\end{pmatrix}
\end{align*}
is independent of $(\alpha,\tilde{\kappa},\tau)$ and of diagonal form. Hence for small $\beta$ the total projections
\begin{align*}
P^{S,f}_{\alpha,\beta,\tilde{\kappa}}(\tau), P^s_{\alpha,\beta,\tilde{\kappa}}(\tau) \; \textnormal{and}\; P^{U,f}_{\alpha,\beta,\tilde{\kappa}}(\tau)
\end{align*}
onto the eigenvalues with real part (strongly) smaller, close-to and (strongly) larger than zero are well-defined and smooth in their parameters. For $\beta = 0$ they have constant values
\begin{align*}
P^{S,f}_0:= P^{S,f}_{\alpha,0,\tilde{\kappa}}(\tau) & = \begin{pmatrix}
\Pi^< & 0\\0 & 0
\end{pmatrix}\\
P^s_0:=P^s_{\alpha,0,\tilde{\kappa}}(\tau) & = \begin{pmatrix}
E_{kk} & 0\\0 & \ID_n
\end{pmatrix}\\
P^{U,f}_0:=P^{U,f}_{\alpha,0,\tilde{\kappa}}(\tau)& = \begin{pmatrix}
\Pi^> & 0\\0 & 0
\end{pmatrix}.
\end{align*}
We apply Kato's reduction process \cite{K95} to the $0$-group. A Taylor expansion shows that
\begin{align*}
A^{II,s}_{\alpha,\beta,\tilde{\kappa}}(\tau) := \frac{1}{\beta} A^{II}_{\alpha,\beta,\tilde{\kappa}}(\tau) P^s_{\alpha,\beta,\tilde{\kappa}}(\tau)
\end{align*}
defines a smooth function with
\begin{align*}
A^{II,s}_{\alpha,0,\tilde{\kappa}}(\tau) = P^s_0 \partial_\beta A^{II}_{\alpha,0,\tilde{\kappa}}(\tau) P^s_0.
\end{align*}
From Lemma \ref{lemdiagf} and $u_\varepsilon(\tau) = \varepsilon \tau e_k + \mathcal{O}(\varepsilon^2)$ one finds $\partial_\beta M^0_{0}(\tau) = 2\alpha a \tau$ and hence
\begin{align*}
A^{II,s}_{\alpha,0,\tilde{\kappa}}(\tau) = \begin{pmatrix}
2 \alpha a \tau E_{kk} & E_{kk} \\ \tilde{\kappa} G_0 E_{kk} & 0
\end{pmatrix}.
\end{align*}
For small $\alpha>0$ and all $(\tilde{\kappa},\tau)$ the matrix $A^{II,s}_{\alpha,0,\tilde{\kappa}}(\tau)$ has two non-zero simple eigenvalues
\begin{align*}
\tilde{\mu}^{U/S,sf}_{\alpha,\tilde{\kappa}}(\tau) = \alpha a \tau \pm \sqrt{(\alpha a \tau)^2 + \tilde{\kappa}(G_0)_{kk}}
\end{align*}
Hence, for small $\alpha, \beta > 0$ there is a simple unstable and a simple stable slow-fast eigenvalues $\mu^{U/S,sf}_{\alpha,\beta,\tilde{\kappa}}(\tau)$ of $A^{II}_{\alpha,\beta,\tilde{\kappa}}(\tau)$ with
\begin{align*}
\mu^{U/S,sf}_{\alpha,\beta,\tilde{\kappa}}(\tau) = \beta\left(\tilde{\mu}^{U/S,sf}_{\alpha,\tilde{\kappa}}(\tau) + \OO(\beta)\right)
\end{align*}
and smooth one dimensional eigenprojections $P^{U/S,sf}_{\alpha,\beta,\tilde{\kappa}}(\tau)$. For $\alpha,\beta >0$ small enough we have for all $(\tilde{\kappa},\tau)$
\begin{align*}
\textnormal{Re}\, \mu^{U/S,sf}_{\alpha,\beta,\tilde{\kappa}}(\tau) =\beta (\pm \textnormal{Re}\, \sqrt{\tilde{\kappa} (G_0)_{kk}} + \mathcal{O}(\alpha +\beta) )\gtrless 0.
\end{align*}
Let $P^{ss}_{\alpha,\beta,\tilde{\kappa}}(\tau)$ be the total projection onto the group of eigenvalues of $A^{II,s}_{\alpha,\beta,\tilde{\kappa}}(\tau)$ associated to the $2(n-1)$-fold semi-simple zero eigenvalue of $A^{II,s}_{\alpha,0,\tilde{\kappa}}(\tau)$. Set $\Delta^{sf}_{\alpha,\tilde{\kappa}} = \tilde{\mu}^{S,sf}_{\alpha,\tilde{\kappa}}-\tilde{\mu}^{U,sf}_{\alpha,\tilde{\kappa}}$ and remember
\begin{align*}
\tilde{R} = \begin{pmatrix}
(r_j)_{j<k},0,(r_j)_{j>k}
\end{pmatrix} \in \RR^{n \times n}.
\end{align*}
The matrices
\begin{align*}
R_{\alpha,\tilde{\kappa}}(\tau) & = \begin{pmatrix}
\ID_n & E_{kk}\\ -\frac{\tilde{\mu}^{S,sf}}{(G_0)_{kk}} G_0 E_{kk} &  G_0\left(\tilde{R} -\frac{\tilde{\mu}^{U,sf}}{(G_0)_{kk}} E_{kk}\right)
\end{pmatrix}\\
R_{\alpha,\tilde{\kappa}}(\tau)^{-1} & = \begin{pmatrix}
\ID_n - \frac{\tilde{\mu}^{S,sf}}{\Delta^{sf}}E_{kk} & - \frac{1}{\Delta^{sf}} E_{kk}\\ \frac{\tilde{\mu}^{S,sf}}{\Delta^{sf}}E_{kk} & \tilde{R}^t+\frac{1}{\Delta^{sf}} E_{kk}
\end{pmatrix}
\end{align*}
diagonalize $A^{II,s}_{\alpha,0,\tilde{\kappa}}(\tau)$ to
\begin{align*}
R_{\alpha,\tilde{\kappa}}(\tau)^{-1}A^{II,s}_{\alpha,0,\tilde{\kappa}}(\tau) R_{\alpha,\tilde{\kappa}}(\tau) = \begin{pmatrix}
\tilde{\mu}^{U,sf}_{\alpha,\tilde{\kappa}}(\tau)E_{kk} & 0 \\ 0 & \tilde{\mu}^{S,sf}_{\alpha,\tilde{\kappa}}(\tau) E_{kk}
\end{pmatrix}.
\end{align*}
With the help of the above diagonalizers we find
\begin{align*}
P^{U,sf}_{\alpha,\tilde{\kappa}}(\tau) := &P^{U,sf}_{\alpha,0,\tilde{\kappa}}(\tau)\\
 = & R_{\alpha,\tilde{\kappa}}(\tau) \begin{pmatrix}
E_{kk} & 0 \\0 & 0
\end{pmatrix} R_{\alpha,\tilde{\kappa}}(\tau)^{-1},\\
P^{S,sf}_{\alpha,\tilde{\kappa}}(\tau) := &P^{S,sf}_{\alpha,0,\tilde{\kappa}}(\tau)\\
 = & R_{\alpha,\tilde{\kappa}}(\tau) \begin{pmatrix}
0 & 0 \\0 & E_{kk}
\end{pmatrix} R_{\alpha,\tilde{\kappa}}(\tau)^{-1},\\
P^{ss}_0 := &P^{ss}_{\alpha,0,\tilde{\kappa}}(\tau)\\
= & R_{\alpha,\tilde{\kappa}}(\tau) \begin{pmatrix}
\ID_n - E_{kk} & 0 \\0 & \ID_n - E_{kk}
\end{pmatrix} R_{\alpha,\tilde{\kappa}}(\tau)^{-1}\\
= & \begin{pmatrix}
\ID_n - E_{kk} & 0\\ 0 & p^{ss}
\end{pmatrix},
\end{align*}
where $p^{ss} = \ID_n - \underline{G}_0 E_{kk}$. Carrying on with the reduction process define the smooth function
\begin{align*}
A^{II,ss}_{\alpha,\beta,\tilde{\kappa}}(\tau) = \frac{1}{\beta}P^{ss}_{\alpha,\beta,\tilde{\kappa}}(\tau)A^{II,s}_{\alpha,\beta,\tilde{\kappa}}(\tau)P^{ss}_{\alpha,\beta,\tilde{\kappa}}(\tau).
\end{align*}
With
\begin{align*}
\tilde{F}^{-1}_0 =  \begin{pmatrix}
(M^<_0(0))^{-1} & 0 & 0\\ 0&0&0 \\ 0& 0 & (M^<_0(0))^{-1}
\end{pmatrix}
\end{align*}
it holds
\begin{align*}
A^{II,ss}_{\alpha,0,\tilde{\kappa}}(\tau) = \frac{1}{2}\begin{pmatrix}
0 & 0\\ 0 & - \tilde{\kappa} p^{ss} G_0 \tilde{F}^{-1}_0 p^{ss}
\end{pmatrix}
\end{align*}
where we have used
\begin{align*}
\partial_\beta P^s_{\alpha,0,\tilde{\kappa}}(\tau) = \begin{pmatrix}
0 & -\tilde{F}^{-1}_0\\ - \tilde{\kappa} G_0 & 0
\end{pmatrix}.
\end{align*}
Like in the previous treatment of the inner regime we have
\begin{align*}
(\tilde{R}^t+E_{kk}) p^{ss} G_0 \tilde{F}^{-1}_0 p^{ss} G_0(\tilde{R} + &(G_0)_{kk}^{-1} E_{kk})\\
& = \diag(\mu_1^{-1},\ldots,\mu_{k-1}^{-1},0,\mu_{k+1}^{-1},\ldots,\mu_n^{-1}),
\end{align*}
and the non-zero and simple eigenvalues of $A^{II,ss}_{\alpha,0,\tilde{\kappa}}(\tau)$ are
\begin{align*}
\lbrace - \tilde{\kappa} (2\mu_j)^{-1} \rbrace_{j \neq k}.
\end{align*}
Hence there are $n-1$ simple eigenvalues of $A^{II}$ expanding as
\begin{align*}
\mu^{U,ss}_{j;\alpha,\beta,\tilde{\kappa}}(\tau) & = \beta^2(- \tilde{\kappa} (2\mu_j)^{-1} + \mathcal{O}(\beta)) \;\; (j < k)\\
\mu^{S,ss}_{j;\alpha,\beta,\tilde{\kappa}}(\tau) & = \beta^2(- \tilde{\kappa} (2\mu_j)^{-1} + \mathcal{O}(\beta)) \;\; (j > k)
\end{align*}
with smooth eigenprojections $P^{S/U;ss}_{j;\alpha,\beta,\tilde{\kappa}}(\tau)$. For small $\beta >0$ the inequalities
\begin{align*}
\textnormal{Re}\, \mu^{S,ss}_{j;\alpha,\beta,1}(\tau)< 0 < \textnormal{Re}\, \mu^{U,ss}_{j;\alpha,\beta,1}(\tau)
\end{align*}
hold. Furthermore the super-slow projections are given for $\beta = 0$ by
\begin{align*}
P^{S/U;ss}_{j}:=P^{S/U;ss}_{j;\alpha,0,\tilde{\kappa}}(\tau) = R_{\alpha,\tilde{\kappa}}(\tau) \begin{pmatrix}
0 & 0\\ 0 & E_{jj}
\end{pmatrix} R_{\alpha,\tilde{\kappa}}(\tau)^{-1} = \begin{pmatrix}
0 & 0\\ 0 & G_0 \tilde{R} E_{jj} \tilde{R}^t
\end{pmatrix}.
\end{align*}
Set
\begin{align*}
Q = &\partial_\beta P^{S,f}P^{S,f} + \partial_\beta P^{U,f}P^{U,f} + \partial_\beta P^{S,sf}P^{S,sf} + \partial_\beta P^{U,sf}P^{U,sf}\\
& + \sum_{j>k} \partial_\beta P^{S,ss}_jP^{S,ss}_j + \sum_{j<k} \partial_\beta P^{U,ss}_jP^{U,ss}_j,
\end{align*}
and for $(\alpha,\tilde{\kappa},\tau)$ let $\DS_{\alpha,\beta,\tilde{\kappa}}(\tau)$ solve
\begin{align*}
\partial_\beta \DS_{\alpha,\beta,\tilde{\kappa}}(\tau) = Q_{\alpha,\beta,\tilde{\kappa}}(\tau) \DS_{\alpha,\beta,\tilde{\kappa}}(\tau), \;\; \DS_{\alpha,0,\tilde{\kappa}}(\tau) = \ID_{2n}.
\end{align*}
Define
\begin{align*}
R_{\alpha,\beta,\tilde{\kappa}}(\tau) := \DS_{\alpha,\beta,\tilde{\kappa}}(\tau) R_{\alpha,\tilde{\kappa}}(\tau)
\end{align*}
and note that for all $(\tilde{\kappa},\tau)$
\begin{align*}
R_{0,0,\tilde{\kappa}}(\tau) = R_{0,\tilde{\kappa}}(\tau) = R_{\tilde{\kappa}} := R_{0,\tilde{\kappa}}(0),
\end{align*}
which implies the claimed properties of $R$ by smoothness of $R$, the compactness of the parameter domain and the identity
\begin{align*}
R_{\alpha,\beta,\tilde{\kappa}}(\tau) = R_{\tilde{\kappa}} + \int_0^1 \nabla_{\alpha,\beta} R_{y \alpha,y \beta,\tilde{\kappa}}(\tau)dy \begin{pmatrix}
\alpha\\ \beta
\end{pmatrix}.
\end{align*}

After taking care of rearrangements by multiplication of $R$ with an elementary matrix, and naming the resulting product $R$ again, we find that
\begin{align*}
R^{-1} A^{II} R = \begin{pmatrix}
A^> & 0\\
0 & A^<
\end{pmatrix}
\end{align*}
where
\begin{align*}
A^> &= \begin{pmatrix}
A^{>,f} & 0\\
0 & \diag(\mu^{U,ss}_{1},\ldots,\mu^{U,ss}_{k-1},\mu^{U,sf})
\end{pmatrix},\\
A^< & = \begin{pmatrix}
A^{<,f} & 0\\
0 & \diag(\mu^{S,sf}, \mu^{S,ss}_{k+1},\ldots,\mu^{S,ss}_{n})
\end{pmatrix}
\end{align*}
with $A^{>,f} = A^{>,f}_{\alpha,\beta,\tilde{\kappa}}(\tau) \in \CC^{(n-k) \times (n-k)}$ and $A^{<,f} = A^{<,f}_{\alpha,\beta,\tilde{\kappa}}(\tau) \in \CC^{(k-1) \times (k-1)}$ satisfying
\begin{align*}
A^{>,f}_{\alpha,\beta,\tilde{\kappa}}(\tau) & = \diag(\lambda_{k+1},\ldots,\lambda_{n}) + \mathcal{O}(\beta),\\
A^{<,f}_{\alpha,\beta,\tilde{\kappa}}(\tau) & = \diag(\lambda_1,\ldots,\lambda_{k-1}) + \mathcal{O}(\beta).
\end{align*}
Furthermore, for $\alpha \geq 0, \beta >0$ small enough and all $(\tilde{\kappa},\tau)$ the inequalities
\begin{align*}
\textnormal{Re}\, \mu^{S,sf}, \textnormal{Re}\, \mu^{S,ss}_j < 0 < \textnormal{Re}\, \mu^{U,sf}, \textnormal{Re}\, \mu^{U,ss}_j 
\end{align*}
hold because these inequalities are true for $\tilde{\kappa} = 1$ and $\beta > 0$ small, and because, by continuity, real parts of eigenvalues of $A^{II}_{\alpha,\beta,\tilde{\kappa}}(\tau)$ may not change signs as $\tilde{\kappa}$ is varied in $\partial D_1^+$. The latter argument applies by Lemma \ref{lemspecouter} which states $i\RR \cap \sigma (A^{II}_{\alpha,\beta,\tilde{\kappa}}(\tau)) = \emptyset$ if $\beta > 0$. Hence we find for $\beta >0$ small enough the desired inequalities
\begin{align*}
\textnormal{Re}\, A^<_{\alpha,\beta,\tilde{\kappa}}(\tau) < 0 < \textnormal{Re}\, A^>_{\alpha,\beta,\tilde{\kappa}}(\tau).
\end{align*}
\end{proof}
Combing Lemma \ref{lemdiagouter} and Lemma \ref{lemma6a} the parametrization $\varepsilon = \alpha \beta$ and $\beta^2 = |\kappa|$ in the outer regime yields:
\begin{cor}
\label{lemregime2}
There exist $r_0,r_1, \varepsilon_0 >0$ such that for $\varepsilon \in (0,\varepsilon_0]$ and $\zeta \in \CH^+$ with $\varepsilon^2 r_0 \leq \varepsilon^2 |\zeta| \leq r_1$ the rescaled Evans bundles
\begin{align*}
H_\varepsilon^\pm(\zeta) = \EH^\pm_\varepsilon(\varepsilon^2 \zeta)
\end{align*}
are transversal:
\begin{align*}
H_\varepsilon^-(\zeta) \cap H^+_\varepsilon(\zeta) = \lbrace 0 \rbrace.
\end{align*}
\end{cor}
\subsection{Outmost regime}
The main purpose of this subsection is to show that Lemma \ref{lemma6a} is applicable to the system
\begin{align}
\label{A1}
\xi^\prime &= A_{\varepsilon,\kappa}(\chi_\varepsilon) \xi
\end{align}
in a large portion of the outmost regime: $r \leq |\kappa|, \textnormal{Re}\, \kappa \geq 0$ for $r>0$ large enough. The connection to the entire outmost regime $r_1 \leq |\kappa|$ with boundary defined by arbitrary $r_1>0$ is then an easy consequence of uniform hyperbolicity of $A_{\varepsilon,\kappa}$ on the compact set $r_1 \leq |\kappa| \leq r$ \cite{C65,H81,S02}.\\
In order to apply Lemma \ref{lemma6a} we need to find appropriate block-diagonalizers $R_{\varepsilon,\kappa}(\tau)$ for all $\varepsilon \in [0,\varepsilon_0], \tau \in J$ and all $\kappa$ with $|\kappa|$ large enough. First note that by changing basis via
\begin{align*}
\begin{pmatrix}
\ID_n & 0\\
0 & \kappa
\end{pmatrix}
\end{align*}
for $\kappa \neq 0$ and then via
\begin{align*}
\begin{pmatrix}
\ID_n & \ID_n\\ \ID_n & - \ID_n
\end{pmatrix}
\end{align*}
one arrives at the representation
\begin{align*}
A_{\varepsilon,\kappa}(\tau) \sim A^{III}_{\varepsilon,\kappa}(\tau) = \kappa \begin{pmatrix}
\ID_n & 0\\
0 & -\ID_n
\end{pmatrix} + \begin{pmatrix}
B^>_\varepsilon(\tau) & B^>_\varepsilon(\tau)\\
B^<_\varepsilon(\tau) & B^<_\varepsilon(\tau)
\end{pmatrix} 
\end{align*}
where
\begin{align*}
B^\gtrless_\varepsilon(\tau) := \frac{1}{2} \left(F_\varepsilon(\tau) \pm G_\varepsilon(\tau)\right).
\end{align*}
By assumption A4 if $\varepsilon_0$ is small enough there are constants $c^\gtrless > 0$ such that for all $\varepsilon \in [0,\varepsilon_0], \tau \in J$ and $\eta \in \CC^n$ it holds
\begin{align}
\label{definiteBpm}
\textnormal{Re}\, \skp{\eta,B^>_\varepsilon(\tau) \eta}{} & \geq c^>|\eta|^2\\
\textnormal{Re}\, \skp{\eta,B^<_\varepsilon(\tau) \eta}{} & \leq - c^<|\eta|^2
\end{align}
For $r > 0$ set
\begin{align*}
\CH^+_r := \lbrace \kappa \in \CH^+|\; |\kappa|\geq r \rbrace.
\end{align*}
\begin{lem}
\label{lemdiagoutmost}
There exist $r,\varepsilon_0>0$ and smooth functions
\begin{align*}
R&:[0,\varepsilon_0] \times \CH^+_r \times J \rightarrow \textnormal{GL}_{2n}(\CC),\\
A^< &:[0,\varepsilon_0] \times \CH^+_r \times J \rightarrow \CC^{n\times n},\\
A^> &:[0,\varepsilon_0] \times \CH^+_r \times J \rightarrow \CC^{n\times n},
\end{align*}
such that for all $(\varepsilon,\kappa,\tau) \in[0,\varepsilon_0] \times \CH^+_r \times J$ it holds
\begin{align*}
R_{\varepsilon,\kappa}(\tau)^{-1} A^{III}_{\varepsilon,\kappa}(\tau) R_{\varepsilon,\kappa}(\tau) = \begin{pmatrix}
A^<_{\varepsilon,\kappa}(\tau) & 0 \\ 0 & A^>_{\varepsilon,\kappa}(\tau)
\end{pmatrix},
\end{align*}
and
\begin{align*}
\RE A^<_{\varepsilon,\kappa}(\tau) < 0 < \RE A^>_{\varepsilon,\kappa}(\tau).
\end{align*}
Furthermore, there exists a $c>0$ such that
\begin{align*}
|R_{\varepsilon,\kappa}(0)^{-1}| & \leq c,\\
|\partial_\tau R_{\varepsilon,\kappa}(\tau)| & \leq c \varepsilon.
\end{align*}
\end{lem}
\begin{proof}
We will carry out the matrix bifurcation of
\begin{align*}
B_{\varepsilon,\lambda}(\tau) = \begin{pmatrix}
\ID_n & 0\\
0 & -\ID_n
\end{pmatrix} + \lambda \begin{pmatrix}
B^>_\varepsilon(\tau) & B^>_\varepsilon(\tau)\\
B^<_\varepsilon(\tau) & B^<_\varepsilon(\tau)
\end{pmatrix}
\end{align*}
uniformly in $\varepsilon$ and $\tau$ for $|\lambda| \leq r^{-1}$. Note that for $\kappa \neq 0$ the identity 
\begin{align*}
A^{III}_{\varepsilon,\kappa}(\tau) = \kappa B_{\varepsilon,\kappa^{-1}}(\tau)
\end{align*}
holds. For $\nu >0$ small and $r>0$ large enough the total projections
\begin{align*}
P^{U/S}_{\varepsilon,\lambda}(\tau) := \frac{1}{2 \pi i} \int\limits_{\partial B_\nu(\pm 1)} (z \ID_{2n} - B_{\varepsilon,\lambda}(\tau))^{-1} dz
\end{align*}
are well-defined, smooth in $\varepsilon$ and $\tau$ and holomorphic in $\lambda$. Then
\begin{align*}
P_{\varepsilon,0}(\tau) & = P^U_0 := \begin{pmatrix}
\ID_n & 0 \\ 0&0
\end{pmatrix}\\
P_{\varepsilon,0}(\tau) & = P^S_0:= \begin{pmatrix}
0&0 \\ 0 &\ID_n
\end{pmatrix}
\end{align*}
and a short calculation shows
\begin{align*}
\frac{\partial}{\partial \lambda} P^{U/S}_{\varepsilon,0}(\tau) = \pm \frac{1}{2} \begin{pmatrix}
0 & B^>_\varepsilon(\tau) \\ B^<_\varepsilon(\tau) & 0
\end{pmatrix}.
\end{align*}
Following again Kato's reduction process \cite{K95} let
\begin{align*}
Q_{\varepsilon,\lambda}(\tau) := (\partial_\lambda P^U_{\varepsilon,\lambda}(\tau)) P^U_{\varepsilon,\lambda}(\tau) + (\partial_\lambda  P^S_{\varepsilon,\lambda}(\tau)) P^S_{\varepsilon,\lambda}(\tau)
\end{align*}
and let $\DS_{\varepsilon,\lambda}(\tau)$ be the solution of
\begin{align*}
\partial_\lambda \DS_{\varepsilon,\lambda}(\tau) = Q_{\varepsilon,\lambda}(\tau) \DS_{\varepsilon,\lambda}(\tau), \; \DS_{\varepsilon,0}(\tau) = \ID_{2n}.
\end{align*}
Then
\begin{align*}
P^{U/S}_0 = \DS_{\varepsilon,\lambda}(\tau)^{-1} P^{U/S}_{\varepsilon,\lambda}(\tau) \DS_{\varepsilon,\lambda}(\tau)
\end{align*}
and $B_{\varepsilon,\lambda}(\tau)$ is given with respect to the transformation $\DS_{\varepsilon,\lambda}(\tau)$ by
\begin{align*}
M_{\varepsilon,\lambda}(\tau) & = \DS_{\varepsilon,\lambda}(\tau)^{-1} B_{\varepsilon,\lambda}(\tau) \DS_{\varepsilon,\lambda}(\tau)\\
& = P^U_0 M_{\varepsilon,\lambda}(\tau) P^U_0 + P^S_0 M_{\varepsilon,\lambda}(\tau) P^S_0
\end{align*}
which is of the form
\begin{align*}
\begin{pmatrix}
M^>_{\varepsilon,\lambda}(\tau) & 0\\
0 & M^<_{\varepsilon,\lambda}(\tau)
\end{pmatrix}.
\end{align*}
We expand $M_{\varepsilon,\lambda}(\tau)$:
\begin{align*}
M_{\varepsilon,\lambda}(\tau) & = M_{\varepsilon,0}(\tau) + \lambda \frac{\partial}{\partial \lambda} M_{\varepsilon,0}(\tau) + \mathcal{O}(\lambda^2)\\
& = \begin{pmatrix} \ID_n & 0\\ 0 & - \ID_n \end{pmatrix} + \lambda \begin{pmatrix} B^>_{\varepsilon}(\tau) & 0 \\ 0 & B^<_\varepsilon(\tau)\end{pmatrix} + \mathcal{O}(\lambda^2),
\end{align*}
hence
\begin{align*}
M^\gtrless_{\varepsilon,\lambda}(\tau) = \pm \ID_n + \lambda B^\gtrless_\varepsilon(\tau) + \mathcal{O}(\lambda^2).
\end{align*}
By choosing $r>0$ large enough we find by setting
\begin{align*}
R_{\varepsilon,\kappa}(\tau) := \DS_{\varepsilon,\kappa^{-1}}(\tau)
\end{align*}
and
\begin{align*}
A^\gtrless_{\varepsilon,\kappa}(\tau) := \kappa M^\gtrless_{\varepsilon,\kappa^{-1}}(\tau) = \pm \kappa \ID_n + B^\gtrless_\varepsilon(\tau) + \mathcal{O}(|\kappa|^{-1})
\end{align*}
that for all $\varepsilon \in [0,\varepsilon_0], \tau \in J$ and $|\kappa| \geq r$
\begin{align*}
R_{\varepsilon,\kappa}(\tau)^{-1} A^{III}_{\varepsilon,\kappa}(\tau) R_{\varepsilon,\kappa}(\tau) = \begin{pmatrix}
A^>_{\varepsilon,\kappa}(\tau) & 0 \\ 0 & A^<_{\varepsilon,\kappa}(\tau)
\end{pmatrix}.
\end{align*}
For $\eta \in \CC^n$ we have for $r$ large enough and $\textnormal{Re}\, \kappa \geq 0$ that
\begin{align*}
\textnormal{Re}\, \skp{\eta, A^>_{\varepsilon,\kappa}(\tau)\eta}{} & = \textnormal{Re}\, \kappa |\eta|^2 + \textnormal{Re}\, \skp{\eta, B^>_{\varepsilon}(\tau) \eta}{} + \textnormal{Re}\,\skp{\eta,\mathcal{O}(\kappa^{-1}) \eta}{}\\
& \geq (\textnormal{Re}\, \kappa + c^> + \mathcal{O}(|\kappa|^{-1})) |\eta|^2\\
& \geq \frac{1}{2} c^> |\eta|^2\\
\textnormal{Re}\, \skp{\eta, A^<_{\varepsilon,\kappa}(\tau)\eta}{} & = -\textnormal{Re}\, \kappa |\eta|^2 + \textnormal{Re}\, \skp{\eta, B^<_{\varepsilon}(\tau) \eta}{} + \textnormal{Re}\,\skp{\eta,\mathcal{O}(\kappa^{-1}) \eta}{}\\
& \leq (-\textnormal{Re}\, \kappa - c^< + \mathcal{O}(|\kappa|^{-1})) |\eta|^2\\
& \leq -\frac{1}{2} c^< |\eta|^2.
\end{align*}
Since $\DS_{\varepsilon,\kappa^{-1}}(0)^{-1}$ is a continuous function and $\varepsilon$ and $\kappa^{-1}$ range over a compact set there is a constant $c_1>0$ such that
\begin{align*}
|R_{\varepsilon,\kappa}(0)^{-1}| \leq c_1.
\end{align*}
Furthermore, since $\DS_{\varepsilon,\tau}(\kappa^{-1})$ is continuously differentiable in $\tau$ and $\varepsilon$ (and its arguments range, again, over a compact set) there is a constant $c_2 > 0$ such that
\begin{align*}
\left\vert \frac{\partial^2}{\partial \varepsilon \partial\tau}R_{\varepsilon,\kappa}(\tau)\right\vert \leq c_2.
\end{align*}
By $R_{0,\kappa}(\tau) = R_{0,\kappa}(0)$ (for $\varepsilon =0$ there is no $\tau$ dependence of $A^{III}_{\varepsilon,\kappa}(\tau)$) we deduce from
\begin{align*}
R_{\varepsilon,\kappa}(\tau) = R_{0,\kappa}(0) + \varepsilon \int\limits_0^1 \frac{\partial}{\partial \varepsilon} R_{s \varepsilon,\kappa}(\tau) ds
\end{align*}
that
\begin{align*}
\left\vert \frac{\partial}{\partial \tau}R_{\varepsilon,\kappa}(\tau)\right\vert \leq \varepsilon c_2.
\end{align*}
\end{proof}
Combining the above estimate gives
\begin{align*}
\left\vert R_{\varepsilon,\kappa}(0)^{-1} \frac{\partial}{\partial \tau}R_{\varepsilon,\kappa}(\tau) \right\vert \leq c_1c_2 \varepsilon,
\end{align*}
and Lemma \ref{lemma6a} is applicable if $\varepsilon$ is small enough. Finally note that Lemma \ref{lemma6a} makes a statement on the evolution of vector spaces which is independent of the choice of basis, hence its conclusions hold for the equations \eqref{chi1} and \eqref{A1}, i.e.~for
\begin{align*}
\chi_\varepsilon^\prime & = \varepsilon (1- \chi_\varepsilon^2)h_\varepsilon(\chi_\varepsilon)\\
\xi^\prime &= A_{\varepsilon,\kappa}(\chi_\varepsilon) \xi.
\end{align*}
This proves:
\begin{cor}
\label{lemregime3}
For any $r_1 > 0$ there exists $\varepsilon_0=\varepsilon_0(r_1)$ such that for all $\varepsilon \in (0,\varepsilon_0]$ and all $\zeta \in \CH^+$ with $r_1 \leq \varepsilon^2 \zeta$ the rescaled Evans bundles are transversal:
\begin{align*}
H^-_\varepsilon(\zeta) \cap H^+_\varepsilon(\zeta) = \lbrace 0 \rbrace.
\end{align*}
\end{cor}
Theorem \ref{maintheorem} follows from Corollaries \ref{corregime1}, \ref{lemregime2} and \ref{lemregime3} by choosing $r_0 >0$ in Corollary \ref{corregime1} large enough and $r_1>0$ in Corollary \ref{lemregime3} small enough such that for small $\varepsilon_0 >0$ the statement of Corollary \ref{lemregime2} holds for the constants $r_0, r_1, \varepsilon_0$. The three spectral regimes connect to cover all of $\CH^+$.
\begin{rem}
We close with the side remark that for $\varepsilon >0$ small enough there exists a $d_\varepsilon> 0$ such that the Evans bundles $\EH^\pm_\varepsilon$ and an Evans function $\EE_\varepsilon$ are holomorphically defined on
\begin{align*}
\CH^+-d_\varepsilon = \lbrace \kappa \in \CC|\;\; \RE \kappa \geq - d_\varepsilon \rbrace
\end{align*}
with the Evans function condition holding on the above superset of $\CH^+$, too.
\end{rem}
\section{Acknowledgments}
We want to thank H. Freistühler for proposing this research topic and making helpful suggestions on how to approach it.\\
This research has been supported by Deutsche Forschungsgemeinschaft under Grant No. FR822/10-1.


\begin{thebibliography}{99}
\footnotesize

\bibitem{AGJ90}
J.~Alexander, R.~Gardner, and C.~Jones.
\newblock A topological invariant arising in the stability analysis of
  travelling waves.
\newblock {\em J. Reine Angew. Math.}, 410:167--212, 1990.

\bibitem{B24}
J. Bärlin.
\newblock Spectral stability of shock profiles for hyperbolically regularized systems of conservation laws.
\newblock {\em Arch. Rational Mech. Anal.}, 248, 125, 2024.

\bibitem{BB09}
B.~H. Barker.
\newblock Evans function computations.
\newblock Master's thesis, Brigham Young University, 2009.

\bibitem{B74}
G.~Boillat.
\newblock Sur l'existence et la recherche d'\'{e}quations de conservation suppl\'{e}mentaires pour les syst\`emes hyperboliques.
\newblock {\em C. R. Acad. Sci. Paris S\'{e}r. A}, 278:909--912, 1974.

\bibitem{C65}
W.~A. Coppel.
\newblock {\em Stability and asymptotic behavior of differential equations}.
\newblock D. C. Heath and Company, Boston, Mass., 1965.

\bibitem{DM90}
B.~F. Doolin and C.~F. Martin.
\newblock {\em Introduction to differential geometry for engineers}, volume 136
  of {\em Monographs and Textbooks in Pure and Applied Mathematics}.
\newblock Marcel Dekker, Inc., New York, 1990.

\bibitem{E74}
J.~W. Evans.
\newblock Nerve axon equations. {IV}. {T}he stable and the unstable impulse.
\newblock {\em Indiana Univ. Math. J.}, 24(12):1169--1190, 1974/75.

\bibitem{E10}
L.~C. Evans.
\newblock {\em Partial differential equations}.
\newblock American Mathematical Society, 2 edition, 2010.

\bibitem{F71}
N.~Fenichel.
\newblock Persistence and smoothness of invariant manifolds for flows.
\newblock {\em Indiana Univ. Math. J.}, 21:193--226, 1971.

\bibitem{F79}
N.~Fenichel.
\newblock Geometric singular perturbation theory for ordinary differential
  equations.
\newblock {\em J. Differential Equations}, 31(1):53--98, 1979.

\bibitem{F00}
C.~Fries.
\newblock Stability of viscous shock waves associated with non-convex modes.
\newblock {\em Arch. Ration. Mech. Anal.}, 152(2):141--186, 2000.

\bibitem{FS21}
H.~Freist{\"u}hler and M.~Sroczinski.
\newblock A class of uniformly dissipative symmetric hyperbolic-hyperbolic
  systems.
\newblock {\em J. Differential Equations}, 288:40--61, 2021.

\bibitem{FS02}
H.~Freist{\"u}hler and P.~Szmolyan.
\newblock Spectral stability of small shock waves.
\newblock {\em Arch. Ration. Mech. Ana.}, 164: 287--309, 2002.

\bibitem{FS10}
H.~Freist{\"u}hler and P.~Szmolyan.
\newblock Spectral stability of small-amplitude viscous shock waves in several
  space dimensions.
\newblock {\em Arch. Ration. Mech. Ana.}, 195(2):353--373, 2010.

\bibitem{FT14}
H.~Freist{\"u}hler and B.~Temple.
\newblock Causal dissipation and shock profiles in the relativistic fluid dynamics of pure radiation.
\newblock {\em Proc. R. Soc. Lond. A Math. Phys. Eng. Sci.}, 470, 20140055: 17 pp., 2014.

\bibitem{FL71}
K.~O. Friedrichs and P.~D. Lax.
\newblock Systems of conservation equations with a convex extension.
\newblock {\em Proc. Nat. Acad. Sci. U.S.A.}, 68:1686--1688, 1971.

\bibitem{FGP94}
J.~Ferrer, M.~I. Garc\'{\i}a, and F.~Puerta.
\newblock Differentiable families of subspaces.
\newblock {\em Linear Algebra Appl.}, 199:229--252, 1994.

\bibitem{FSW14}
H.~Freist{\"u}hler, P.~Szmolyan, and J.~W\"{a}chtler.
\newblock Spectral stability of shock waves associated with not genuinely
  nonlinear modes.
\newblock {\em J. Differential Equations}, 257(1):185--206, 2014.

\bibitem{G59}
I.~M. Gelfand.
\newblock Some problems in the theory of quasi-linear equations.
\newblock {\em Uspehi Mat. Nauk}, 14(2 (86)):87--158, 1959.

\bibitem{G61}
S.~K. Godunov.
\newblock An interesting class of quasi-linear systems.
\newblock {\em Dokl. Akad. Nauk SSSR}, 139:521--523, 1961.

\bibitem{G86}
J.~Goodman.
\newblock Nonlinear asymptotic stability of viscous shock profiles for
  conservation laws.
\newblock {\em Archive for Rational Mechanics and Analysis}, 95(4):325--344,
  1986.
  
\bibitem{GJ91}
R.~Gardner and C.~K. R.~T. Jones.
\newblock Stability of travelling wave solutions of diffusive predator-prey
  systems.
\newblock {\em Trans. Amer. Math. Soc.}, 327(2):465--524, 1991.

\bibitem{GJ912}
R.~A. Gardner and C.~K. R.~T. Jones.
\newblock {\em Stability of one-dimensional waves in weak and singular limits}.
\newblock SIAM, Philadelphia, PA, 1991.

\bibitem{GH78}
P.~Griffiths and J.~Harris.
\newblock {\em Principles of algebraic geometry}.
\newblock Pure and Applied Mathematics. Wiley-Interscience [John Wiley \&
  Sons], New York, 1978.
  
\bibitem{GZ98}
R.~A. Gardner and K.~Zumbrun.
\newblock The gap lemma and geometric criteria for instability of viscous shock
  profiles.
\newblock {\em Comm. Pure Appl. Math.}, 51(7):797--855, 1998.

\bibitem{H81}
D.~Henry.
\newblock {\em Geometric theory of semilinear parabolic equations}, volume 840
  of {\em Lecture Notes in Mathematics}.
\newblock Springer-Verlag, Berlin-New York, 1981.

\bibitem{H03}
J.~ Humpherys.
\newblock Stability of {J}in-{X}in relaxation shocks.
\newblock {\em Quart. Appl. Math.}, 61(2):251--263, 2003.

\bibitem{HPS77}
M.~W. Hirsch, C.~C. Pugh, and M.~Shub.
\newblock Invariant manifolds.
\newblock pages ii+149, 1977.

\bibitem{J84}
C.~K. R.~T. Jones.
\newblock Stability of the travelling wave solution of the
  {F}itz{H}ugh-{N}agumo system.
\newblock {\em Trans. Amer. Math. Soc.}, 286(2):431--469, 1984.

\bibitem{J95}
C.~K.~R.~T. Jones.
\newblock Geometric singular perturbation theory.
\newblock 1609:44--118, 1995.

\bibitem{JX95}
S.~Jin and Z.~P. Xin.
\newblock The relaxation schemes for systems of conservation laws in arbitrary
  space dimensions.
\newblock {\em Comm. Pure Appl. Math.}, 48(3):235--276, 1995.

\bibitem{JGK93}
C.~K. R.~T. Jones, R.~Gardner, and T.~Kapitula.
\newblock Stability of travelling waves for nonconvex scalar viscous
  conservation laws.
\newblock {\em Comm. Pure Appl. Math.}, 46(4):505--526, 1993.

\bibitem{K95}
T.~Kato.
\newblock {\em Perturbation theory for linear operators}.
\newblock Classics in Mathematics. Springer-Verlag, Berlin, 1995.
\newblock Reprint of the 1980 edition.

\bibitem{KS98}
T.~Kapitula and B.~Sandstede.
\newblock Stability of bright solitary-wave solutions to perturbed nonlinear
  {S}chr\"{o}dinger equations.
\newblock {\em Phys. D}, 124(1-3):58--103, 1998.

\bibitem{KY04}
S.~Kawashima and W.-A. Yong.
\newblock Dissipative structure and entropy for hyperbolic systems of balance
  laws.
\newblock {\em Arch. Ration. Mech. Anal.}, 174(3):345--364, 2004.

\bibitem{KMN86}
S.~Kawashima, A.~Matsumura, and Kenji Nishihara, K.
\newblock Asymptotic behavior of solutions for the equations of a viscous
  heat-conductive gas.
\newblock {\em Proc. Japan Acad. Ser. A Math. Sci.}, 62, 1986.

\bibitem{L57}
P.~D. Lax.
\newblock Hyperbolic systems of conservation laws. {II}.
\newblock {\em Comm. Pure Appl. Math.}, 10:537--566, 1957.

\bibitem{L73}
P.~D. Lax.
\newblock {\em Hyperbolic systems of conservation laws and the mathematical
  theory of shock waves}.
\newblock Conference Board of the Mathematical Sciences Regional Conference
  Series in Applied Mathematics, No. 11. Society for Industrial and Applied
  Mathematics, Philadelphia, Pa., 1973.
  
\bibitem{L13}
J.~M. Lee.
\newblock {\em Introduction to smooth manifolds}, volume 218 of {\em Graduate
  Texts in Mathematics}.
\newblock Springer, New York, second edition, 2013.

\bibitem{L87}
T.-P. Liu.
\newblock Hyperbolic conservation laws with relaxation.
\newblock {\em Comm. Math. Phys.}, 108(1):153--175, 1987.

\bibitem{L97}
T.-P. Liu.
\newblock Pointwise convergence to shock waves for viscous conservation laws.
\newblock {\em Comm. Pure Appl. Math.}, 50(11):1113--1182, 1997.

\bibitem{L03}
H.~Liu.
\newblock Asymptotic stability of relaxation shock profiles for hyperbolic
  conservation laws.
\newblock {\em J. Differential Equations}, 192(2):285--307, 2003.

\bibitem{LMRS16}
C.~Lattanzio, C.~Mascia, R.~G. Plaza, and C.~Simeoni.
\newblock Analytical and numerical investigation of traveling waves for the
  {A}llen-{C}ahn model with relaxation.
\newblock {\em Math. Models Methods Appl. Sci.}, 26(5):931--985, 2016.

\bibitem{MP85}
A.~Majda and R.~L. Pego.
\newblock Stable viscosity matrices for systems of conservation laws.
\newblock {\em Journal of Differential Equations}, 56(2):229 -- 262, 1985.

\bibitem{MN85}
A.~Matsumura and K.~Nishihara.
\newblock On the stability of travelling wave solutions of a one-dimensional
  model system for compressible viscous gas.
\newblock {\em Japan J. Appl. Math.}, 2(1):17--25, 1985.

\bibitem{MZ02}
C.~Mascia and K.~Zumbrun.
\newblock Pointwise Green's function bounds and stability of relaxation shocks.
\newblock {\em Indiana Univ. Math. J.}, 51: 773--904, 2002.

\bibitem{MZ04}
C.~Mascia and K.~Zumbrun.
\newblock Stability of large-amplitude viscous shock profiles of
  hyperbolic-parabolic systems.
\newblock {\em Arch. Ration. Mech. Anal.}, 172(1):93--131, 2004.

\bibitem{MZ05}
C.~Mascia and K.~Zumbrun.
\newblock Stability of large-amplitude shock profiles of general relaxation
  systems.
\newblock {\em SIAM J. Math. Anal.}, 37(3):889--913, 2005.

\bibitem{MZ09}
C.~Mascia and K.~Zumbrun.
\newblock Spectral stability of weak relaxation shock profiles.
\newblock {\em Comm. Partial Differential Equations}, 34(1-3):119--136, 2009.

\bibitem{P83}
A.~Pazy.
\newblock {\em Semigroups of linear operators and applications to partial
  differential equations}, volume~44 of {\em Applied Mathematical Sciences}.
\newblock Springer-Verlag, New York, 1983.

\bibitem{PW92}
R.~L. Pego and M.~I. Weinstein.
\newblock Eigenvalues, and instabilities of solitary waves.
\newblock {\em Philos. Trans. Roy. Soc. London Ser. A}, 340(1656):47--94, 1992.

\bibitem{PZ04}
R.~Plaza and K.~Zumbrun.
\newblock An Evans function approach to spectral stability of small-amplitude shock profiles.
\newblock {\em Discrete Contin. Dyn. Syst.}, 10(4): 885--924, 2004.

\bibitem{RS81}
T.~Ruggeri and A.~Strumia.
\newblock Main field and convex covariant density for quasilinear hyperbolic
  systems. {R}elativistic fluid dynamics.
\newblock {\em Ann. Inst. H. Poincar\'{e} Sect. A (N.S.)}, 34(1):65--84, 1981.

\bibitem{S02}
B.~Sandstede.
\newblock Stability of travelling waves.
\newblock In {\em Handbook of dynamical systems, {V}ol. 2}, pages 983--1055. North-Holland, Amsterdam, 2002.

\bibitem{S76}
D.~H. Sattinger.
\newblock On the stability of waves of nonlinear parabolic systems.
\newblock {\em Advances in Math.}, 22(3):312--355, 1976.

\bibitem{S94}
J.~Smoller.
\newblock {\em Shock waves and reaction-diffusion equations}, volume 258 of
  {\em Grundlehren der mathematischen Wissenschaften [Fundamental Principles of
  Mathematical Sciences]}.
\newblock Springer-Verlag, New York, second edition, 1994.

\bibitem{S20}
M.~Sroczinski.
\newblock Asymptotic stability in a second-order symmetric hyperbolic system
  modeling the relativistic dynamics of viscous heat-conductive fluids with
  diffusion.
\newblock {\em J. Differential Equations}, 268(2):825--851, 2020.

\bibitem{S91}
P.~Szmolyan.
\newblock Transversal heteroclinic and homoclinic orbits in singular
  perturbation problems.
\newblock {\em J. Differential Equations}, 92, 1991.

\bibitem{SK85}
Y.~Shizuta and S.~Kawashima.
\newblock Systems of equations of hyperbolic-parabolic type with applications
  to the discrete {B}oltzmann equation.
\newblock {\em Hokkaido Math. J.}, 14(2):249--275, 1985.

\bibitem{SX93}
A.~Szepessy and Z.~Xin.
\newblock Nonlinear stability of viscous shock waves.
\newblock {\em Arch. Rational Mech. Anal.}, (122):53--103, 1993.

\bibitem{U09}
Y.~Ueda.
\newblock Stability of travelling wave solutions to a semilinear hyperbolic
  system with relaxation.
\newblock {\em Math. Methods Appl. Sci.}, 32(4):419--434, 2009.

\bibitem{W99}
G.~B. Whitham.
\newblock {\em Linear and nonlinear waves}.
\newblock Pure and Applied Mathematics (New York). John Wiley \& Sons, Inc.,
  New York, 1999.
\newblock Reprint of the 1974 original, A Wiley-Interscience Publication.

\bibitem{Y04}
W.-A. Yong.
\newblock Entropy and global existence for hyperbolic balance laws.
\newblock {\em Arch. Ration. Mech. Anal.}, 172(2):247--266, 2004.

\bibitem{Z99}
Y.~Zeng.
\newblock Gas dynamics in thermal nonequilibrium and general hyperbolic systems
  with relaxation.
\newblock {\em Arch. Ration. Mech. Anal.}, 150(3):225--279, 1999.

\bibitem{ZH98}
K.~Zumbrun and P.~Howard.
\newblock Pointwise semigroup methods and stability of viscous shock waves.
\newblock {\em Indiana Univ. Math. J.}, 47(3):741--871, 1998.

%
%

\end{thebibliography}
\end{document}